%% file: arxiv-main.tex
\documentclass[10pt]{amsart}

\input{arxiv-headers}

\begin{document}
\title{
        Solving polynomial inequalities over spaces of convex sets and applications
    }
  
\author{Saugata Basu}
\address{Department of Mathematics,
Purdue University, West Lafayette, IN 47906, U.S.A.}
\email{sbasu@math.purdue.edu}

\author{Hamidreza Amini Khorasgani}
\address{Department of Computer Science,
Purdue University, West Lafayette, IN 47906, U.S.A.}
\email{aminikhhr@gmail.com}

\author{Hemanta K. Maji}
\address{Department of Computer Science,
Purdue University, West Lafayette, IN 47906, U.S.A.}
\email{hmaji@purdue.edu}

\author{Hai H. Nguyen}
\address{Center for Quantum Technologies, 
 National University of Singapore, Singapore}
\email{hai.h.nguyen@nus.edu.sg}

\input{arxiv-abstract}
\maketitle

\tableofcontents

\input{arxiv-intro}

\input{arxiv-system-algebra}

\input{arxiv-laminar}

\input{arxiv-example}

\input{arxiv-ai}

\phantomsection
\bibliographystyle{amsplain}
\bibliography{arxiv-local}

\end{document}

%% file: arxiv-headers.tex
\usepackage[utf8]{inputenc}
\usepackage{amsfonts, color,epsf}
\usepackage{mathbbol}
\usepackage{amsthm} 
\usepackage{amssymb} 
\usepackage{xspace}
\usepackage[margin=1in]{geometry} 
\usepackage{boxedminipage} 
\usepackage{enumerate}
\usepackage{comment} 
\usepackage{todonotes}
\usepackage{stmaryrd} 
\usepackage{tikz}
    \usetikzlibrary[backgrounds, fit, intersections, calc, decorations.markings]

\usepackage{pgfplots}
    \pgfplotsset{compat=1.6}

\usepackage[pdftex,bookmarks=true,pdfstartview=FitH,colorlinks,linkcolor=blue,filecolor=blue,citecolor=blue,urlcolor=blue,pagebackref=true]{hyperref}
\usepackage[lambda,advantage,operators,sets,adversary,landau,probability,notions,logic,ff,mm,primitives,events,complexity,asymptotics,keys]{cryptocode}

\newlength{\oldparindent}
\theoremstyle{plain}
    \newtheorem{theorem}{Theorem}
    \newtheorem{corollary}{Corollary}
    
    \newtheorem{lemma}{Lemma}[section]
    \newtheorem{proposition}[lemma]{Proposition}

\theoremstyle{definition}
    
    \newtheorem{definition}[lemma]{Definition}
    \newtheorem{example}[lemma]{Example}
    \newtheorem{notation}[lemma]{Notation}
    \newtheorem{remark}[lemma]{Remark}

\providecommand{\namedref}[2]{\hyperref[#2]{#1~\ref*{#2}}\xspace}
\providecommand{\theoremref}[1]{\namedref{Theorem}{thm:#1}}

\providecommand{\lemmaref}[1]{\namedref{Lemma}{lem:#1}}
\providecommand{\figureref}[1]{\namedref{Figure}{fig:#1}}
\providecommand{\sectionref}[1]{\namedref{Section}{sec:#1}}

\providecommand{\caratheodory}[0]{Carath\'{e}odory\xspace}

\providecommand{\cC}[0]{\ensuremath{\mathcal C}\xspace}
\providecommand{\cG}[0]{\ensuremath{\mathcal G}\xspace}
\providecommand{\cI}[0]{\ensuremath{\mathcal I}\xspace}

\providecommand{\bP}[0]{\ensuremath{\mathbf P}\xspace}
\providecommand{\bQ}[0]{\ensuremath{\mathbf Q}\xspace}
\providecommand{\bX}[0]{\ensuremath{\mathbf X}\xspace}
\providecommand{\bY}[0]{\ensuremath{\mathbf Y}\xspace}
\providecommand{\bZ}[0]{\ensuremath{\mathbf Z}\xspace}

\providecommand{\ie}[0]{\text{i.e.}\xspace}

\let\leq\leqslant
\let\geq\geqslant

\providecommand{\defeq}[0]{\ensuremath{\;\coloneqq\;}\xspace}

\providecommand{\imm}[0]{\ensuremath{\mathrm{imm}}\xspace}

\providecommand{\p}[1]{\ensuremath{^{\left(#1\right)}}\xspace}

\providecommand{\cplus}[0]{\ensuremath{\hat{\oplus}}\xspace}
\providecommand{\Cplus}[0]{\ensuremath{\mathop{\widehat{\bigoplus}}}\xspace}

\providecommand{\cmult}[0]{\ensuremath{{\,\mathring{\star}\,}}\xspace}
\providecommand{\Cmult}[0]{\ensuremath{\mathop{\mathring{\bigstar}}}\xspace}

\providecommand{\mink}[0]{\ensuremath{+}\xspace}

\providecommand{\scal}[0]{\ensuremath{\cdot}\xspace}

\providecommand{\CL}[0]{\ensuremath{\mathrm{CL}}\xspace}
\providecommand{\conv}[1]{\ensuremath{\mathrm{conv}{\left({#1}\right)}\xspace}}
\providecommand{\convo}[1]{\ensuremath{\mathrm{conv}^o{\left({#1}\right)}\xspace}}
\providecommand{\relint}[1]{\ensuremath{\mathrm{relint}{\left({#1}\right)}\xspace}}
\providecommand{\convset}[0]{\ensuremath{\cC}\xspace}
\providecommand{\sol}{\ensuremath{\mathrm{sol}}\xspace}
\providecommand{\ssol}{\ensuremath{\mathrm{ss}}\xspace}
\providecommand{\evalmap}[0]{\ensuremath{\mathrm{eval}}\xspace}
\providecommand{\itr}[0]{\ensuremath{\mathrm{itr}}\xspace}
\providecommand{\card}[0]{\ensuremath{\mathrm{card}}\xspace}

\providecommand{\mono}[0]{\ensuremath{\mathrm{mono}}\xspace}
\providecommand{\elem}[0]{\ensuremath{\mathrm{supp}}\xspace}

\providecommand{\substitute}[2]{\ensuremath{\left\llbracket {#1} \gets {#2} \right\rrbracket}\xspace}

\let\phi\varphi

\newcommand{\fr}{\mathfrak{r}}
\newcommand{\fm}{\mathfrak{m}}
\newcommand{\fn}{\mathfrak{n}}
\newcommand{\fp}{\mathfrak{p}}
\newcommand{\fq}{\mathfrak{q}}

\usepackage{lineno}

%% file: arxiv-abstract.tex
\begin{abstract}
We develop a symbolic elimination theory for finite systems of recursive
containment inequalities whose unknowns are convex subsets of a
finite-dimensional real vector space. The right-hand sides are formal
expressions generated from variables and parameters by 
convex linear combinations,
finite union, and a positive geometric join encoding strict
convex combinations. We prove that every parameter assignment has a unique
smallest convex-set-valued solution and give a finite
Gaussian-elimination-type procedure that eliminates the unknowns while
preserving this solution and produces parameter-only expressions for its
coordinate sets.

More generally, let \(\mathcal B\) be a family of subsets 
containing \(\emptyset\) and 
closed under finite unions, nonnegative dilation, Minkowski
sums, positive geometric joins, and convex hulls. If all parameter sets lie in
\(\mathcal B\), then every coordinate set of the smallest solution lies in
\(\mathcal B\); when these operations are effective, so is the resulting
description. In particular, if the parameters are finite unions of
hemihedra---where a hemihedron is a bounded convex semi-linear set,
equivalently a convex finite union of relative interiors of polytopes---then
each coordinate set is a hemihedron and admits a quantifier-free semi-linear
description.

We apply this theory to lamination hulls. For
\[
V=U\oplus\bigoplus_{i=1}^{k}W_i,\qquad
\dim W_i=1,\qquad
\Lambda=\bigcup_{i=1}^{k}(U+W_i),
\]
we prove that the lamination hull \(G_{\Lambda}^{(\infty)}(S)\) of every
finite \(S\subset V\) is semi-algebraic and effectively computable by a
quantifier-free formula over the reals. The proof reduces finitely many grid
fibers to a convex-set inequality system and reconstructs the full hull using
gridded witness trees. This remains nontrivial because the defining iterates
need not stabilize and the hull need not be semi-linear.
\end{abstract}

%% file: arxiv-intro.tex
\section{Introduction}

Classical elimination methods, from Gaussian elimination to
Fourier--Motzkin elimination, provide finite procedures for eliminating
scalar variables from systems of linear equations and inequalities
\cite{Ziegler}; more general semi-algebraic elimination is available
through quantifier elimination over the real field \cite{BPRbook2}. The
first contribution of this paper is a different kind of elimination
theory: the unknowns are not real numbers, but convex subsets of a
finite-dimensional real vector space, and the constraints are recursive
containment inequalities between such convex sets. These systems arise
naturally when one tries to describe hulls generated by iterated convex
combinations. However, the usual algebra generated by Minkowski sum,
and union is not stable under the substitutions
needed for elimination, since one must also keep track of the
relative-interior information created by convex combinations. We
therefore introduce a formal calculus of set operations, including a
positive geometric join operation, and prove a
Gaussian-elimination-type theorem for convex-set-valued inequalities:
the procedure eliminates the unknown convex sets while preserving the
smallest convex solution. This yields effective semi-linear
descriptions whenever the parameter sets are finite unions of relative
interiors of polytopes (we call such sets \emph{hemihedra} in the current paper). 

\medskip
The study of algebras generated by convex polytopes has a long history starting from the foundational work of McMullen on polytope algebras \cite{McMullen89}. More recently, driven partially by the connection with tropical combinatorics \cite{Joswigbook} there has been renewed interest in the study of semirings generated by convex polytopes, including studying systems of equations over this semiring \cite{valnere2022basics, Manjunath2024}. 
In the above semiring, a scalar $a$ may be identified with the ray $[a,\infty)$, 
and addition and multiplication become, respectively, the convex hull of the union and the Minkowski sum. 
This one-dimensional picture extends to polyhedra and more generally to convex sets with a common recession cone~\cite{valnere2022basics}.

\medskip
Our work is distinguished from these works in several 
distinct ways. 
We develop a new extended algebra replacing the polyhedral semiring.
Besides dilation, Minkowski sum, and union, we introduce the positive geometric join $A\cmult B$, which encodes strict convex combinations and thus retains boundary-incidence data that are suppressed in the polyhedral semiring
mentioned before.
The basic geometric objects of this calculus, hemihedra -- convex finite unions of relative interiors of polytopes -- do not appear in the semirings
studied before. On the other hand, the boundary data is very important 
to keep track of in many applications, including the application to 
existence of secure protocols in cryptography \cite{FOCS:BKMN22}.
Secondly, to the best of our knowledge, the Gaussian-elimination theory for systems of inequalities (even over polyhedral semirings), and characterizing their minimal solutions, over spaces of convex sets, is new (see Remark~\ref{rem:prior} for a discussion of related works). 
We consider this as one of our main contributions. We believe that this technique will have applications beyond the one (proving semi-algebraicity of lamination hulls) that we give in this paper.  

\medskip
A key application of the elimination theory that we develop in the current paper is in the study of a generalization of the convex hull operator  
called \emph{lamination hulls}  (see Definition~\ref{def:G} below).
These operators form a
constructive geometric counterpart of rank-one convex hulls (which we do not define in this paper). In the
calculus of variations and nonlinear elasticity, rank-one convexity and
laminates are central in the analysis of fine phase mixtures,
microstructure, and relaxation phenomena \cite{Morrey52,Ball-James1987,Dacorogna1989}. 

\medskip
Given
a finite set \(S\) and a set of admissible directions \(\Lambda\), the
lamination hull is obtained by repeatedly adjoining strict convex
combinations of points whose difference lies in \(\Lambda\); in the
rank-one setting this gives a finite-order-laminate inner approximation
to the rank-one convex hull. 
Lamination hulls arise in the geometric approach to
secure two-party computation, where they encode round-complexity
phenomena \cite{FOCS:BKMN22,BKMN23}. 

\medskip
These connections motivate the second
main result of the present paper: for
\[
V=U\oplus\bigoplus_{i=1}^k W_i,\qquad
\dim W_i=1,\qquad
\Lambda=\bigcup_{i=1}^k (U+W_i),
\]
we prove that the lamination hull of every finite set is
semi-algebraic and effectively describable, even though the iterates
defining the hull need not stabilize, and the resulting hull need not
be semi-linear.

\medskip
We note  that
Matou\v{s}ek and Plech\'a\v{c} \cite{Matousek98}
studied 
a related object called the functional separately convex hull.  
Matou\v{s}ek \cite{Matousek2001}, 
and later Fran\v{e}k--Matou\v{s}ek \cite{Matousek-Franek2009}
developed examples and algorithms for 
such hulls,
and recent work of
Meroni--Rai\c{t}\u{a} \cite{Meroni2025} 
proves semi-algebraicity for the special case of rank-one convex
hulls of finite sets of \(2\times 2\) triangular matrices. 

\medskip
Our work is distinguished from the prior works of Matou\v{s}ek \cite{Matousek2001}, Fran\v{e}k--Matou\v{s}ek \cite{Matousek-Franek2009}
and
Meroni--Rai\c{t}\u{a} \cite{Meroni2025} in one very important respect. 
These works mostly study the \emph{functional convex hull},
which is defined in terms of $\Lambda$-convex functions (see \cite{Matousek2001} for definitions). The functional 
convex hull always contains the lamination hull, but is a simpler object (for instance, it is always closed). On the other hand, 
lamination hulls studied in this paper can have much more intricate boundary behavior. Their boundary behavior is key in applications to cryptography \cite{FOCS:BKMN22} and cannot be handled with the techniques of the papers mentioned above.

\medskip
We now fix for the rest of the paper a real finite-dimensional vector space $V$ and define below a 
formal algebra of certain operations on subsets of $V$.

\subsection{A formal algebra of set operations}
\label{subsec:formal-algebra}
The formal algebra will involve four set operations. 
The first three are the standard {\em scalar multiplication}, {\em Minkowski sum}, and {\em union} operators. 
The fourth one, namely, {\em positive geometric join}, defined in \eqref{eqn:pgj-def}, 
is not standard.
We introduce it as it facilitates reasoning about the properties of the systems of inequalities under the transformations we consider.

For any two subsets $A,B \subset V$ we define:
\begin{align}
    \text{Scaling:} && \rho\scal A &\defeq \left\{ \rho\cdot x \;\colon\; x\in A\right\},\text{ for $\rho \in \mathbb{R}$}, \\
    \text{Minkowski sum:} && A \mink B &\defeq \left\{a+b \;\colon\; a\in A, b\in B\right\}, \\
    \text{Union:} && A \cplus B &\defeq \left\{ x \;\colon\; x\in A\text{ or }x\in B\right\},\\
    \text{Positive Geometric Join:} && A\cmult B &\defeq \left\{\lambda\cdot a + (1-\lambda)\cdot b \;\colon\; 0<\lambda < 1,a\in A, b\in B\right\}.\label{eqn:pgj-def}
\end{align}
Here, the $\cdot$, $+$, and $\cplus$ operations denote the standard scaling, Minkowski sum, and union operations. 
As mentioned earlier,
the $\cmult$ is a more specialized operation;
$A\cmult B$ is the union
of the relative interiors of all line segments joining some point $a\in A$ to a point $b\in B$. 
\begin{remark}[Geometric Join]
    Note that the standard geometric join of two non-empty subsets $A, B \subset V$,
        $$A\star B \defeq \left\{\lambda\cdot a + (1-\lambda)\cdot b \;:\; 0\leq \lambda \leq 1, a\in A, b\in B\right\}$$
    can be expressed as $A\star B = A \;\cplus\; A\cmult B \;\cplus\; B$.
\end{remark}

\begin{remark}[Convex hull]
    Notice also that the convex hull, $\conv{A}$,  for any nonempty subset 
    $A \subset V$ can be expressed as
    \[
    \conv A  = \underbrace{A \cmult \cdots \cmult A}_{r+1},
    \]
    where $r =\dim\operatorname{aff}(A)$.
\end{remark}

\subsection{System of Inequalities}
\label{subsec:formal-systems}
We now introduce polynomials and formal systems of polynomial inequalities. We will reuse the set-theoretic operator symbols 
``$+$'',``$\cplus$'', ``$\cmult$'' defined in the last section, 
in the definition of formal polynomials (see below). This should not create confusion since when we define evaluation maps later (see Definition~\ref{def:eval-map}), these symbols will play their 
expected roles. 

We fix two disjoint finite sets $\Omega_v$ and $\Omega_p$. We will call elements of $\Omega_v$ \emph{variables}, and those of $\Omega_p$ \emph{parameters}, and denote 
$\Omega =(\Omega_v,\Omega_p)$.

We denote 
\begin{equation}\label{eqn:cl-omega}
    \CL(\Omega) \defeq \{ \left(
             \lambda_\omega\right)_{\omega\in\Omega_v \cup \Omega_p}, \lambda_\omega\geq0 \text{ and }
            \sum_{\omega\in\Omega_v \cup \Omega_p}\lambda_\omega=1 
        \; \}.
\end{equation}
We will often write elements of $\CL(\Omega)$ as formal 
convex linear combinations
$
\sum_{\omega\in \Omega_v \cup \Omega_p} \lambda_\omega\cdot \omega,
$
of elements of $\Omega_v \cup \Omega_p$,
where 
$
\lambda_\omega\geq0\text{ and }
            \sum_{\omega\in\Omega_v \cup \Omega_p}\lambda_\omega=1 
$.
The reason for preferring the latter notation will be clear once we define the evaluation map below (Subsection~\ref{subsec:eval-map}).

Let $\mathbb{M}[\Omega]$ denote the quotient of the free commutative monoid 
(with the binary monoid operation denoted by $\cmult$) generated by $\CL(\Omega)$ subject to the relations $E = E \cmult E, E \in \CL(\Omega)$.

We call a non-identity element of 
$\mathbb{M}[\Omega]$ a \emph{monomial over $\Omega$}.
In other words, a monomial $M \in \mathbb{M}[\Omega]$ is represented (uniquely up to permutations of the factors) by an expression of the form
\[
E_1\cmult E_2\cmult \dotsi \cmult E_k,
\]
where $k\in\{1,2,\dotsc\}$ and $E_1,E_2,\dotsc, E_k$ distinct elements of $\CL(\Omega)$.

We denote 
\[
\elem(M)\defeq \left\{E_1, E_2, \dotsc, E_k\right\} \text{ (\emph{support of $M$}),}
\]
and 
\[
\deg(M)\defeq k \text{ (\emph{degree of $M$}).} 
\]

We denote by $\mathbb{P}[\Omega]$ the free commutative monoid (with the monoid operation $\cplus$) generated by the set of monomials over $\Omega$,
subject to the relations $M=M\cplus M, M \text{ monomial over $\Omega$}$, and call
elements of $\mathbb{P}[\Omega]$
\emph{polynomials over $\Omega$}.
Each $\varphi \in \mathbb P[\Omega]$ admits (up to permutation of the summands)
a unique expression
\[
M_1 \cplus M_2 \cplus \dotsi \cplus M_k, k \in \mathbb{N},
\]
where $M_1,\ldots, M_k$ are distinct monomials over $\Omega$.
 
Note that we allow $k=0$, and denote the unique polynomial
with $k=0$ (i.e., the identity element of the monoid) by $\mathbb{0}$. 

For 
$
\varphi:= M_1 \cplus M_2 \cplus \dotsi M_k \in \mathbb{P}[\Omega]
$
we will denote by 
\[
\mono(\varphi)\defeq \{M_1, M_2, \dotsc, M_k\}
\]
(the set of all monomials of $\varphi$).
Clearly, 
\[
\mono(\mathbb{0}) = \emptyset.
\]

The following identity holds for any $\varphi \in \mathbb{P}[\Omega]$. 
    $$ \varphi = \Cplus_{M\in\mono(\varphi)} ~ \Cmult_{E\in\elem(M)} ~ E.$$

\begin{remark}
We will define an evaluation map below in which the operation $\cmult$ is interpreted as a positive geometric join. We note here that
the relation $E=E \cmult E$ in the definition $\mathbb{M}[\Omega]$ is a 
purely \emph{formal} idempotence relation. 
While the relation $E=E \cmult E$ might not hold for arbitrary sets $E$,
the equality is obtained if we take convex hulls of both sides.
This will be sufficient for the containment inequalities defined later
(Definition~\ref{def:system-of-inequalities}), 
because all left-hand sides are convex sets. 
Finally, we note that this issue does not arise for the second monoid $\mathbb{P}[\Omega]$, as the monoid operation $\cplus$ is interpreted as the union of sets, which is naturally idempotent as a set operation.
\end{remark}
We are now in a position to define systems of inequalities involving polynomials defined above.

\begin{definition}
\label{def:system-of-inequalities}
A {\em system of inequalities over $\Omega$} 
is a tuple
\begin{equation}
\label{eqn:def:system-of-inequalities}
\left(X_i \geq \varphi_i \right)_{1 \leq i \leq n},
\end{equation}
where each $\varphi_i \in \mathbb{P}[\Omega]$,
and $\Omega_v = \{X_1,\ldots,X_n\}$.
\end{definition}

\begin{remark}
    We will assume without loss of generality that the left-hand side of the inequalities in the system \eqref{eqn:def:system-of-inequalities} are in bijection with $\Omega_v$. Multiple inequalities $(X_i \geq \varphi_{ij})_{1 \leq j \leq m}$ can be replaced by one inequality $X_i \geq \Cplus_j \phi_{ij}$.
    If a variable $X_i$ does not occur on the left-hand side of any inequality, then we add the inequality $X_i \geq \mathbb{0}$ to the system. These modifications do not change the set of solutions of the system defined below.
\end{remark}

\subsection{Evaluation map}
\label{subsec:eval-map}
We will now assign subsets of 
$V$
to the elements of $\Omega_v,\Omega_p$ (the sets of parameters and variables respectively), and 
for each such assignment we will
define an evaluation map 
which will assign to each $\varphi \in \mathbb{P}[\Omega]$ a subset of $V$.

\begin{notation}
For any real vector space $U$, we denote by  $\convset(U)$ the set of all convex subsets of $U$. 
Given any $A\subset U$, its {\em convex hull}, denoted by $\conv A\in\convset(U)$, is the smallest convex set containing it. 
Finally, for any subsets $A,B \subset U$,
we will sometimes write $A\geq B$ instead of $A \supset B$.
\end{notation}

\begin{definition}
\label{def:assignment}
We call a pair $(\bX,\bP) \in \convset(V)^{\Omega_v} \times  (2^V)^{\Omega_p}$ (resp. $\bP \in (2^V)^{\Omega_p}$, $\bX \in \convset(V)^{\Omega_v}$)
an \emph{assignment} of $\Omega$ (resp. $\Omega_p$, $\Omega_v$).
\end{definition}

We now associate to each $\varphi \in \mathbb{P}[\Omega]$ and
assignment $(\bX,\bP)$ of $\Omega$ a subset of $V$ as follows.

\begin{definition}[Evaluation map under the assignment $(\bX,\bP)$]
\label{def:eval-map}

We define for 
\[
E = {\sum_{X \in \Omega_v} \lambda_X \cdot X} + {\sum_{P \in \Omega_p} \lambda_P \cdot P} \in \CL(\Omega),
\]
\[
\evalmap(E;\bX,\bP) \defeq \left(\sum_{X \in \Omega_v, \lambda_X > 0} \lambda_X\cdot\bX(X)\right) + \left(\sum_{P \in \Omega_p, \lambda_P > 0}\lambda_{P}\cdot \bP(P)
\right) \subset V.
\]
(The sum of sets appearing in  the above formula is the Minkowski sum, noting that an empty sum of sets equals $\mathbf{0}$.)

For a monomial $M$
over $\Omega$, we define
\[
\evalmap(M;\bX,\bP) \defeq 
\Cmult_{E \in \elem(M)}\evalmap(E;\bX,\bP),
\]
and finally for a polynomial 
$\varphi$
over $\Omega$
we define
\begin{align*}
\evalmap(\varphi;\bX,\bP) &\defeq \emptyset \text{ if $\varphi = \mathbb{0}$}, \\
\evalmap(\varphi;\bX,\bP) &\defeq 
\bigcup_{M \in \mono(\varphi)} \evalmap(M;\bX,\bP), \text{ else.}
\end{align*}
\end{definition}

\begin{example}
    The following example is instructive. Let $X \in \Omega_v = \{X\}$ and $\Omega_p = \{P\}$, 
    and let $\bX(X) = \emptyset$ and $\bP(P) = \{p\}$. 
    Let  $E,E' \in \CL(\Omega)$ (viewed as degree-one monomials)
    be defined by 
    \begin{align*}
        E &= 0 \cdot X + 1 \cdot P, \\
        E' &= (1/2) \cdot X + (1/2) \cdot P.
    \end{align*}
    Then, 
    \begin{align*}
    \evalmap(E;\bX,\bP) &= \{P\}, \\
    \evalmap(E';\bX,\bP) &= \emptyset.
    \end{align*}
\end{example}

\subsection{Solutions of a system of inequalities and the smallest solution}
\label{subsec:solutions}
We are now in a position to define solutions (as well as the smallest solution) of systems of inequalities of the form \eqref{eqn:def:system-of-inequalities}.

Let $\Omega_v = \{X_1,\ldots,X_n\}$, and for an assignment
$\bX \in \convset(V)^{\Omega_v}$, 
we will denote $\bX(X_i)$ by $\bX_i$ in what follows.
\begin{definition}
\label{def:sol}
Fix an assignment $\bP \in (2^V)^{\Omega_p}$.
Then, 
$\bX 
\in\convset(V)^{\Omega_v}
$ 
is a {\em solution} of a system $I$ of inequalities 
\[
\left(X_i\geq \varphi_i\right)_{i=1}^n,
\]
under the assignment $\bP$
if 
\[
\bX_i\geq \evalmap(\varphi_i;\bX,\bP),
\]
for all $1 \leq i \leq n$. 
\end{definition}

\begin{notation}
\label{not:sol}
We denote by  $\sol(I;\bP)\subset \convset(V)^n$ the set of all solutions of this system $I$ under the assignment $\bP$.
\end{notation}

\begin{remark}
\label{rem:ss}
We will show later that $\sol(I;\bP)$ is always non-empty,
closed under arbitrary (component-wise) intersection (see Proposition~\ref{prop:closure-under-intersection}),
which implies that $\sol(I;\bP)$ has a unique minimum element
(under containment) (see Notation~\ref{not:ss}).
\end{remark}

\subsection{Hemihedra}
We next introduce a class of convex subsets below which will play an important role in the rest of the paper.
\begin{definition}
\label{def:hemihedra}
A hemihedron is a convex set $H \subset V$, for which there
exist finitely many polytopes $Q_1,\ldots,Q_m \subset V$ such that
\[
H = \bigcup_{i=1}^{m} \relint{Q_i}
\]
(by definition $\relint{Q} = Q$ if $Q$ is a point). 
The empty union is allowed.
Equivalently, a hemihedron is a bounded convex
semi-linear subset of $V$.
\end{definition}

Illustrative examples of hemihedra in $V=\mathbb{R}^2$ can be seen in Figure~\ref{fig:hemihedra-eg}.
They arise very naturally in the current paper (see Theorem~\ref{thm:hemihedra} below).

\input{arxiv-fig-hemihedra-eg}

\subsection{The good closure property}
\label{subsec:good-closure-property}
\begin{definition}[Good closure property]
\label{def:property:closure}
Let $\mathcal{B}$ be a set of subsets of $V$ containing the empty set. We say that $\mathcal{B}$ has 
the \emph{good closure property} if it is closed under taking
    finite unions, nonnegative scalar multiplication, Minkowski sums, positive geometric joins, and convex hulls.

Moreover, we say that $\mathcal{B}$ has the \emph{effective good closure property},
if it has the good closure property, and additionally there exists an effective algorithm
that takes as input a description of elements $B_1,\ldots,B_m \in \mathcal{B}$ and $b_1,\ldots,b_m \in \mathbb{R}, b_i \geq 0,1 \leq i \leq m$,
and produces as output descriptions of:
\begin{equation}
\label{eqn:def:good-closure-property}
\Cplus_{i=1}^m B_i, \Sigma_{i=1}^{m} B_i, (b_1 \cdot B_1,\ldots,b_m \cdot B_m),
\Cmult_{i=1}^m B_i, \conv{B_1 \cup \cdots \cup B_m}.
\end{equation}
\end{definition}

\begin{example}
\label{eg:closure}
We prove in Proposition~\ref{prop:closure-under-ops} that the following families of subsets of $V$ (in decreasing order with respect to containment) 
have the good closure property.

\begin{enumerate}[(1)]
\item 
\label{itemlabel:eg:closure:1}
The set $\mathcal{B}_1$ of finite unions of
bounded convex subsets of $V$.
\item 
\label{itemlabel:eg:closure:2}
The set  $\mathcal{B}_2$ of finite unions of
bounded convex subsets of $V$ definable in
some fixed o-minimal expansion of $\mathbb{R}$ \cite{Dries-book}.

\item 
\label{itemlabel:eg:closure:3}
The set $\mathcal{B}_3$ of finite unions of
bounded convex semi-algebraic subsets of $V$.

\item 
\label{itemlabel:eg:closure:4}
The set $\mathcal{B}_4$ of finite unions of
hemihedra of $V$.
\end{enumerate}

Moreover, the families $\mathcal{B}_3$ and $\mathcal{B}_4$,  
have the effective good closure property (see Proposition~\ref{prop:closure-under-ops}). 
\end{example}  

\begin{remark}
\label{rem:BSS}
    By the word ``algorithm''  in the current paper, we will mean 
    an algorithm in the sense of the book \cite{BPRbook2}. The input as well as the output  
    are described by 
    quantifier-free first-order formulas with atoms of the form $P=0, P > 0$, where $P$ is any polynomial with coefficients in $\mathbb{R}$. 
    The size of the input is measured by the number and degrees of the polynomials appearing in the input and the number of variables (which is equal to $\dim V$ in our context).
    The complexity of the algorithm
    is the number of arithmetic operations and comparisons that it performs. 
    Effective quantifier elimination in the theory of reals \cite[Chapter 14]{BPRbook2} (in case $\mathcal{B} = \mathcal{B}_3$, see
     Example~\ref{eg:closure})
    or Fourier-Motzkin elimination
    (in case $\mathcal{B} = \mathcal{B}_4$)
    yields an algorithm producing descriptions of the sets in \eqref{eqn:def:good-closure-property}.
\end{remark}

\begin{remark}
\label{rem:hemihedra}
The family of finite unions of hemihedra 
($\mathcal{B}_4$ in Example~\ref{eg:closure})
is the smallest family having the good closure property and containing every singleton subset of $V$.

Indeed, if \(Q\) is a polytope with vertex set
\(\{v_1,\ldots,v_r\}\), then
\[
\{v_1\}\mathring\star\cdots\mathring\star\{v_r\}
=
\operatorname{relint}(Q).
\]
Thus every family with the good closure property containing all
singletons contains the relative interior of every polytope, and hence
every finite union of hemihedra. The converse follows from
Proposition~\ref{prop:closure-under-ops}.
Moreover, it has the effective good closure property (see Proposition~
\ref{prop:closure-under-ops}). It is this family that will play a central role in the application of our main theorem to the study of lamination hulls discussed later in the paper.
\end{remark}

\subsection{Main Result I}
\label{subsec:main1}
We are now in a position to state the first main result of the paper.

\begin{theorem}
\label{thm:hemihedra}
    Let $\mathcal{B}$ be a set of subsets of $V$ that has the good closure property. Let $I
    = (X_i \geq \phi_i)_{1 \leq i \leq n}$,
    and suppose that $\bP$ is an assignment in which each $\bP(P),P \in \Omega_p $ belongs to $\mathcal{B}$.
    Let $\bX$ be
    the unique smallest solution $\bX \in \sol(I;\bP)$
    (see Remark~\ref{rem:ss}). 
    Then, each $\bX_i$ belongs to $\mathcal{B}$ for $1 \leq i \leq n$.

    Moreover, if $\mathcal{B}$ has the effective good closure property, then
    there exists an effective algorithm which takes as input the system $I$, and descriptions 
    of $\bP(P), P \in \Omega_p$, and produces as output  
    descriptions of the $\bX_i, 1 \leq i \leq n$.
\end{theorem}

The following corollary of Theorem~\ref{thm:hemihedra} is a key result used in later applications
and deserves special mention.

\begin{corollary}
  \label{cor:hemihedra}
    Let $I = (X_i \geq \phi_i)_{1 \leq i \leq n}$, 
    and suppose that $\bP$ is an assignment in which each $\bP(P),P \in \Omega_p $ is a finite union of hemihedra. Let $\bX$ be the unique smallest solution $\bX \in \sol(I;\bP)$
    (see Remark~\ref{rem:ss}). 
    Then, each $\bX_i$ is a hemihedron for $1 \leq i \leq n$.

    Moreover, there exists an algorithm which takes as input the system $I$, and descriptions 
    of the finite unions of hemihedra $\bP(P), P \in \Omega_p$, and produces as output 
    semi-linear descriptions of the hemihedra $\bX_i, 1 \leq i \leq n$ (i.e., each $\bX_i$ is described by a quantifier-free first-order formula whose atoms are affine equalities or inequalities).  
\end{corollary}

\begin{remark}[Fixed-point interpretation]
Note that for \emph{each fixed parameter assignment \(\mathbf P\)}, 
the existence of the
smallest solution is also an immediate consequence of the
Knaster--Tarski fixed-point theorem \cite{Tar55}, applied to the monotone operator
\[
\mathbf X\longmapsto
\left(
\operatorname{conv}
\left(
X_i\cup
\operatorname{eval}(\varphi_i;\mathbf X,\mathbf P)
\right)
\right)_{1\leq i\leq n}
\]
on the complete lattice \(\mathcal C(V)^{\Omega_v}\).
Our main result goes substantially beyond this 
existence statement (see Remark~\ref{rem:Knaster-Tarski} for further elaboration on this point).

\end{remark}

\subsection{Lamination hulls}
\label{subsec:laminar}

The second main result of the paper (Theorem~\ref{thm:laminar} below) states the semi-algebraicity of a particular class of lamination hulls. 
The proof of this result uses Theorem~\ref{thm:hemihedra} as an important tool.

\begin{definition}[$\Lambda$-hull]
\label{def:G}
Let $V$ be a finite-dimensional real vector space, and
$\Lambda \subset V$ be some fixed subset 
containing $\mathbf{0}$.
For any subset $S \subset V$, we denote
\begin{equation}
\label{eqn:G}
G_\Lambda(S) = \{\alpha x + (1-\alpha)y \; \mid \; x,y \in S, x - y \in \Lambda, 0 < \alpha < 1\},
\end{equation}
 and for each $t \in \mathbb{N}$, we denote
\[
G^{(t)}_\Lambda(S) = \underbrace{G_\Lambda \circ \cdots \circ G_{\Lambda}}_{t}(S)
\] 
(by convention $G^{(0)}_\Lambda(S)= S$).

We call the subset 
$G_\Lambda^{(\infty)}(S) = \bigcup_{t \geq 0} G^{(t)}_\Lambda(S) \subset V$ 
the \emph{$\Lambda$-hull of $S$}.
\end{definition}

\begin{remark}
\label{rem:monotonicity}
Note that
since we assume $\mathbf{0} \in \Lambda$, then taking $x = y \in S$ in  \eqref{eqn:G}, we get that $S \subset G_\Lambda(S)$, and more generally $G^{(t)}_\Lambda(S) \subset G^{(t+1)}_\Lambda(S)$ for all
$t \geq 0$.
\end{remark}

\begin{remark}
If $\Lambda = V$, then it follows from Carath\'eodory's theorem that for any non-empty subset $S \subset V$, with 
$r:=\dim\operatorname{aff}(S)$,
\[
G^{(\infty)}_\Lambda(S) = G^{(\lceil \log_2 (r+1) \rceil)}_\Lambda(S) = \mathrm{conv}(S).
\]
\end{remark}

\subsection{Main Result II}
\label{subsec:main2}
The following theorem is the second main result of the paper.

\begin{theorem}
\label{thm:laminar}
    Let $V$ be a finite-dimensional real vector space and let 
    $U, W_i, 1 \leq i \leq k$, subspaces of $V$ such that 
    \[
    V = U \oplus \bigoplus_{i=1}^{k} W_i,
    \]
    and $\dim W_i = 1, 1 \leq i \leq k$.
    Let 
    \begin{equation}
        \label{eqn:thm:laminar}
        \Lambda = \bigcup_{i=1}^k (U + W_i).
    \end{equation}
 Then, for each finite subset $S \subset V$, $G_\Lambda^{(\infty)}(S)$ is a semi-algebraic subset of $V$. 

Moreover, there exists an algorithm that, 
given $S \subset V = \mathbb{R}^d$ as input,
produces a semi-algebraic description 
of $G_\Lambda^{(\infty)}(S)$ by a quantifier-free 
first-order formula in the language of the reals 
as its output (the decomposition
\(V=U\oplus\bigoplus_i W_i\) is regarded as fixed).
\end{theorem}

\begin{remark}[Connection to cryptography]
The lamination hulls studied here are motivated by a fundamental problem in cryptography: two-party secure computation. 
Alice holds an input $x\in\{0,1\}$, Bob holds an input $y\in\{0,1\}$, and their goal is to securely sample from a prescribed distribution $f(x,y)$ over a finite set $Z$. 
Can this be done by direct communication alone? 
This question goes back to the 1980s~\cite{FOCS:Yao82b, STOC:GolMicWig87} and has recently been reduced to membership problems for lamination hulls with
\[
    k=2,\qquad \dim W_1=\dim W_2=1,\qquad \dim U=|Z|;
\]
see~\cite{FOCS:BKMN22}. 
In this application, $\dim U$ may be arbitrarily large,
a regime not covered by the previous semi-algebraicity results discussed below. 
\end{remark}

\begin{remark}[Connections to prior work]
As discussed above, lamination hulls associated with different choices of the set of admissible directions $\Lambda$ have been studied extensively, in part because of their connections with the calculus of variations and partial differential equations. 
The particular class of sets $\Lambda$ considered in Theorem~\ref{thm:laminar} has also appeared in several earlier works.
For example, in the special case 
\[
        k=D,\qquad \dim W_i=1\ \text{for all }i,\qquad \dim U=0,
\]
Theorem~\ref{thm:laminar} specializes to the $D$-convex setting. 
More precisely, after choosing coordinates so that the spaces $W_i$ are the (independent) coordinate directions, the set $G^{(\infty)}_\Lambda(S)$ is exactly the $D$-convex hull of $S$. 
Thus, Theorem~\ref{thm:laminar} recovers Matou\v{s}ek's semi-algebraicity theorem for $D$-convex hulls for finite subsets of $\mathbb R^D$ \cite{Matousek2001}. 
In the case
\[
k=2,\qquad
\dim U=\dim W_1=\dim W_2=1,
\]
Theorem~\ref{thm:laminar} gives the semi-algebraicity of the lamination hull associated
with the rank-one directions of upper-triangular \(2\times2\) matrices.

To make this correspondence explicit, let $V$ denote the vector space of $2 \times 2$ upper triangular matrices,
and $U,W_1,W_2$ the subspaces of $V$ defined by:
\begin{align*}
U &= \mathrm{span}\left(\begin{bmatrix} 0 & 1 \\0 & 0 \end{bmatrix}\right), &
W_1 &= \mathrm{span}\left(\begin{bmatrix} 1 & 0 \\0 & 0 \end{bmatrix}\right), &
W_2 &= \mathrm{span}\left(\begin{bmatrix} 0 & 0 \\0 & 1 \end{bmatrix}\right), \\
&& \text{and }\Lambda &= (U+W_1) \cup  (U + W_2).
\end{align*}
Then the lamination hull associated with the rank-one directions
of any finite subset $S$ of $V$ is precisely
$G^{(\infty)}_\Lambda(S)$.
Indeed, two upper triangular \(2\times2\) matrices differ by a rank-one
matrix precisely when their difference lies in one of the two subspaces
\(U+W_1\) or \(U+W_2\).
Theorem~\ref{thm:laminar} implies that 
$G^{(\infty)}_\Lambda(S)$ is  a semi-algebraic set.

Theorem~\ref{thm:laminar} thus contains the preceding examples as special cases. 
Its main new feature is that there is no restriction on $\dim U$: previous semi-algebraicity results for this class of lamination hulls covered only the cases $\dim U\leq 1$.
Moreover, Theorem~\ref{thm:laminar} allows arbitrary $k$; semi-algebraicity of the corresponding $\Lambda$-hull was previously known only in the case $\dim U=0$. 
\end{remark}

The proof methods of the prior results do not extend in any obvious way to the more general case tackled in Theorem~\ref{thm:laminar}. We develop a new method utilizing Theorem~\ref{thm:hemihedra} as an intermediate step.
The following remarks encapsulate the key difficulties that need to be overcome in any proof of Theorem~\ref{thm:laminar}. 

\begin{remark}[$(G^{(t)}_\Lambda(S))_{t \geq 0}$ need not stabilize]
\label{rem:obs1}
Firstly, a naive iterative approach, with the hope that the sequence 
$(G^{(t)}_\Lambda(S))_{t \geq 0}$ will eventually stabilize, does not work.
The  
sequence $(G^{(t)}_\Lambda(S))_{t \geq 0}$ need not stabilize.
Let $V = \mathbb{R}^7$, and let 
$U = \mathrm{span}(e_3,\ldots,e_7)$, and $W_i = \mathrm{span}(e_i), i=1,2$,
where the $e_i$ denote the standard basis vectors.
In this case, $\Lambda = (U + W_1) \cup (U+W_2)$ is the union of two coordinate subspaces of codimension one
defined by $X_2 =0$ and $X_1=0$ respectively. 
Consider the following points in $\RR^2$  
$
a_1 = (3/4, 1/4), \;  a_2 = (1/4, 1/2), \; a_3 = (1/2, 1),\; a_4 = (1, 3/4), \;
a_5= (3/4, 1/2). 
$

Let
\begin{align*}
    S = \{P\in \RR^2 \times \RR^5 \colon   \exists\ i\in \{1,2,3,4,5\},\  \pi(P)=a_i \text{ and } \rho(P)=e_i\},
\end{align*}
where $\pi$ and $\rho$ denote the projections onto 
the first two and the last five coordinates, respectively.

With this choice of $\Lambda$ and $S$, 
it is
shown in \cite{BKMN23} that the sequence $(G_\Lambda^{(t)}(S))_{t > 0}$ does not stabilize. 
\end{remark}

\begin{remark}[$G^{(\infty)}_\Lambda(S)$ need not be semi-linear]
\label{rem:obs2}
A second observation is that
even for finite subsets $S \subset V$, $G^{(\infty)}_\Lambda(S)$
need not always be a semi-linear set. 

Take $V = \RR^3$, 
$U = \mathrm{span}(e_3), W_i = \mathrm{span}(e_i), i=1,2$, and 
$\Lambda = (U + W_1) \cup (U + W_2)$.
Let
\[
S =\{(0,0,0),(0,1,0), (1,0,0), (1,1,1)\}.
\]
It is easy to verify that with these choices,
\[
G^{(\infty)}_\Lambda(S) = G^{(2)}_\Lambda(S) = \{(x,y,z) \mid z = xy, 0 \leq x,y \leq 1\},
\]
which is not a semi-linear (though still a semi-algebraic) subset of $\mathbb{R}^3$.
\end{remark}

\begin{remark}[$G^{(\infty)}_\Lambda(S)$ need not be closed]
\label{rem:obs3}
The set $G^{(\infty)}_\Lambda(S)$ need not be closed even though each $G^{(t)}_\Lambda(S)$ is a closed semi-algebraic set for every finite subset $S \subset V$.
For example, $G^{(\infty)}_\Lambda(S)$ is not a closed set,  
for $S$ defined in 
the example in Remark~\ref{rem:obs1} (see \cite{BKMN23}).
To obtain a precise description of  $G^{(\infty)}_\Lambda(S)$ (rather than its closure) is particularly important in cryptographic applications (this is explained in \cite{FOCS:BKMN22}).
\end{remark}

In view of Remarks~\ref{rem:obs1}, \ref{rem:obs2} and \ref{rem:obs3}, we need a different strategy to prove Theorem~\ref{thm:laminar} from the naive iterative one. 

\medskip
\noindent
\textbf{Our strategy:}
    we first define a finite subset $\cG \subset W$ and prove that 
    the union of fibers $G^{(\infty)}_\Lambda(S)_\cG := G^{(\infty)}_\Lambda(S) \cap \pi^{-1}(\cG)$ 
    (denoting by $\pi:V \rightarrow W$ the canonical projection) 
    can be characterized as the smallest convex solution of a special system of inequalities $I_{\cG, S}$ derived from $S$ and $\cG$ (Proposition~\ref{prop:reduction-lam-system}). 
    We then use Theorem~\ref{thm:hemihedra} to prove that each grid fiber
\(G_\Lambda^{(\infty)}(S)_w\), \(w\in G\), is a hemihedron; hence
\(G_\Lambda^{(\infty)}(S)_G\) is a finite union of hemihedra and is
semi-algebraic. This is a key step. Finally, we prove 
    $G^{(\infty)}_\Lambda(S)$ can be recovered from the set of fibers $G^{(\infty)}_\Lambda(S)_\cG$ 
     using effective quantifier elimination in the theory of reals (Proposition~\ref{prop:struc}), proving Theorem~\ref{thm:laminar}.
     This last step explains why we obtain semi-algebraic sets as 
     $G^{(\infty)}_\Lambda(S)$ which may not be semi-linear (cf. Remark~\ref{rem:obs2}).
\begin{remark}
\label{rem:focs}
Some of the results of the current paper were announced in an
extended abstract
that was published in the Proceedings of the Computer Science Conference, Symposium on Foundations of Computer Science (FOCS), 2025 \cite{FOCS:BKMN25}, 
emphasizing the connections to cryptography,
but without complete proofs of the main theorems. In this paper, we
provide complete proofs, generalize some of the results, 
and also correct an error in the FOCS announcement.
\footnote{
In \cite{FOCS:BKMN25}, 
Theorem~\ref{thm:laminar} was claimed (with $k = 2$) but without any restrictions
on $\dim W_i$. In the current paper we handle arbitrary $k$, but with the added assumption that $\dim W_i =1$. Currently we do not know whether Theorem~\ref{thm:laminar} holds without the restriction on $\dim W_i$.
}
\end{remark}

\medskip
The rest of the paper is devoted to the proofs of Theorems~\ref{thm:hemihedra} and \ref{thm:laminar}. 
These theorems will follow from more technical results that might be of independent interest.
Theorem~\ref{thm:hemihedra} will follow from a more technical result giving a Gaussian-elimination-inspired algorithm for reducing a given system of inequalities to an upper triangular one without changing the 
smallest solution (see Theorem~\ref{thm:gaussian-elim} in Section~\ref{sec:system-algebra} below). 
Similarly, the proof of Theorem~\ref{thm:laminar} depends on several intermediate results that reduce it to solving systems of inequalities to which Theorem~\ref{thm:hemihedra} is applicable (see Propositions~\ref{prop:reduction-lam-system} and \ref{prop:gridded-witness} in Section~\ref{sec:laminar} below). 
Finally, in Section~\ref{sec:example}, we give an example of a concrete
calculation illustrating the different steps of the proof of 
Theorem~\ref{thm:gaussian-elim}.

%% file: arxiv-fig-hemihedra-eg.tex
\begin{figure}[htp]\centering

    \pgfdeclarelayer{background}
    \pgfdeclarelayer{face}
    \pgfdeclarelayer{lines}
    \pgfdeclarelayer{vertices}
    \pgfdeclarelayer{main}
    \pgfdeclarelayer{foreground}
    \pgfsetlayers{background, face, lines, vertices, main, foreground}

    \begin{tikzpicture}[scale = 3]
        \foreach \x/\y/\pos [count=\i] in {0/0/north, -1/2/south east, 1/2/south west, -0.5/2/south, 0.5/2/south}{
            \coordinate (p\i) at (\x,\y);
            \node[anchor=\pos, inner sep = 2.5mm, circle] at (p\i) {$P_{\i}$};
        }
        \node at ($(p1)+(0,-0.5)$) {$\begin{matrix}
                \text{(a) }\convo{\left\{P_1, P_2, P_3\right\}} \bigcup \convo{\left\{P_4,P_5\right\}} \\ 
                \bigcup \convo{\left\{P_1\right\}} \bigcup \convo{\left\{P_4\right\}}
            \end{matrix}$};

        \begin{pgfonlayer}{vertices}
            \foreach \i/\c in { 1/gray, 2/white, 3/white, 4/gray, 5/white} {
                \draw[fill=\c, preaction={draw, line width=8pt, white}] (p\i) circle (0.05);
            }
        \end{pgfonlayer}

        \begin{pgfonlayer}{lines}
            \path[draw, line width=8pt, white] (p1) -- (p2) -- (p3) -- cycle; 
            \path[draw=gray, preaction={draw, line width=8pt, white}] (p4) -- (p5); 
        \end{pgfonlayer}

        \begin{pgfonlayer}{face}
            \fill[fill=gray] (p1) -- (p2) -- (p3) -- cycle; 
        \end{pgfonlayer}

        \coordinate (o) at (2.75,1); 
        \foreach \t/\col [count=\i] in {0/white, 45/gray, 90/white, 135/gray, 180/white, 225/gray, 270/white, 315/gray} {
            \path (o) ++(\t:1) coordinate (p\i); 
            \path (o) ++ (\t:1.2) node {$P_\i$};
            \begin{pgfonlayer}{vertices}
                \path[draw, fill=\col, preaction={draw, line width=8pt, white}] (p\i) circle (0.05);
            \end{pgfonlayer}
        }
        \node at ($(p7)+(0,-0.5)$) {$\begin{matrix}
                \text{(b) }\convo{\left\{P_1, P_2, \dotsc,P_8\right\}} \bigcup \convo{\left\{P_2\right\}}  \\ 
                \bigcup \convo{\left\{P_4\right\}} \bigcup \convo{\left\{P_6\right\}} \bigcup \convo{\left\{P_8\right\}}
            \end{matrix}$};

        \begin{pgfonlayer}{lines}
            \path[draw, line width=8pt, white] (p1) -- (p2) -- (p3) -- (p4) -- (p5) -- (p6) -- (p7) -- (p8) -- cycle; 
        \end{pgfonlayer}

        \begin{pgfonlayer}{face}
            \fill[fill=gray] (p1) -- (p2) -- (p3) -- (p4) -- (p5) -- (p6) -- (p7) -- (p8) -- cycle; 
        \end{pgfonlayer}

    \end{tikzpicture}
    \caption{Examples of hemihedra.
        Here $\convo S:= \relint{\conv{S}}$ represents the relative interior of the convex hull of a finite set of points $S$.
        Note that if $S$ is a singleton set, then $\convo S = S$.
        Shaded circles indicate points that are included, and empty ones indicate points not included in the depicted hemihedra.
    }
    \label{fig:hemihedra-eg}
\end{figure}
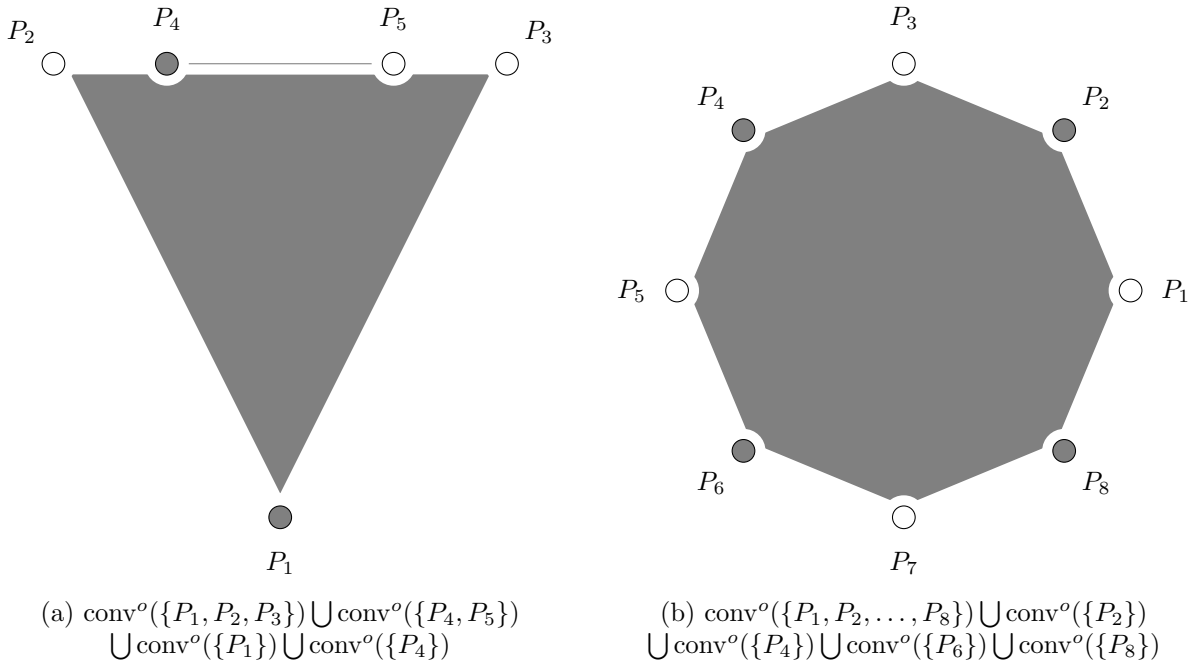

%% file: arxiv-system-algebra.tex
\section{Proof of Theorem~\ref{thm:hemihedra}}
\label{sec:system-algebra}

In this section, we prove Theorem~\ref{thm:hemihedra}. The outline of the section is as follows.
In Subsection~\ref{subsec:preliminary} we prove some preliminary results about the set operations used in Definition~\ref{def:system-of-inequalities} that will be used in the proof of Theorem~\ref{thm:hemihedra}. In Subsection~\ref{subsec:sol} we deduce some basic properties of solutions
of systems of inequalities appearing in Theorem~\ref{thm:hemihedra}.
In Subsection~\ref{subsec:algebraic-solution} we describe a general {\em Gaussian elimination-inspired algebraic technique} to formally transform a system of inequalities while preserving its smallest solution under any evaluation map.
After completing the transformation, the smallest solution is easily characterized, from which we can easily deduce Theorem~\ref{thm:hemihedra}.
In Subsection~\ref{subsec:op-real} 
we describe an iterative method
to obtain a description of 
the smallest solution,
which will play a role in the proof of Theorem~\ref{thm:laminar}.

We include a working example (see Example~\ref{eg:working-example}) throughout to illustrate the various steps of our constructions.

\subsection{Some preliminary results}
\label{subsec:preliminary}

\begin{lemma}
\label{lem:convOfconvo-contains-convoOfconv}
    Let $A,B\subset V$. Then, $\conv{A}\cmult\conv{B}\subset \conv{A\cmult B}$.
    More generally, if $A_1,\ldots,A_r \subset V$, then 
    \[
    \Cmult_{i=1}^r \conv{A_i} \subset \conv{\Cmult_{i=1}^r A_i}. 
    \]
\end{lemma}
\begin{proof}
Let $z\in\conv{A}\cmult \conv{B}$. It follows from \caratheodory's theorem that there are subsets $A'\subseteq A,B'\subseteq B$ of size at most $(\dim V +1)$ such that $$z=\lambda\cdot\left(\sum_{a\in A'}\lambda_a\cdot a\right) + \overline\lambda\cdot \left(\sum_{b\in B'}\lambda_b\cdot b\right),$$
where $\lambda_a\in [0,1]$ for each $a\in A'$, $\lambda_b\in [0,1]$ for each $b\in B'$, and $\lambda\in (0,1)$; and $\sum_{a\in A'}\lambda_a=1$, $\sum_{b\in B'}\lambda_b=1$, and $\lambda+\overline{\lambda}=1$. Then, since $$\sum_{a\in A', \; b\in B'} \lambda_a\lambda_b=\left(\sum_{a\in A'}\lambda_a\right)\left(\sum_{b\in B'}\lambda_b\right)=1,$$
we have,
\begin{align*}
    z=\lambda\cdot\left(\sum_{a\in A'}\lambda_a\cdot a\right) + \overline\lambda\cdot \left(\sum_{b\in B'}\lambda_b\cdot b\right)=\sum_{a\in A', \; b\in B'} \lambda_a\lambda_b \cdot\left(\lambda\cdot a +\overline\lambda \cdot b\right) \in \conv{A\cmult B}.
\end{align*}
This completes the proof of the first inclusion. The more general statement follows from the first inclusion by induction on $r$. 
\end{proof}

We define an equivalence relation on the set of subsets of $V$.
\begin{notation}
\label{not:sim}
For $A,B\subset V$, 
we write  $A\sim B$  if (and only if) $\conv A=\conv B$. 
\end{notation}
We note the following properties of the equivalence relation $\sim$ that we will use without further mention.

\begin{lemma}[Compatibility of \(\sim\) with the set operations]
\label{lem:compatiblity}
Let \(A,A',B,B'\subset V\), and suppose that
\[
A\sim A',
\qquad
B\sim B'.
\]
Then, for every \(\rho\geq 0\),
\[
\rho A\sim \rho A',
\]
and
\[
A\widehat\oplus B
\sim
A'\widehat\oplus B',
\qquad
A+B\sim A'+B',
\qquad
A\mathring\star B
\sim
A'\mathring\star B'.
\]
More generally, the analogous statements hold for finite unions,
finite Minkowski sums, and finite positive geometric joins.
Moreover, if \(C\subset V\) is convex and \(A\sim A'\), then
\[
C\supset A
\quad\Longleftrightarrow\quad
C\supset A'.
\]
\end{lemma}

\begin{proof}
    The proof for scalar multiplication, union, and Minkowski sum follows from
the easily checked facts:
$\conv{\rho \;A} = \rho\; \conv{A},
\conv{A+B}=\conv{A} + \conv{B},
\conv{A \cup B} = \conv{ \conv{A} \cup \conv{B}}$.
For positive geometric joins, 
Lemma~\ref{lem:convOfconvo-contains-convoOfconv}  gives
\[
\conv{A}\cmult \conv{B} \subset \conv{A \cmult B},
\]
while monotonicity gives the reverse inclusion after taking convex hulls. 
Thus
\[
\conv{A \cmult B}=\conv{\conv{A} \cmult \conv{B}}
\]
which depends only on $\conv{A}$ and $\conv{B}$.

The more general statements follow by a simple inductive argument.

Finally, if \(C\) is convex and \(C\supset A\), then
\[
C\supseteq\operatorname{conv}(A)
=\operatorname{conv}(A')
\supset A'.
\]
The converse follows by symmetry.
\end{proof}
 
We now state and prove some basic algebraic properties of the set operations defined before which will be useful later.

\begin{lemma}
\label{lem:prop-set-ops}
    For any subsets $A,B,C\subset V$ and $0<\rho<1$, the following identities hold. 
    \begin{align}
        A \cmult A &\sim A \label{eqn:idempotent-cmult}\\
        A\cmult(B\cmult C) &= (A\cmult B)\cmult C \label{eqn:associative-cmult}\\
        A\cmult(B\cplus C) &= (A\cmult B) \;\cplus\; (A\cmult C) \label{eqn:dist-cmult-cplus}\\
        \rho\cdot (A \cplus B) &= (\rho\cdot A) \;\cplus\; (\rho\cdot B) \label{eqn:dist-scal-cplus},\\
        \rho\cdot (A\cmult B) &= (\rho\cdot A) \;\cmult\; (\rho\cdot B) \label{eqn:dist-scal-cmult}\\
        (A\cmult B) + C &\sim (A+C)\;\cmult\;(B+C) \label{eqn:dist-minkowski-cmult}
    \end{align}

More generally, for $A,A_1,\ldots,A_m, B_1,\ldots,B_n, C\subset V$ and $0 < \rho < 1$,
\begin{align}
       A \cmult \cdots \cmult A  &\sim A \label{eqn:idempotent-cmult-general}\\
        A\cmult\left(\Cplus_{i=1}^n B_i\right) &= \Cplus_{i=1}^n\left(A\cmult B_i\right) \label{eqn:dist-cmult-cplus-general}\\
        \rho\cdot \left(\Cplus_{i=1}^m A_i\right) &= \Cplus_{i=1}^m \left(\rho\cdot A_i\right) \label{eqn:dist-scal-cplus-general},\\
        \rho\cdot \left(\Cmult_{i=1}^m A_i\right) &= \Cmult_{i=1}^m \left(\rho\cdot A_i\right) \label{eqn:dist-scal-cmult-general}\\
        \Cmult_{i=1}^m A_i + C &\sim  \Cmult_{i=1}^m \left(A_i+C\right) \label{eqn:dist-minkowski-cmult-general}
    \end{align}
\end{lemma}
The proofs of the above equivalences/equalities are not difficult. We 
include them below for the reader's convenience. 

\begin{proof}[Proof of Lemma~\ref{lem:prop-set-ops}]
\noindent
\begin{proof}[Proof of \eqref{eqn:idempotent-cmult}: $A\cmult A \sim A$]
    We will show that $\conv{A\cmult A} = \conv A$. 
    To prove $\conv{A} \subseteq \conv{A\cmult A}$, it suffices to prove that $A\subseteq A\cmult A$. 
    This follows from the fact that, for any $a\in A$, we have
    $a = \frac12\cdot a + \frac12\cdot a \in A\cmult A$. 
    
    To prove $\conv{A \cmult A} \subseteq \conv{A}$, it suffices to prove that $A\cmult A \subseteq \conv A$. 
    Let $\lambda\cdot a + (1-\lambda)\cdot a' \in A\cmult A$, for some  $a,a'\in A$ and $0<\lambda<1$.
    Since $\conv A$ is convex, it is immediate that $\lambda\cdot a + (1-\lambda)\cdot a' \in \conv A$ for any $a,a'\in A\subseteq \conv A$. 
\end{proof}

\begin{proof}[Proof of \eqref{eqn:associative-cmult}: $A\cmult (B\cmult C)=(A\cmult B)\cmult C$] 
We show that both sets $A\cmult (B\cmult C)$ and $(A\cmult B)\cmult C$ are equal to the following set:
 $$ L\defeq \left\{\vphantom{2^{2^2}}\alpha\cdot a + \beta\cdot b + \gamma\cdot c \;\colon\; a\in A, b\in B, c\in C, \text{ and }\alpha,\beta,\gamma>0 \text{ satisfying } \alpha+\beta+\gamma = 1\right\}.$$
We first show that $L=A\cmult (B\cmult C)$.
To show $L\subset A\cmult (B\cmult C)$ let $a\in A, b\in B,c\in C$ and $\alpha,\beta,\gamma\in (0,1)$, with $\alpha+\beta+\gamma=1$. Then, 
\[
\alpha\cdot a + \beta\cdot b + \gamma\cdot c = \alpha\cdot a+(1-\alpha)\cdot \left(\frac{\beta}{1-\alpha} \cdot b + \frac{\gamma}{1-\alpha}\cdot c\right) \in A\cmult (B\cmult C),
\]
since $\alpha, \frac{\beta}{1-\alpha}, \frac{\gamma}{1-\alpha}\in (0,1)$, and $\frac{\beta}{1-\alpha}+\frac{\gamma}{1-\alpha}=1$.

To show
$A\cmult (B\cmult C)\subset L$, let $a\in A,b\in B,c\in C$, and  $\alpha,\lambda\in (0,1)$. Then, we can rewrite the point $\alpha\cdot a+(1-\alpha)\cdot \left(\lambda\cdot b+(1-\lambda)\cdot c\right)\in A\cmult (B\cmult C)$ as
$\alpha\cdot a+(1-\alpha)\lambda \cdot b+(1-\alpha)(1-\lambda)\cdot c$.
The latter clearly belongs to $L$ because $\alpha, (1-\alpha)\lambda, (1-\alpha)(1-\lambda) > 0$ and 
$\alpha+(1-\alpha)\lambda + (1-\alpha)(1-\lambda) = 1$.\newline

\noindent The proof of $L=(A\cmult B)\cmult C$ follows similarly after exchanging the roles of $A$ and $C$.
 \end{proof}

\begin{proof}[Proof of \eqref{eqn:dist-cmult-cplus}: $A\cmult (B\cplus C)=(A\cmult B)\cplus (A\cmult C)$] 
To prove that $A\cmult (B\cplus C)\subset (A\cmult B)\cplus (A\cmult C)$, let $e\in A\cmult (B\cplus C)$. Then, there exist $a\in A, d\in B\cplus C$ and $\lambda\in (0,1)$ such that $e=\lambda \cdot a+(1-\lambda)\cdot d$. Since $d\in B\cplus C$, we have $d\in B$ or $d\in C$. If $d\in B$, then $e\in A\cmult B$, and if $d\in C$, then $e\in A\cmult C$. Thus, we have $e\in (A\cmult B)\cplus (A\cmult C)$.

To prove that $(A\cmult B)\cplus (A\cmult C)\subset A\cmult (B\cplus C)$,  let $e\in A\cmult B$. Then, there exist $a\in A, b\in B$, and $\lambda\in (0,1)$ such that $e=\lambda\cdot a+(1-\lambda)\cdot b$. Since $b\in B\cplus C$, we have $e\in A\cmult (B\cplus C)$. This implies that $A\cmult B\subset A\cmult (B\cplus C)$. Similarly, we can show that $A\cmult C\subset A\cmult (B\cplus C)$. Thus, we have $(A\cmult B)\cplus (A\cmult C)\subset A\cmult (B\cplus C)$.  
\end{proof}

\begin{proof}[Proof of \eqref{eqn:dist-scal-cplus}: $\rho\cdot (A\cplus B)=(\rho\cdot A)\cplus (\rho\cdot B)$] 
To prove that $\rho\cdot (A\cplus B) \subset(\rho\cdot A)\cplus (\rho\cdot B)$,
let $d\in \rho\cdot (A\cplus B)$. Then, there exists $c\in A\cplus B$ such that $d=\rho\cdot c$. It follows from $c\in A\cplus B$ that $c\in A$ or $c\in B$.
If $c\in A$, then we have $d=\rho\cdot c\in \rho\cdot A$, and if $c\in B$, then we have $d=\rho\cdot c\in \rho\cdot B$. Thus, we conclude that $d\in \rho\cdot A\cplus \rho\cdot B$.
We now prove that $(\rho\cdot A)\cplus (\rho\cdot B)\subset \rho\cdot (A\cplus B)$.
    Since $A\subseteq A\cplus B$, we have $\rho\cdot A\subset \rho\cdot (A\cplus B)$. Similarly, we have $\rho\cdot B\subset \rho \cdot (A\cplus B)$. This implies that $(\rho\cdot A)\cplus (\rho\cdot B)\subset \rho\cdot (A\cplus B)$.
\end{proof}

\begin{proof}[Proof of \eqref{eqn:dist-scal-cmult}: $\rho\cdot (A\cmult B) = (\rho\cdot A) \cmult(\rho\cdot B)$] 
To prove that $\rho\cdot (A\cmult B) \subset (\rho\cdot A) \cmult (\rho\cdot B)$,
let $a\in A, b\in B$, and $\lambda\in (0,1)$. Then, 
we can rewrite 
$\rho\cdot (\lambda\cdot a+(1-\lambda)\cdot b)\in \rho\cdot (A \cmult B)$ as $\lambda\cdot (\rho\cdot a)+(1-\lambda)\cdot (\rho\cdot b)\in (\rho\cdot A) \cmult (\rho\cdot B)$. 

To prove that $(\rho\cdot A) \cmult (\rho\cdot B)\subset \rho\cdot (A\cmult B)$, let $a\in A, b\in B$, and $\lambda\in (0,1)$. Then, we can rewrite  $\lambda\cdot (\rho\cdot a)+(1-\lambda)\cdot (\rho\cdot b)\in (\rho\cdot A) \cmult (\rho\cdot B)$ as $\rho\cdot (\lambda\cdot a+(1-\lambda)\cdot b)\in \rho\cdot (A \cmult B)$.
\end{proof}

\begin{proof}[Proof of \eqref{eqn:dist-minkowski-cmult}: $(A\cmult B) + C \sim (A+C)\cmult(B+C)$]
    We need to prove that 
    \[
    \conv{(A\cmult B) + C} = \conv{(A+C)\cmult(B+C)}.
    \]
    
 To prove that $\conv{(A\cmult B) + C}\subset \conv{(A+C)\;\cmult\;(B+C)}$,
it suffices to prove that $(A\cmult B) + C \subset (A+C) \cmult (B+C)$.
Let $(\lambda\cdot a + (1-\lambda)\cdot b) + c \in \left(A\cmult B\right) + C$ for some $a\in A, b\in B, c\in C$, and $0<\lambda<1$. 
    We rewrite this point as $\lambda\cdot(a+c) \;+\; (1-\lambda)\cdot(b+c)$, which clearly belongs to $(A+C) \cmult (B+C)$.

   To prove that $\conv{(A+C)\;\cmult\;(B+C)}\subset \conv{(A\cmult B) + C}$,
    it suffices to prove that $(A+C)\;\cmult\;(B+C) \subset \conv{(A\cmult B) + C}$.
    Let $\lambda\cdot(a+c) \;+\; (1-\lambda)\cdot (b+c')\in (A+C)\;\cmult\;(B+C)$ for some $a\in A, b\in B,$  $c,c'\in C$, and $0<\lambda<1$. We
    rewrite it as $$\lambda\cdot\left((\lambda\cdot a+(1-\lambda)\cdot b)+c\right)+(1-\lambda)\cdot \left((\lambda\cdot a+(1-\lambda)\cdot b)+c'\right)\in \conv{(A\cmult B) + C}. \qedhere$$
\end{proof}
Finally, the more general versions, namely 
\eqref{eqn:idempotent-cmult-general}, \eqref{eqn:dist-cmult-cplus-general}, \eqref{eqn:dist-scal-cplus-general},
\eqref{eqn:dist-scal-cmult-general},
\eqref{eqn:dist-minkowski-cmult-general}
follow by induction from the corresponding statements 
(respectively, \eqref{eqn:idempotent-cmult}, \eqref{eqn:dist-cmult-cplus}, \eqref{eqn:dist-scal-cplus},
\eqref{eqn:dist-scal-cmult},
\eqref{eqn:dist-minkowski-cmult}) proved previously and using Lemma~\ref{lem:compatiblity}.
\end{proof}

\begin{example}
\label{eg:working-example}
Let $\Omega_p = \{P_1,\ldots,P_4\}$ and $\Omega_v = \{X_1,X_2\}$.
We will use the following system of inequalities as a running example
(see Section~\ref{sec:example}).

\begin{equation}
\left\{
\begin{aligned}
    X_1 &\geq P_1  \;\cplus\;  X_1\cmult \left(\frac12\cdot X_2 + \frac12\cdot P_3\right) \\
    X_2 &\geq P_2  \;\cplus\;  X_2\cmult \left(\frac12\cdot X_1 + \frac12\cdot P_4\right) 
\end{aligned}
\right.
    \label{eqn:running-example}
\end{equation}
\end{example}

\subsection{Properties of solutions of systems of inequalities}
\label{subsec:sol}
We first observe that:
\begin{proposition}
\label{prop:sol}
    $\sol(I;\bP)\neq \emptyset$.
\end{proposition}
\begin{proof}
    Note that $\bX_1 = \dotsi = \bX_n = U$, where $U$ is the convex hull of $\bP_1 \cplus \dotsi \cplus \bP_m$, is a solution. 
    This is because $U$ contains the evaluation of any element in $\CL(\Omega)$ with assignments that are subsets of $U$. 
    After that, the containment of monomials and polynomials is also immediate. 
\end{proof}
We next observe that the set $\sol(I;\bP)$ is closed under arbitrary intersection.

\begin{proposition}
\label{prop:closure-under-intersection}
   Let $\mathcal{X} \subset \sol(I;\bP)$.
    Then, 
    \[
    \bigcap_{\bX \in \mathcal{X}} \bX \defeq \left(\bigcap_{\bX \in \mathcal{X}}\bX_i\right)_{1\leq i \leq n} \in \sol(I;\bP).
    \]
\end{proposition}

Before proving Proposition~\ref{prop:closure-under-intersection} we need the following lemma.

\begin{lemma}
\label{lem:monotonicity}
    Suppose $\bX,\bY \in \convset(V)^{\Omega_v}$ with $\bX\geq \bY$ and
    $\bP,\bQ \in (2^V)^{\Omega_p}$ with $\bP\geq \bQ$, and  $\varphi \in \mathbb{P}[\Omega]$.
    Then,
        $ \evalmap\left(\varphi ; \bX,\bP\right) \geq \evalmap\left(\varphi ; \bY,\bQ\right) $.
    Furthermore, 
    if $\bX(X) \sim \bY(X)$ for all $X \in \Omega_v$, and
    $\bP(P) \sim \bQ(P)$ for all $P \in \Omega_p$, then
    \[
    \evalmap\left(\varphi ; \bX,\bP\right) \sim \evalmap\left(\varphi ; \bY,\bQ\right).
    \]
\end{lemma}
\begin{proof}
    It suffices to prove the result for monomials. 
    Suppose $\bX\geq \bY$, and $\bP\geq \bQ$. Then, $\bY(X)\subset \bX(X)$ for all $X \in \Omega_v$,
    and  $\bQ(P) \subset \bP(P) $ for all $P \in \Omega_p$. 
    
    Let $\varphi$ be a monomial $M$. For each 
    $E = 
\sum_{X \in \Omega_v, \lambda_X >0} \lambda_X \cdot X + \sum_{P \in \Omega_p, \lambda_P > 0} \lambda_P \cdot P    
    \in \elem(M)$, we have, 
\begin{align*}
\evalmap\left(E;\bY,\bQ\right)&=
\sum_{X \in \Omega_v, \lambda_X >0} \lambda_X \cdot \bY(X) + \sum_{P \in \Omega_p, \lambda_P > 0} \lambda_P \cdot \bQ(P) \\
&\subset 
\sum_{X \in \Omega_v, \lambda_X > 0} \lambda_X \cdot \bX(X) + \sum_{P \in \Omega_p, \lambda_P > 0} \lambda_P \cdot \bP(P)\\
&=\evalmap\left(E;\bX,\bP\right).
\end{align*}
Thus, we have,
\begin{align*}
    \evalmap\left(M; \bY,\bQ\right)&=\Cmult_{E\in\elem(M)} ~ \evalmap(E;\bY,\bQ)\\
    &\subset \Cmult_{E\in\elem(M)} ~ \evalmap(E;\bX,\bP)\\
    &=\evalmap\left(M; \bX,\bP\right).
\end{align*}
This implies that $\evalmap\left(M;\bX,\bP\right)\geq \evalmap\left(M;\bY,\bQ\right)$.

Now, suppose that $\bX(X)\sim \bY(X)$ for all $X \in \Omega_v$, and $\bP(P)\sim\bQ(P)$ for all $P \in \Omega_p$.

Thus, $\conv{\bX(X)}=\conv{\bY(X)}$ for all $X \in \Omega_v$, and $\conv{\bP(P)}=\conv{\bQ(P)}$ for all $P \in \Omega_p$.

For each $E = 
\sum_{X \in \Omega_v, \lambda_X > 0} \lambda_X \cdot X + \sum_{P \in \Omega_p, \lambda_P > 0}\lambda_P \cdot P \in \elem(M)$, we have
\begin{align*}
    \evalmap\left(E;\bY,\bQ\right)&\subset \conv{\evalmap(E;\bY,\bQ)}\\
    &=\conv{\sum_{X \in \Omega_v, \lambda_X > 0} \lambda_X \cdot \bY(X) + \sum_{P \in \Omega_p, \lambda_P > 0} \lambda_P \cdot \bQ(P)}\\
    &= \sum_{X \in \Omega_v, \lambda_X > 0} \lambda_X\cdot \conv{\bY(X)} + 
    \sum_{P \in \Omega_p, \lambda_P > 0} \lambda_P \cdot \conv{\bQ(P)} \\
    &=\sum_{X \in \Omega_v, \lambda_X > 0} \lambda_X\cdot \conv{\bX(X)} + 
    \sum_{P \in \Omega_p, \lambda_P > 0} \lambda_P \cdot \conv{\bP(P)} \\
    &=\conv{\sum_{X \in \Omega_v, \lambda_X > 0} \lambda_X \cdot \bX(X) + \sum_{P \in \Omega_p, \lambda_P > 0} \lambda_P \cdot \bP(P)}\\
    &=\conv{\evalmap(E; \bX,\bP)}.
\end{align*}
This implies that
    \begin{align*}
        \evalmap\left(M; \bY,\bQ\right)&=\Cmult_{E\in\elem(M)} ~ \evalmap(E;\bY,\bQ)\\
        &\subset \Cmult_{E\in\elem(M)} ~ \conv{\evalmap(E;\bX,\bP)}\\ \tag{By Lemma~\ref{lem:convOfconvo-contains-convoOfconv}}
        &\subset \conv{\Cmult_{E\in\elem(M)} ~ \evalmap(E;\bX,\bP)}\\
        &= \conv{\evalmap\left(M; \bX,\bP\right)}.
    \end{align*}
    Thus, we have
$$\conv{\evalmap\left(M; \bY,\bQ\right)}\subset \conv{\evalmap\left(M; \bX,\bP\right)}.$$
Similarly, we can show that $\conv{\evalmap\left(M; \bX,\bP\right)}\subset \conv{\evalmap\left(M; \bY,\bQ\right)}$. Thus, we have $$\conv{\evalmap\left(M; \bX,\bP\right)}= \conv{\evalmap\left(M; \bY,\bQ\right)},$$ as desired. 
\end{proof}

\begin{proof}[Proof of Proposition~\ref{prop:closure-under-intersection}]
We have for each $\bX \in \mathcal{X}$
\begin{align*}
    \bX_i &\geq \evalmap\left(\varphi_i;\bX,\bP\right) & \text{ (since $\bX \in \mathcal{X} \subset \sol(I;\bP)$)}\\
    & \geq  \evalmap\left(\varphi_i;\mathop{\bigcap}\limits_{\bX \in \mathcal{X}}\bX ,\bP\right) & \text{ (using Lemma~\ref{lem:monotonicity} and $\bX \geq \mathop{\bigcap}\limits_{\bX \in \mathcal{X}}\bX$)}.
\end{align*}

Thus, we conclude that $\mathop{\bigcap}\limits_{\bX \in \mathcal{X}}\bX_i\geq \evalmap\left(\varphi_i;\mathop{\bigcap}\limits_{\bX \in \mathcal{X}}\bX ,\bP\right)$. Therefore, we have the following for every $i\in \{1,2,\dots,n\}$: $$\left(\mathop{\bigcap}\limits_{\bX \in \mathcal{X}}\bX\right)_i=\mathop{\bigcap}\limits_{\bX \in \mathcal{X}}\bX_i\geq \evalmap\left(\varphi_i;\mathop{\bigcap}\limits_{\bX \in \mathcal{X}}\bX ,\bP\right),$$
which implies that $\mathop{\bigcap}\limits_{\bX \in \mathcal{X}}\bX \in \sol(I;\bP)$. 
\end{proof}

Taking $\mathcal{X} = \sol(I;\bP)$ in Proposition~\ref{prop:closure-under-intersection} we get that the intersection of all solutions in
$\sol(I;\bP)$ is also an element of $\sol(I;\bP)$ -- the {\em smallest solution} of $I$. 

\begin{notation}
\label{not:ss}
We denote the smallest solution of $I$ by 
\begin{equation}
\label{eqn:smallest-sol-def}
    \ssol(I;\bP) \defeq \bigcap\limits_{\bX\in\sol(I;\bP)} \bX.
\end{equation}
\end{notation}

\paragraph{\emph{Intuition behind the $\cmult$ operation}.}
First, let us elaborate on the evaluation of an expression $A\cmult B$, where $A, B\subset V$. 
When $A$ and $B$ are singleton sets, $A\cmult B$ represents the relative interior of the line segment joining the two points. 
Likewise, for singleton sets $A, B,\dotsc, C\subset V$, the set $A\cmult B \cmult \dotsi \cmult C$ is equal to the relative interior of the convex hull $\conv{A\cplus B \cplus \dotsi \cplus C}$. 

In general (when $A$ and $B$ are not singleton sets), the set $A\cmult B$ is the set of all points that can be expressed as $\lambda\cdot a + (1-\lambda)\cdot b$ for some $a\in A$ and $b\in B$. 
Intuitively, these points are in the relative interior of the line segment $\overline{ab}$ for some $a\in A$ and $b\in B$. 
Clearly, $A\cmult B$ is contained in $\conv{A\cplus B}$.
We do not know how to characterize this set using other elementary set operators precisely. 
However, the set $A\cmult B\cmult \dotsi \cmult C$ is semi-algebraic if the sets $A, B,\dotsc, C$ are semi-algebraic, using standard quantifier elimination (see, for example, \cite[Chapter 14]{BPRbook2}).
These sets will be crucial to {\em characterizing the smallest solution to our systems with a succinct closed-form expression}. 

We note that even if the original system does not have $\cmult$ in the inequalities, its smallest solutions may contain $\cmult$.
For example, the following system of equations, which does not use the $\cmult$ operator in its inequalities, has a solution set identical to that of the working example system \eqref {eqn:running-example} we have been considering. 

\begin{equation}
\left\{
\begin{aligned}
    X_1 &\geq P_1  \;\cplus\;  \left(\frac12\cdot X_1 + \frac14\cdot X_2 + \frac14\cdot P_3\right)\\
    X_2 &\geq P_2  \;\cplus\;  \left(\frac12\cdot X_2 + \frac14\cdot X_1 + \frac14\cdot P_4\right)
\end{aligned}
\right.
\end{equation}
This fact follows from the property that $X\in \convset(V)$ satisfies $X\geq \rho\cdot X + (1-\rho)\cdot A$ if (and only if) $X\geq X\cmult A$, for any $A\subset V$ and $0<\rho<1$ 
and the following lemma.

\begin{lemma}
\label{lem:pullout-X}
Consider convex $X\in\convset(V)$, arbitrary sets $A,B\subset V$, and $0<\rho<1$.
\begin{enumerate}[1.]
\item
\label{itemlabel:lem:pullout-X:1}
$X\geq (\rho\cdot X + (1-\rho)\cdot A)$ if and only if $X\geq X\cmult A$
\item 
\label{itemlabel:lem:pullout-X:2}
$X\geq (\rho\cdot X + (1-\rho)\cdot A) \;\cmult\; B$ if and only if $X\geq X\cmult A\cmult B$
\end{enumerate}
\end{lemma}

\begin{proof}
We prove both parts below. \\

\noindent
Proof of Part~\eqref{itemlabel:lem:pullout-X:1}: 
\paragraph{Proof of `if.'} By definition, we have 
$\rho\cdot X+(1-\rho)\cdot A\subseteq X\cmult A$ when 
$\rho\in (0,1)$. Therefore, $X\geq X\cmult A$ implies 
$X\geq (\rho\cdot X + (1-\rho)\cdot A)$. 

\paragraph{Proof of `only if.'} Suppose that $X\geq (\rho\cdot X + (1-\rho)\cdot A)$. Let $x\in X$, and $a\in A$. It follows from the hypothesis that $\rho\cdot x+(1-\rho)\cdot a\in X$.
 We now show that if $\rho\cdot x+(1-\rho)\cdot a\in X$ then $\lambda\cdot x+(1-\lambda)\cdot a\in X$, for all $\lambda\in (0,1)$.

Define $x\p 0 \defeq x \in X$. 
    Then, inductively for $i\in\{0,1,2,\dotsc\}$, the point $x\p{i+1} \defeq \rho\cdot x\p i + (1-\rho)\cdot a$ also belongs to $X$ using the fact that $X \geq \rho\cdot X + (1-\rho)\cdot A$.
    By convexity of $X$, the line segment joining the points $x$ and $x\p i$ is a subset of $X$. 
    
    Note that $x\p i = \rho^i\cdot x + (1-\rho^i)\cdot a$. 
    Let $\lambda\in(0,1)$ and any $i_\lambda\in\{0,1,2,\dotsc\}$ satisfying 
    $\rho^{i_\lambda} \leq \lambda$.
    Since $\rho^i \rightarrow 0$, such an $i_\lambda$ exists for every
    $\lambda \in (0,1)$.
    Then, the point $\lambda\cdot x  +(1-\lambda)\cdot a$ is on the line segment joining $x$ and $x\p{i_\lambda}$, which is a subset of $X$.

    This shows  that $\lambda\cdot X+(1-\lambda)\cdot A\subseteq X$, for any $\lambda\in (0,1)$; in turn, implying that $X\geq X\cmult A$. 
    Thus, $X\geq (\rho\cdot X + (1-\rho)\cdot A)$ implies $X\geq X\cmult A$.

\medskip
\noindent
Proof of Part~\eqref{itemlabel:lem:pullout-X:2}:
\paragraph{Proof of `if.'} By definition, we have 
$\left(\rho\cdot X+(1-\rho)\cdot A\right)\cmult B\subseteq \left(X\cmult A\right) \cmult B= X\cmult A\cmult B$ when $\rho\in (0,1)$. Therefore, $X\geq X\cmult A\cmult B$ implies $X\geq (\rho\cdot X+(1-\rho)\cdot A)\cmult B$. 

\paragraph{Proof of `only if.'} Suppose $X\geq (\rho\cdot X + (1-\rho)\cdot A)\cmult B$. 
Let $x\in X, a\in A$, and $b\in B$, and 
$u,v\in (0,1)$. Let $\overline{u}=1-u, \overline{v}=1-v$. We show that the point $p\defeq u\cdot x+ \overline{u}\cdot (v\cdot a+\overline{v}\cdot b)\in X\cmult A\cmult B$ is 
also in $X$. 

Inductively define a sequence of points $x\p i$, for $i\in \{0,1,2,\dots\}$. To begin, define $x\p 0\defeq x$, and define
$$x\p {i+1}\defeq \frac{v}{v+\overline{v}\;\overline{\rho}}\cdot \left(\rho\cdot x\p i+\overline{\rho}\cdot a\right)+\frac{\overline{v}\;\overline{\rho}}{v+\overline{v}\;\overline{\rho}}\cdot b,
$$
where $\overline{\rho}=1-\rho$. Note that, inductively, if $x\p i\in X$, then $x\p {i+1}\in (\rho\cdot X + (1-\rho)\cdot A)\cmult B$, so 
$x\p {i+1}\in X$ according to the hypothesis. Hence, $\{x\p 0,x\p 1,\dots\}\subseteq X$. 

We will prove that $x\p i$ can be written in the following form: 
    $$x\p i=\mu \p i\cdot x+v(1-\mu \p i)\cdot a+\overline{v}(1-\mu \p i)\cdot b.$$
For (base case) $i=0$, we know $\mu\p 0=1$. By the recursive definition, we have:
$$\mu \p {i+1}=\frac{v\rho}{v+\overline{v}\;\overline{\rho}}\cdot \mu\p i$$
Let $\mu=\frac{v\rho}{v+\overline{v}\;\overline{\rho}}$. Then 
for $i\in \{0,1,2,\dots\}$, we have:
$$x\p i=\mu ^ i\cdot x+v(1-\mu ^ i)\cdot a+\overline{v}(1-\mu ^ i)\cdot b.$$

Observe that $0 < \mu <1$. 
Let $i_u\in\{0,1,2,\dotsc\}$ be an index such that $\mu^{i_u}<u$. 
Then, the point $p$ belongs to the line segment joining the points $x$ and $x\p{i_u}$. 
By convexity of $X$, we conclude that $p\in X$.
\end{proof}

\paragraph{Working example.}
For illustrative purposes, consider 
the case $V = \RR^2$.
Here, 
$\cC(\RR^2)$
denotes the set of all convex subsets of $\RR^2$. 
Fix an arbitrary assignment $\bP$ to the constants. 
The semantics of the first equation in our system is 
\begin{align*}
    X_1 &\text{ contains the set } \bP_1 \text{, and }\\
    X_1 &\text{ contains the set } X_1 \cmult \left(\frac12\cdot X_2 + \frac12\cdot \bP_3\right).
\end{align*}
The semantics of the second equation is analogous. 
The smallest solution of our example system has the following closed-form expression. 

\begin{align}
\label{eqn:example:ss:1}
    \ssol(I;\bP)_1 &= \conv{ \;  
                            \bP_1  
                            \;\cplus\; \bP_1\cmult\left(\frac12\cdot \bP_2 + \frac12\cdot\bP_3\right) 
                            \;\cplus\; \bP_1\cmult\left(\frac12\cdot \bP_2 + \frac12\cdot\bP_3\right) \cmult \left(\frac23\cdot \bP_3+ \frac13\cdot \bP_4\right)
                        \;}\\
    \label{eqn:example:ss:2}
    \ssol(I;\bP)_2 &= \conv{ \;  
                            \bP_2  
                            \;\cplus\; \bP_2\cmult\left(\frac12\cdot \bP_1 + \frac12\cdot\bP_4\right) 
                            \;\cplus\; \bP_2\cmult\left(\frac12\cdot \bP_1 + \frac12\cdot\bP_4\right) \cmult \left(\frac23\cdot \bP_4+ \frac13\cdot \bP_3\right)
                        \;}
\end{align}

This expression is not unique by any means and is different from
the one produced by the algorithm that we describe later for obtaining such expressions (see Theorem~\ref{thm:gaussian-elim} and the explicit calculation in Section~\ref{sec:example}).

Note also that the formal expression for the smallest solution on the RHS is independent of the specific constant assignment $\bP$ used; the expression holds for any constant assignment. 
Our algebraic approach to identifying the smallest solution of a system will also be independent of the specific constant assignment.
Determining the evaluation of the smallest solution will need $\bP$. 
The next section presents a finite procedure to obtain their succinct closed-form expression.

Consider singleton sets $\bP_1, \bP_2, \bP_3, \bP_4$ to illustrate the smallest solution; refer to \figureref{our-sol} for an example. 
The set $\ssol(I;\bP)_1$ is the smallest convex set containing:
\begin{enumerate}
    \item the point in $\bP_1$, 
    \item the relative interior of the line segment joining the two points in $\bP_1$ and $\frac12\cdot \bP_2 + \frac12\cdot\bP_3$, and 
    \item the relative interior of the triangle formed by the three points in $\bP_1$, $\frac12\cdot \bP_2 + \frac12\cdot\bP_3$, and $\frac23\cdot \bP_3+ \frac13\cdot \bP_4$.
\end{enumerate}
Note that the union of the three sets above is convex in this case. 
The set $\ssol(I;\bP)_2$ is similarly defined. 
As we will see later, our Gaussian elimination-inspired solution methodology will recover these solutions, albeit possibly with slightly different descriptions.

\begin{remark}[Solutions Restricted to Polytopes]
    Consider the objective of restricting solutions to polytopes (instead of allowing arbitrary convex sets). 
    In this case, our positive geometric join operator $\cmult$ is not needed to represent the smallest solution
        because the smallest polytope containing the set $A\cmult B$ is identical to the polytope containing $A\cplus B$. 
    Thus, ``linear'' polynomials (i.e., polynomials with only degree-1 monomials) can express the constraints for polytope solutions.  
\end{remark}

\subsection{Algebraic Characterization of the Smallest Solution}
\label{subsec:algebraic-solution}
In this section, we introduce a Gaussian elimination-inspired algorithm to algebraically characterize the smallest solution of a system $I$ of inequalities. 

We now define a notion of substitution.
\subsubsection{Substitution}
\label{subsubsec:algebraic-solution:prelim}
Let $X \in \Omega_v$, and let $\varphi_X \in \mathbb{P}[\Omega']$ 
where $\Omega' = (\Omega_v - \{X\}, \Omega_p)$.

Given an assignment $(\bX, \bP) \in \convset(V)^{\Omega_v} \times (2^V)^{\Omega_p}$, 
the assignment $(\bX,\bP)\left\llbracket X\gets \varphi_X\right\rrbracket \in\convset(V)^{\Omega_v}$ is defined as follows:
\begin{equation}
    (\bX,\bP)\left\llbracket X\gets \varphi_X\right\rrbracket(Y) = \begin{cases}
        \conv{\;\evalmap\left(\varphi_X ;\bX,\bP\right)\;\vphantom{2^{2^2}}}, &\text{ if } Y=X.\\
        \bX(Y),&\text{ otherwise.}
    \end{cases}
\end{equation}

Next, for a polynomial $\varphi \in \mathbb{P}[\Omega]$ and $\varphi_X \in \mathbb{P}[\Omega']$
we first define a polynomial $\varphi\left\llbracket X\gets \varphi_X\right\rrbracket \in \mathbb{P}[\Omega']$ as follows.
(Note that
formally substituting every symbol $X$ in $\varphi$ with the polynomial $\varphi_X$ does not yield a polynomial.) 

For $E=\left(\vphantom{2^{2^2}}\rho\cdot X + (1-\rho)\cdot E'\right) \in\CL(\Omega)$, 
    where $0\leq \rho\leq 1$ and 
    $E'\in\CL(\Omega')$, 
    we define:
    \begin{equation}\label{eqn:element-substitute-def}
        E\left\llbracket X\gets \varphi_X\right\rrbracket \defeq 
        \begin{cases}
        E',\text{ if $\rho = 0$},\\
        \phi_X, \text{ if $\rho = 1$},\\
        \Cplus_{N\in\mono(\varphi_X)} ~ 
            \Cmult_{F\in\elem(N)} ~ 
                \underbrace{\left(\rho\cdot F + (1-\rho)\cdot E'\vphantom{2^{2^2}}\right)}_{E\substitute XF},  \text{ if $0 < \rho < 1$ and $\varphi_X \neq \mathbb{0}$},\\
            \mathbb{0}, \text{ if $0 < \rho < 1$ and $\varphi_X = \mathbb{0}$}.
        \end{cases}
    \end{equation}
Note that the $E \longmapsto E\substitute XF$ maps $\CL(\Omega)$ to 
$\mathbb{P}[\Omega']$.

When $E\in\CL(\Omega')$, this map is the identity map. 
For each monomial $M$ over $\Omega$, we define 
an element of $\mathbb{P}[\Omega']$ as follows.

Let $D(M) = \{E \in \elem(M) \;:\; \text{ coefficient of $X$ is positive in $E$}\}$.

    \begin{equation}\label{eqn:monomial-substitute-def}
        M\left\llbracket X\gets \varphi_X\right\rrbracket \defeq 
            \begin{cases}
            M \text{ if $D(M)  = \emptyset$,}\\
            \mathbb{0} \text{ if $D(M) \neq \emptyset$ and $\varphi_X = \mathbb{0}$}, \\
            \psi \text{ otherwise,}
                \end{cases}
    \end{equation}
where 
\[
\psi = \Cplus_{\vec N \in \mono(\varphi_X)^{D(M)}} ~
              \left( \left( \Cmult_{E\in\elem(M) \setminus D(M)} E\right) ~ \cmult ~
                \Cmult_{E\in D(M)} ~
                    \left(\Cmult_{F\in\elem(\,\vec N(E)\,)} E\left\llbracket X\gets F\right\rrbracket \right)
                \right).
\]

Finally, for each $\varphi \in \mathbb{P}[\Omega]$, we define 
    \begin{equation}\label{eqn:poly-substitute-def}
        \varphi\left\llbracket X\gets \varphi_X\right\rrbracket \defeq 
            \Cplus_{M\in \mono(\varphi)} ~ M\left\llbracket X\gets \varphi_X\right\rrbracket \in \mathbb{P}[\Omega'].         \end{equation}
We will prove the following property of the substituted polynomial.
\begin{lemma}[Substituted Polynomial]
\label{lem:substitute-poly-property}
    Let $\varphi \in \mathbb{P}[\Omega]$, $X\in\Omega_v$, and  $\varphi_X \in \mathbb{P}[\Omega']$, where $\Omega' = (\Omega_v - \{X\},\Omega_p)$.
    
    For all assignments $(\bX,\bP) \in \convset(V)^{\Omega_v} \times (2^V)^{\Omega_p}$ of $\Omega$, the following 
    equivalence holds

        $$ \evalmap\left(\varphi ; (\bX,\bP)\left\llbracket X\gets \varphi_X\right\rrbracket,\bP \right) \sim 
            \evalmap\left(\varphi\left\llbracket X\gets \varphi_X \right\rrbracket ; \bX', \bP\right),$$
            where we denote by $\bX' = \bX\restriction_{\Omega_v - \{X\}}$.

\end{lemma}

\begin{proof}[Proof of Lemma~\ref{lem:substitute-poly-property}]
It suffices to prove the result when $\varphi$ is a monomial.
As a warmup, it is instructive to prove the result for a monomial having degree one.

\paragraph{Warmup.} 
Suppose $\varphi = E= \left(\rho\cdot X + (1-\rho)\cdot E'\right)$, where $E' \in \CL(\Omega')$. 
There are three cases. 
\begin{enumerate}[1.]
\item Case $\rho = 0$. This is immediate.
\item Case $\rho = 1$. In this case 
$$ \evalmap\left(X ; (\bX,\bP)\left\llbracket X\gets \varphi_X\right\rrbracket, \bP\right) =
\conv{\evalmap(\phi_X; \bX',\bP)}
\sim \evalmap(\phi_X; \bX',\bP).
$$
\item Case $0 < \rho < 1$. In this 
case, we use properties of our set operations presented in Lemma~\ref{lem:prop-set-ops} for the following derivation in $\dagger$ and $\ddagger$ steps. 
\allowdisplaybreaks
        \begin{align*}
            \evalmap&\left(\varphi ; (\bX,\bP)\left\llbracket X\gets\varphi_X\right\rrbracket,\bP \right) \\
            &= 
            \evalmap\left(\rho\cdot X + (1-\rho)\cdot E' ; (\bX,\bP)\left\llbracket X\gets\varphi_X\right\rrbracket,\bP \right)
                    \tag{By the definition of $\varphi$}\\
            &=
            \rho\cdot \evalmap\left(X ; (\bX,\bP)\left\llbracket X\gets\varphi_X\right\rrbracket \right) 
                +(1-\rho)\cdot \evalmap\left(E'; (\bX,\bP)\left\llbracket X\gets\varphi_X\right\rrbracket,\bP \right)
                    \tag{By the definition of the evaluation map}\\
            &\sim
            \rho\cdot \evalmap\left(\varphi_X ; \bX',\bP\right) 
                +(1-\rho)\cdot \evalmap\left(E'; (\bX,\bP)\left\llbracket X\gets\varphi_X\right\rrbracket,\bP\right) 
                    \tag{By the definition of $\bX\left\llbracket X\gets\varphi_X\right\rrbracket$ and the fact that $\varphi_X \in \mathbb{P}[\Omega']$}\\ 
            &= \rho\cdot \evalmap\left(\varphi_X ; \bX',\bP\right) +(1-\rho)\cdot \evalmap\left(E';\bX',\bP\right)
                    \tag{Because $E'\in\CL(\Omega')$}\\
            &= \rho\cdot \evalmap\left( \mathop{\cplus}_{M\in\mono(\varphi_X)} \mathop{\cmult}_{F \in \elem(M)} F ; \bX',\bP\right) +(1-\rho)\cdot \evalmap\left(E';\bX',\bP\right)
                    \tag{By the definition of the polynomial $\varphi_X$}\\
            &= 
            \rho\cdot\left( \mathop{\cplus}_{M\in\mono(\varphi_X)} \mathop{\cmult}_{F \in \elem(M)} \evalmap\left(F ; \bX',\bP\right) \right) +(1-\rho)\cdot \evalmap\left(E';\bX',\bP\right)
                    \tag{By the definition of the evaluation map}\\
            &\stackrel\dagger=  \left(\mathop{\cplus}_{M\in\mono(\varphi_X)} \mathop{\cmult}_{F \in \elem(M)} \rho\cdot \evalmap\left(F ; \bX',\bP\right) \right) +(1-\rho)\cdot \evalmap\left( E' ; \bX',\bP\right) 
                    \tag{Because scalar multiplication distributes over $\cplus$ and $\cmult$ (Lemma~\ref{lem:prop-set-ops},\eqref{eqn:dist-scal-cplus-general} and \eqref{eqn:dist-scal-cmult-general})}\\
            &\stackrel\ddagger\sim \mathop{\cplus}_{M\in\mono(\varphi_X)} \mathop{\cmult}_{F \in \elem(M)} \left( \; \rho\cdot \evalmap\left(F ; \bX',\bP\right) +(1-\rho)\cdot \evalmap\left( E' ; \bX',\bP\right) \; \vphantom{2^{2^2}}\right)
                    \tag{Because Minkowski sum distributes over $\cplus$ and $\cmult$ (Lemma~\ref{lem:prop-set-ops}  \eqref{eqn:dist-minkowski-cmult-general})}\\
            &= \Cplus_{M\in\mono(\varphi_X)} \Cmult_{F \in \elem(M)} \evalmap\left(\;  E\llbracket X\gets F\rrbracket ; \bX',\bP \;\right)
                    \tag{By the definition of $E\llbracket X\gets F\rrbracket$}\\
            &= \evalmap\left(\mathop{\cplus}_{M\in\mono(\varphi_X)} \mathop{\cmult}_{F \in \elem(M)} E\llbracket X\gets F\rrbracket ; \bX',\bP\right)
                    \tag{By the definition of the evaluation map}\\
            &= \evalmap\left( \varphi\left\llbracket X \gets \varphi_X\right\rrbracket ; \bX',\bP \right).
                    \tag{By the definition of the polynomial $\varphi\left\llbracket X \gets \varphi_X\right\rrbracket$ in $\mathbb{P}[\Omega']$}
        \end{align*}
\end{enumerate}
        This completes the proof of the warmup case.

\paragraph{Primary case: $\varphi= M$ is a monomial.}
        The full proof is similar to the warmup proof. \allowdisplaybreaks
There are three cases. 
\begin{enumerate}[1.]
\item Case $D(M) = \emptyset$.
In this case 
\begin{align*}
\evalmap\left(M ; (\bX,\bP)\llbracket X\gets \varphi_X\rrbracket,\bP\right)
&= \evalmap\left(M ; \bX',\bP\right) \\
&= \evalmap\left(M\llbracket X\gets \varphi_X\rrbracket;\bX',\bP\right)
\end{align*}
using \eqref{eqn:monomial-substitute-def} for the second equality.

\item Case $D(M) \neq \emptyset$, $\varphi_X = \mathbb{0}$.
In this case, 
\begin{align*}
\evalmap\left(M ; (\bX,\bP)\llbracket X\gets \varphi_X\rrbracket,\bP\right)
&= \emptyset \\
&= \evalmap\left(\mathbb{0} ; \bX',\bP\right) \\
&= \evalmap\left(M\llbracket X\gets \varphi_X\rrbracket;\bX',\bP\right)
\end{align*}
using \eqref{eqn:monomial-substitute-def} for the last equality.

\item Case $D(M) \neq \emptyset$, $\varphi_X \neq \mathbb{0}$. In this case

\end{enumerate}
        \begin{align*}
            \evalmap&\left(M ; (\bX,\bP)\llbracket X\gets \varphi_X\rrbracket,\bP\right) \\
            &= 
            \evalmap\left(\left(\Cmult_{E\in \elem(M)\setminus D(M)} E \right) ~\cmult~ \left(\Cmult_{E\in D(M)} E ; (\bX,\bP)\llbracket X\gets \varphi_X\rrbracket,\bP \right) \right)
                    \tag{By the definition of $M$}\\
            &= 
            \left(\Cmult_{E\in \elem(M) \setminus D(M)}
            \evalmap\left( E ; \bX',\bP\right) \right)~ \cmult ~
            \left(\Cmult_{E\in D(M)} \evalmap\left( E ; (\bX,\bP)\llbracket X\gets \varphi_X\rrbracket, \bP \right)\right)
                    \tag{By the definition of the evaluation map}\\
            &\sim 
           \left( \Cmult_{E\in \elem(M) \setminus D(M)}
            \evalmap\left( E ; \bX',\bP\right)\right) ~ \cmult ~
            \Cmult_{E\in D(M)} \left( \Cplus_{M'\in\mono(\varphi_X)} \Cmult_{F\in\elem(M')} \evalmap\left( \; E\llbracket X\gets F\rrbracket ; \bX',\bP\right)  \;\right)
                    \tag{By the derivation in the warmup case to the step after the $\ddagger$ step}\\
                                &= \Cplus_{\vec N\in \mono(\varphi_X)^{D(M)}} \left( \Cmult_{E\in\elem(M) \setminus D(M)} \evalmap (E; \bX',\bP)\right) ~ \cmult ~
                                \Cmult_{E\in D(M)} \left( \; 
                                \Cmult_{F\in\elem(\vec N(E))}   \evalmap\left( \; E\llbracket X\gets F\rrbracket ; \bX',\bP\right) \; \right)
                    \tag{Because $\cmult$ distributes over $\cplus$ (Lemma~\ref{lem:prop-set-ops} \eqref{eqn:dist-cmult-cplus}}\\
            &\sim \evalmap\left( \; \Cplus_{\vec N\in \mono(\varphi_X)^{D(M)}} 
            \left( \left( \Cmult_{E\in\elem(M) \setminus D(M)} E\right) ~ \cmult ~
            \Cmult_{E\in D(M)} \left( \; \Cmult_{F\in\elem(\vec N(E))}    E\llbracket X\gets F\rrbracket\right) ; \bX',\bP \; \right)
        \right)
                    \tag{By the definition of the evaluation map}\\
            &\sim
            \evalmap\left( M\left\llbracket X \gets \varphi_X\right\rrbracket ; \bX',\bP \right).
                    \tag{By the definition of the polynomial $M\left\llbracket X \gets \varphi_X\right\rrbracket 
                    \in \mathbb{P}[\Omega']$}
        \end{align*}
        This completes the proof of \lemmaref{substitute-poly-property}.
\end{proof}

\subsubsection{Gaussian Elimination Algorithm}
\begin{figure}[htp]
    \centering
    \noindent\begin{boxedminipage}{\linewidth}
    \begin{enumerate}[1.] 
        \item Initialize $I\p 0= I$ 
        \item For $j\in \{1,2,\dotsc,n\}$: 
        \begin{enumerate}[(a)] 
            \item Suppose $I\p{j-1}$ is the system $ \left\{ X_i \geq \varphi_i\p{j-1}\right\}_{i=1}^n$, each $\varphi_i\p{j-1}$ is a polynomial over $(\{X_j, \dotsc, X_n\}, \{P_1, \dotsc, P_m\})$
            \item {\bfseries Canceling $X_j$ step.} 
                Use the rearrangement lemma (Lemma~\ref{lem:rearrangement}) and cancellation lemma (Lemma~\ref{lem:cancel}) to obtain a polynomial $\widetilde\varphi$ over $(\Omega_v\setminus\{X_1,X_2,\dotsc,X_j\}, \Omega_p)$. 
                Define the new system $I'$ identical to $I\p{j-1}$ except that the inequality $X_j\geq \varphi_j\p{j-1}$ is replaced by $X_j\geq \widetilde\varphi$. 
            \item {\bfseries Substituting $X_j$ step.} 
                Define the new system $I\p j$ as the system $\left\{ X_i \geq \varphi\p j_i\right\}_{i=1}^n$, where
                    $$\varphi_i\p j \defeq \begin{cases}
                        \widetilde\varphi, &\text{ if } i=j\\ 
                        \\ 
                        \varphi_i\p{j-1}\left\llbracket X_j\gets \widetilde\varphi\right\rrbracket,&\text{ otherwise.}
                    \end{cases}$$
        \end{enumerate}
        \item {\bfseries Characterizing the smallest solution for a constant assignment.} 
            For a constant assignment $\bP$, output $\bX\in\convset(V)^n$, where $\bX_i=\conv{ \evalmap\left(\varphi_i\p n; \bP\right) }$ for $i\in\{1,2,\dotsc,n\}$. 
    \end{enumerate}
    \end{boxedminipage}
    \caption{Our Gaussian elimination-inspired algorithm to solve the system of inequalities $I$.}
    \label{fig:gaussian-elim}
\end{figure}

\begin{theorem}
    \label{thm:gaussian-elim}
    Let $\Omega_v = \{X_1,\dotsc,X_n\}$ and 
     $\Omega_p = \{P_1,\dotsc,P_m\}$.
    Let $I$ be the system of inequalities $(X_i\geq \varphi_i)_{i=1}^n$, where $\varphi_1,\dotsc,\varphi_n \in \mathbb{P}[\Omega]$. Then
    \figureref{gaussian-elim} presents a finite procedure to compute $\varphi^*_1,\dotsc,\varphi^*_n \in \mathbb{P}[(\emptyset,\Omega_p)]$ 
    with the guarantee that
    $\ssol(I;\bP)(X_j) = \conv{ \;\evalmap(\varphi^*_j;\bP)\; }$ for every $j\in\{1,2,\dotsc,n\}$ and assignment $\bP \in (2^V)^{\Omega_p}$.   
\end{theorem}

\begin{remark}[Connection to Knaster--Tarski theorem]
\label{rem:Knaster-Tarski}
As mentioned previously, for a fixed assignment \(\bP\) of the parameters, the existence
of the smallest solution can also be viewed as an instance of the
Knaster--Tarski fixed-point theorem.

Indeed, let
\[
\mathcal{L}:=\convset(V)^{\Omega_v},
\]
ordered component-wise by inclusion. This is a complete lattice: arbitrary
meets are given by component-wise intersections, while arbitrary joins are
given by
\[
\bigvee_{\alpha}\mathbf X^{(\alpha)}
=
\left(
\conv{
\bigcup_{\alpha}X_i^{(\alpha)}
}
\right)_{1\leq i\leq n}.
\]
For a fixed parameter assignment \(\bP\), define
\[
\mathcal F_{\mathbf P}\colon\mathcal{L}\longrightarrow\mathcal{L}
\]
by
\[
\mathcal F_{\mathbf P}(\mathbf X)_i
:=
\operatorname{conv}
\left(
X_i
\cup
\operatorname{eval}
\bigl(
\varphi_i;\mathbf X,\mathbf P
\bigr)
\right),
\qquad 1\leq i\leq n.
\]
By the monotonicity of evaluation, \(\mathcal F_{\mathbf P}\) is
monotone. Moreover,
\[
\mathcal F_{\mathbf P}(\mathbf X)=\mathbf X
\]
if and only if \(\mathbf X\) is a solution of \(I\) under
\(\mathbf P\). Thus the Knaster--Tarski theorem implies that
\(\mathcal F_{\mathbf P}\) has a least fixed point, which is precisely
\[
\operatorname{ss}(I;\mathbf P).
\]

Later, in Lemma~\ref{lem:itr-char},
we will also give an iterative construction,
which is a concrete Kleene-type realization of
this least fixed point, which will be used in the proof of Theorem~\ref{thm:laminar}.

Note that we can package all parameter assignments pointwise into a
single function lattice and thereby obtain the full map
\[
\bP \mapsto \ssol(I;\bP)
\]
as one least fixed point. 
Knaster–Tarski alone, however, does not produce this map by a finite expression in the parameter symbols. 
Theorem~\ref{thm:gaussian-elim} is stronger in precisely this syntactic
sense: before any parameter assignment is chosen,
it produces polynomials
\[
\varphi_1^*,\ldots,\varphi_n^*
\in
P[(\emptyset,\Omega_p)],
\]
depending only on the original system \(I\), such that
\[
\operatorname{ss}(I;\mathbf P)(X_i)
=
\operatorname{conv}
\left(
\operatorname{eval}
\bigl(
\varphi_i^*;\mathbf P
\bigr)
\right)
\]
for every \(1\leq i\leq n\) and simultaneously for every assignment
\[
\mathbf P\in (2^V)^{\Omega_p}.
\]

Thus the unknown convex sets are eliminated once and for all, and the
result is a finite parameter-only expression valid under every
interpretation of the parameter symbols. 
When the parameter sets belong to a class with the effective good
closure property (Definition~\ref{def:property:closure}), these uniform expressions yield effective descriptions
of the coordinate sets of the smallest solution. In the semi-algebraic
and semi-linear specializations, these descriptions can be converted
into quantifier-free formulas.
\end{remark}

\begin{remark}[Prior and related works]
\label{rem:prior}
There have been some prior and related works dealing with
finite elimination in Kleene algebras, as well as on algebraic structures on convex or polyhedral sets, that we discuss below.

The most closely related work coming from the area of formal language theory is that
of Hopkins and Kozen \cite{HK99} who prove that a finite system of polynomial inequalities
$\left(f_i(x_1,\ldots,x_n) \leq x_i\right)_{1 \leq i \leq n}$
over a commutative Kleene algebra has a unique least solution and that each coordinate of this
solution is represented by a finite expression in the coefficients \cite[Theorem 1.1, p. 394]{HK99}.  The ambient algebra that we consider in this paper is different and most importantly unlike that in \cite{HK99} does not already include a Kleene star operator. 

In the more geometric setting of algebras of convex sets, 
Iwano and Steiglitz \cite{IS90}  define a closed semiring of convex polygons with Minkowski sum as multiplication
and convex hull of union as addition, and they use a Kleene-closure construction algorithmically --  but their emphasis is graph-theoretic, and not aligned with the problem considered in this paper.
Chen \cite{Chen99}
studies Minkowski algebra, which includes both closed convex sets and relatively open convex sets from the Euler-characteristic calculus viewpoint, and
Manjunath \cite{Manjunath2024} studies equations over polyhedral semirings
with convex-hull addition and Minkowski multiplication, including local solutions and a local–global principle. UL

Nevertheless, to the best of our knowledge, no prior result gives a finite parameter-only representation
of least solutions for recursive convex-set containments in the calculus of convex
linear combinations, ordinary union, and positive geometric join,
considered in the present paper,
without adjoining a fixed-point or Kleene-star operator.  
\end{remark}

\subsubsection{Preparation for the proof of Theorem~\ref{thm:gaussian-elim}}
\begin{notation}
We extend the evaluation map to Boolean predicates of the following form: 
    $\evalmap(X\geq \varphi ;\bX,\bP)$ is true if (and only if) $\evalmap(X;\bX,\bP) \geq \evalmap(\varphi;\bX,\bP)$, where $X$ is an unknown and $\varphi$ is a polynomial over $\CL(\Omega)$.
\end{notation}

To prove \theoremref{gaussian-elim}, we will need the following preliminary results. 

\paragraph{Overview of the preliminary results.}
\begin{enumerate}[1.] 
    \item Rearrangement lemma (Lemma~\ref{lem:rearrangement}): 
        Given an inequality $X\geq \varphi$, this lemma rewrites it as an ``equivalent'' inequality with a very specific structure: 
            $$X\geq\varphi' \;\cplus\; X\cmult M_1 \;\cplus\; \dotsi \;\cplus\; X\cmult M_k, $$
        where $\varphi'$ is a polynomial and $M_1, \dotsc, M_k$ are monomials over $\Omega'$, where as before $\Omega' = (\Omega_v \setminus \{X\}, \Omega_p)$ 
        Here, two inequalities are considered equivalent when both are simultaneously true, or both are simultaneously false for all assignments. 

    \item Cancellation lemma (Lemma~\ref{lem:cancel}): 
        Consider a system where the inequality for $X$ has the structure promised by the rearrangement lemma above. 
        The Cancellation lemma produces a polynomial $\widetilde{\varphi} \in  \mathbb{P}[\Omega']$ such that replacing the structured inequality with $X\geq \widetilde\varphi$ preserves the smallest solution for all assignments. 
        Together with the rearrangement lemma above, the cancellation lemma eliminates $X$ from the RHS of the inequality for the unknown $X$. 

    \item Substitution lemma (Lemma~\ref{lem:substitute}): 
        Consider a system with inequalities $X\geq \varphi_X$ and $Y\geq\varphi_Y$, where $\varphi_X \in \mathbb{P}[\Omega']$. 
        Our objective is to construct a new system where $Y\geq \varphi_X$ is replaced by the inequality $Y\geq \varphi_Y\left\llbracket X\gets\varphi_X\right\rrbracket$. 
        The substitution lemma will prove that the new system's smallest solution is identical to the smallest solution of the original system. 
        We can iteratively use this lemma for all unknowns $Y\in \Omega_v \setminus\{X\}$ to remove the dependence on the unknown $X$ from every polynomial in the system. 
\end{enumerate}

\paragraph{\emph{Rearrangement and Cancellation Lemmas.}}

\begin{lemma}[Rearrangement Lemma]\label{lem:rearrangement}
    For each $X\in\Omega_v$, and $\varphi \in \mathbb{P}[\Omega]$, there  exists $\varphi' \in \mathbb{P}[\Omega']$ and monomials $M_1, M_2, \dotsc, M_k$ over $\Omega'$, where $k\geq 0$, such that (for any assignment $(\bX,\bP)$) the following identity holds. 
        $$\evalmap(X\geq \varphi ; \bX,\bP) = \evalmap\left(X\geq\varphi' \;\cplus\; X\cmult M_1 \;\cplus\; \dotsi \;\cplus\; X\cmult M_k ;\bX,\bP\right).$$
        (We follow the same notation as before and
denote $\Omega' = (\Omega_v - \{X\},\Omega_p)$).

\end{lemma}
\begin{proof}
    If $\varphi=\mathbb{0}$, then $\varphi'=\mathbb{0}$ and $k=0$. 
    We can also assume without loss of generality that the monomial 
    $X \not\in\mono(\varphi)$.
    If $X \in \mono(\varphi)$, we can replace $\varphi$ by another polynomial
    obtained from $\varphi$ by dropping the monomial $X$. This does not change the set $\evalmap(X\geq \varphi; \bX,\bP)$.

    Otherwise, suppose $\varphi = N_1 \;\cplus\; \dotsi \;\cplus\; N_\ell$ and $\ell\geq 1$. 
    We say that a monomial $M=E_1\cmult E_2\cmult \dotsi \cmult E_u$ {\em depends on $X$} if there is $i\in\{1,2,\dotsc, u\}$ such that $E_i = \rho\cdot X + (1-\rho)\cdot E'$, where $0<\rho\leq 1$ and $E'\in\CL(\Omega')$.
    If the polynomial $\varphi$ has no monomial depending on $X$, then $\varphi'=\varphi$ and $k=0$. 

    Otherwise, $I\subset \{1,2,\dotsc,\ell\}$ be the subset of indices $i$ such that the monomial $N_i$ does not depend on $X$.  
    The complement $J =\{1,2,\dotsc,\ell\}\setminus I$ is the subset of indices $i$ such that the monomial $N_i$ depends on $X$. 
    Without loss of generality, let $J = \{1,2,\dotsc,k\}$, where $k\geq 1$, and $I =\{k+1,\dotsc,\ell\}$. 
    For index $i\in J$, let
    $$N_i = \left(\rho_1\cdot X + (1-\rho_1)\cdot E_1\right) 
                    \cmult \dotsi \cmult
                \left(\rho_{v_i}\cdot X + (1-\rho_{v_i})\cdot E_{v_i}\right) \cmult E_{v_i+1} \cmult \dotsi \cmult E_{u_i},$$ 
    such that $1\leq v_i\leq u_i$, $E_1, \dotsc, E_{u_i} \in \CL(\Omega')$, and $\rho_1, \rho_2,\dotsc, \rho_{v_i}\in(0,1]$.

    Moreover, suppose without loss of generality that 
    $\rho_1 = \cdots = \rho_{w_i} =1$ for some $0 \leq w_i\leq v_i$,
    and $\rho_{w_i +1}, \ldots,\rho_{v_i} < 1$.
    
    Define the following monomial over $\CL(\Omega')$. 
        $$ M_i \defeq E_{w_i+1} \cmult \dotsi \cmult E_{v_i} \cmult E_{v_i+1}\cmult \dotsi \cmult E_{u_i}.$$
    Define the following polynomial over $\CL(\Omega')$. 
        $$\varphi' \defeq N_{k+1} \;\cplus\; \dotsi \;\cplus\; N_\ell.$$
    
    Now, for any assignment $\bX$ and $\bP$, we have the following argument.
    \begin{align*}
        \evalmap\left(X \geq \varphi ; \bX,\bP \right) &= 
            \evalmap\left(X \geq N_1 \;\cplus\; \dotsi \;\cplus\; N_\ell ; \bX,\bP\right)\\
            &= \mathop{\bigwedge}_{i=1}^\ell \evalmap\left(X \geq N_i ; \bX,\bP\right) \\
            &= \left(\mathop{\bigwedge}_{1\leq i\leq k} \evalmap\left(X \geq N_i ; \bX,\bP\right)\right) 
                \wedge \left(\mathop{\bigwedge}_{k+1\leq i\leq \ell} \evalmap\left(X \geq N_i ; \bX,\bP\right)\right)\\
            &= \evalmap\left(X \geq \varphi' ; \bX,\bP\right) \wedge \left(\mathop{\bigwedge}_{1\leq i\leq k} \evalmap\left(X \geq N_i ; \bX,\bP\right)\right)\\
            &\stackrel{\dagger}{=} \evalmap\left(X \geq \varphi' ; \bX,\bP\right) \wedge \left(\mathop{\bigwedge}_{1\leq i\leq k} \evalmap\left(X \geq X\cmult M_i ; \bX,\bP\right)\right)\\
            &= \evalmap\left(X \geq \varphi' \;\cplus\; X\cmult M_1 \;\cplus\; \dotsi \;\cplus\; X\cmult M_k ; \bX,\bP\right)
    \end{align*}
    The explanation of ($\dagger$) is that $\evalmap\left(X \geq N_i ; \bX,\bP\right) = \evalmap\left(X \geq X\cmult M_i ; \bX,\bP\right) $ by using Lemma~\ref{lem:pullout-X} on 
    $E_{w_i+1}, \dotsc, E_{v_i}$
    and, finally, using the idempotence $X=X\cmult X$ (when $X\in \convset$) from \eqref{eqn:idempotent-cmult}. 
\end{proof}

\begin{lemma}[Cancellation Lemma]\label{lem:cancel}
    Consider a system $I$ with an inequality
        $$ X \geq \left(\mathop{\cplus}_{i=1}^{k'} M'_i\right) \;\cplus\; \left(\mathop{\cplus}_{j=1}^k X\cmult M_j\right),$$
    where $M_1, \dotsc, M_k,M'_1, \dotsc, M'_{k'}$ are monomials over $\Omega'$ where $\Omega' = (\Omega_v - \{X\},\Omega_p)$. 
    Let $I'$ be a new system identical to $I$ except that the inequality above is replaced by 
        $$ X \geq \mathop{\cplus}_{i=1}^{k'} \left( \; M_i' \;\cplus\; \left( \mathop{\cplus}_{j=1}^k M'_i\cmult M_j\right) \;\right)$$
    Then, $\ssol(I;\bP)=\ssol(I';\bP)$ for all constant assignments $\bP$. 
\end{lemma}
\begin{proof}
    Our proof will have two components. 
    For arbitrary constant assignments $\bP$, we have: 
    \begin{enumerate}
        \item $\sol(I;\bP) \subseteq \sol(I';\bP)$. 
        \item $\ssol(I';\bP) \in \sol(I;\bP)$.
    \end{enumerate}
    These two results imply that $\ssol(I;\bP)=\ssol(I';\bP)$.

    \paragraph{Part 1.}
    For this part, it suffices to prove that
        $$ \evalmap\left( X \geq \left(\mathop{\cplus}_{i=1}^{k'} M'_i\right) \;\cplus\; \left(\mathop{\cplus}_{j=1}^k X\cmult M_j\right) ; \bX,\bP\right)$$
    implies $\evalmap\left(X \geq M'_i ; \bX, \bP\right)$ and $\evalmap\left(X \geq M'_i\cmult M_j ; \bX, \bP\right)$, for all $i\in\{1,2,\dotsc,k'\}$ and $j\in\{1,2,\dotsc,k\}$.
    Note that the implication $\evalmap\left(X \geq M'_i ; \bX, \bP\right)$ is obvious. 
    Next, observe that we also have the implication $\evalmap\left(X \geq X\cmult M_j ; \bX, \bP\right)$, which (in turn) implies $\evalmap\left(X \geq M'_i\cmult M_j ; \bX, \bP\right)$. 
    This concludes the proof of the first part. 

    \paragraph{Part 2.}
    Define 
        $$A_\bP \defeq \evalmap\left( \; \mathop{\cplus}_{i=1}^{k'} \left( \; M_i' \;\cplus\; \left( \mathop{\cplus}_{j=1}^k M'_i\cmult M_j\right) \;\right); \ssol(I';\bP)\restriction_{\Omega_v \setminus \{X\}},\bP\right).$$
    Here, we are using the fact that $M'_i$ and $M_j$ are monomials over $\Omega'$.
    Note that $\ssol(I';\bP)(X) = \conv{ A_\bP} $; otherwise, replacing $\ssol(I';\bP)(X)$ by  $\conv{A_\bP}$ (and leaving the other unknown assignments identical) creates a smaller solution in $\sol(I';\bP)$.

    After this, to prove $\ssol(I';\bP)\in\sol(I;\bP)$, it suffices to prove that
        $$\evalmap\left( \; X \geq \left(\mathop{\cplus}_{i=1}^{k'} M'_i\right) \;\cplus\; \left(\mathop{\cplus}_{j=1}^k X\cmult M_j\right) ; \ssol(I';\bP), \bP\right)\text{ is true.}$$
    It is equivalent to proving
        $$\evalmap\left( \; \conv{A_\bP} \geq \left(\mathop{\cplus}_{i=1}^{k'} M'_i\right) \;\cplus\; \left(\mathop{\cplus}_{j=1}^k \conv{A_\bP}\cmult M_j\right) ; \ssol(I';\bP)\restriction_{\Omega_v \setminus \{X\}}, \bP\right)\text{ is true.} $$

    For brevity, let us introduce some notation. 
    Define
        \begin{align*}
            U_\bP &\defeq \evalmap\left( \; \mathop{\cplus}_{i=1}^{k'} M'_i ; \ssol(I';\bP)\restriction_{\Omega_v \setminus \{X\}}, \bP \right) \\
            V_\bP &\defeq \evalmap\left( \; \mathop{\cplus}_{j=1}^{k} M_j ; \ssol(I';\bP)\restriction_{\Omega_v \setminus \{X\}}, \bP \right).
        \end{align*}
    Note that $\conv{A_\bP} =\conv{U_\bP \;\cplus\; U_\bP\cmult V_\bP}$.
    Using this new notation, we need to prove that
        \begin{align*} 
            && \conv{U_\bP \;\cplus\; U_\bP\cmult V_\bP} &\geq U_\bP \;\cplus\; \conv{U_\bP \cplus U_\bP\cmult V_\bP}\cmult V_\bP\\
            \iff && \conv{U_\bP \;\cplus\; U_\bP\cmult V_\bP} &\geq U_\bP \;\cplus\; \left({U_\bP \;\cplus\; U_\bP\cmult V_\bP}\right)\cmult V_\bP\\
            \iff && \conv{U_\bP \;\cplus\; U_\bP\cmult V_\bP} &\geq U_\bP \;\cplus\; U_\bP\cmult V_\bP  \;\cplus\; U_\bP\cmult V_\bP\cmult V_\bP \\
            \iff && \conv{U_\bP \;\cplus\; U_\bP\cmult V_\bP} &\geq U_\bP \;\cplus\; U_\bP\cmult V_\bP  \;\cplus\; U_\bP\cmult \conv{V_\bP}
                \tag{By \eqref{eqn:idempotent-cmult}}\\ 
            \iff && \conv{U_\bP \;\cplus\; U_\bP\cmult V_\bP} &\geq U_\bP \;\cplus\; U_\bP\cmult V_\bP  \;\cplus\; U_\bP\cmult {V_\bP} \\
            \iff && \conv{U_\bP \;\cplus\; U_\bP\cmult V_\bP} &\geq U_\bP \;\cplus\; U_\bP\cmult V_\bP,  
        \end{align*}
    which is trivially true, completing the proof of part 2.
\end{proof}

\paragraph{\emph{Substitution Lemma.}}

\begin{lemma}[Substitution Lemma]\label{lem:substitute}
Consider a system $I'$ containing two inequalities $X\geq \varphi_X$ and $Y\geq \varphi_Y$, where $\varphi_X \in \mathbb{P}[\Omega']$ and $\varphi_Y \in \mathbb{P}[\Omega]$. 
Define a new system $I''$ identical to $I'$ except that the inequality 
$Y\geq \varphi_Y$ is replaced by $Y\geq \varphi_Y\left\llbracket X\gets\varphi_X \right\rrbracket$. 
Then, $\ssol(I';\bP)=\ssol(I'';\bP)$ for all assignments $\bP \in (2^V)^{\Omega_p}$. 
\end{lemma}
\begin{proof}
    Our proof will have two components. 
    For arbitrary constant assignments $\bP$, we have:
    \begin{enumerate}
        \item $\sol(I';\bP) \subseteq \sol(I'';\bP)$. 
        \item $\ssol(I'';\bP) \in \sol(I';\bP)$.
    \end{enumerate}
    These two results imply that $\ssol(I';\bP)=\ssol(I'';\bP)$. 

    \paragraph{Part 1.}
    For this part, it suffices to prove that 
    \[
    \evalmap\left(X\geq \varphi_X; \bX,\bP\right)
    \]
    and $\evalmap\left(Y\geq \varphi_Y; \bX,\bP\right)$ implies 
    \[
    \evalmap\left(Y\geq \varphi_Y\left\llbracket X\gets\varphi_X \right\rrbracket; \bX,\bP\right)
    \]
    for every $\bX \in \convset(V)^{\Omega_v}$.
    
    Note that (read the derivation left to right).
        $$\evalmap\left(\varphi_Y\left\llbracket X\gets\varphi_X \right\rrbracket; \bX,\bP\right) 
                \stackrel{*}{\sim} \evalmap\left(\varphi_Y ; (\bX,\bP)\left\llbracket X\gets \varphi_X\right\rrbracket \right) 
                \stackrel\dagger\leq \evalmap\left(\varphi_Y ; \bX,\bP\right) 
                \stackrel\ddagger\leq \evalmap\left(Y ; \bX,\bP\right),$$
    which completes the proof. 
    The explanations for the derivation steps are below.
    \begin{enumerate}
        \item Step $*$ is true by the definition of substituted polynomial, see \lemmaref{substitute-poly-property}
        \item Step $\dagger$ holds because $\evalmap\left(X\geq \varphi_X; \bX,\bP\right)$
        \item Step $\ddagger$ holds because $\evalmap\left(Y\geq \varphi_Y; \bX,\bP\right)$
    \end{enumerate}

    \paragraph{Part 2.}
    It will suffice to prove that 
        $\evalmap\left( Y\geq \varphi_Y ; \ssol(I'';\bP), \bP \right)$. 

    We first claim that $\ssol(I'';\bP)_X \sim \evalmap\left( \varphi_X ; \ssol(I'';\bP), \bP \right)$; otherwise, we will find a smaller solution of $I''$, which is a contradiction. 
    Suppose not; \ie, $\ssol(I'';\bP)(X) \in\convset(V)^{\Omega_v}$ is a proper superset of $A_\bP\defeq \conv{\evalmap(\varphi_X;\ssol(I'';\bP);\bP)}$. 
    Recall that $\varphi_X \in\mathbb{P}[\Omega']$. 
    Thus, replacing $\ssol(I'';\bP)(X)$ by $A_\bP$ in the smallest solution creates a smaller solution. 

    As a result of the claim, for any $\varphi \in \mathbb{P}[\Omega]$, we have
    \[\evalmap\left(\varphi ; \ssol(I'';\bP),\bP\right) \sim \evalmap\left(\varphi ; (\ssol(I'';\bP),\bP)\left\llbracket X\gets\varphi_X\right\rrbracket\right).
    \]
    In particular, 
        $$\evalmap\left(\varphi_Y ; (\ssol(I'';\bP),\bP)\left\llbracket X\gets\varphi_X\right\rrbracket \right) \sim \evalmap\left(\varphi_Y ; \ssol(I'';\bP),\bP\right).$$
    By the definition of the polynomial $\varphi_Y\left\llbracket X\gets\varphi_X \right\rrbracket \in \mathbb{P}[\Omega']$ (see \lemmaref{substitute-poly-property}), we have 
        $$\evalmap\left(\varphi_Y ; (\ssol(I'';\bP),\bP)\left\llbracket X\gets\varphi_X\right\rrbracket \right) 
            \sim \evalmap\left(\varphi_Y\left\llbracket X\gets\varphi_X \right\rrbracket ; \ssol(I'';\bP),\bP\right).$$
    Consequently, we have 
        $$ \evalmap\left(\varphi_Y ; \ssol(I'';\bP),\bP\right) \sim \evalmap\left(\varphi_Y\left\llbracket X\gets\varphi_X \right\rrbracket ; \ssol(I'';\bP),\bP\right).$$
    So evaluations of $\varphi_Y$ and $\varphi_Y\left\llbracket X\gets\varphi_X \right\rrbracket $ have the same convex hull. 

    Recall that $\ssol(I'';\bP)\in\convset(V)^{\Omega_v}$ is a solution of $I''$ and (as a result)  
    \[
    \evalmap\left(Y\geq \varphi_Y\left\llbracket X\gets\varphi_X \right\rrbracket ; \ssol(I'';\bP),\bP\right)
    \]
    holds. 
    Therefore, 
    \[
    \evalmap\left(Y\geq \varphi_Y ; \ssol(I'';\bP),\bP\right)
    \]
    also holds, because $\varphi_Y$ and $\varphi_Y\left\llbracket X\gets\varphi_X \right\rrbracket $ have the same convex hull. 
    This completes the proof of part 2. 
\end{proof}

\begin{remark}
    The procedure above eliminates unknowns $X_1, X_2, \dotsc, X_n$ from the polynomials, one at a time. 
    Changing the elimination order may change the description of the smallest solution. 
\end{remark}

\begin{remark}
    The transformation steps above may result in $\widetilde\varphi=\mathbb{0}$ inside the loop, which can lead to zero polynomials on the RHS of the last system $I\p n$. 
    This occurrence depends on the structure of the initial system $I$, not on the specific constant assignment $\bP$.
\end{remark}

\begin{proof}[Proof of \theoremref{gaussian-elim}]
For $0 \leq j \leq n$, we denote by $\Omega_v^{(j)} = \{X_{j+1},\ldots,X_n\} \subset \Omega_v$, and we denote by $\Omega^{(j)} = (\Omega_v^{(j)},\Omega_p)$.

Beginning with the system $I\p 0 = I$, the algorithm inductively constructs new systems of inequalities $I\p j$,
$\left(X_i \geq \varphi_i^{(j)}\right)_{1 \leq i \leq n}$, with 
$\phi_i^{(j)} \in \mathbb{P}[\Omega^{(j)}]$, maintaining the property that
\[
\ssol\left(I\p 0;\bP\right) =\ssol\left(I\p j;\bP\right)
\]
for any assignment $\bP \in (2^V)^{\Omega_p}$. 
Note that it is possible that 
$\sol\left(I\p 0;\bP\right) \neq \sol\left(I\p j;\bP\right)$.
The system $I\p n$ is $\left\{X_i\geq \varphi_i^*\right\}_{i=1}^n$ and every $\varphi^*_i$ is a polynomial over $(\emptyset,\Omega_p)$. 
It follows that 
\[
\ssol\left(I\p n ; \bP\right)(X_i) = 
\conv{ \, \evalmap(\varphi^{(n)}_i ; \bP) \, }\text{ for every }i\in\{1,2,\dotsc,n\}.
\]

Consider the inner loop $j\in\{1,2,\dotsc,n\}$. 
Note that the system $I\p{j-1}$ will have polynomials 
belonging to $\mathbb{P}[\Omega^{(j-1)}]$.

We consider the $j$-th inequality in this system: $ X_j \geq \varphi\p{j-1}_j$. 
Lemma~\ref{lem:rearrangement}  and Lemma~\ref{lem:cancel} give an explicit polynomial $\widetilde\varphi \in \mathbb{P}[\Omega^{(j)}]$ 
with the property that 
    replacing the inequality $ X_j \geq \varphi\p{j-1}_j$ with the inequality $X_j\geq \widetilde\varphi$ preserves the smallest solution. 
Let $I'$ represent this new system. 
Next, in the system $I'$, the algorithm substitutes every instance of $X_j$ with the polynomial $\widetilde\varphi$ in the polynomials 
    $$\left\{ \; \varphi_\ell\p{j-1} \;\colon\; \ell\in\{1,\dotsc, j-1,j+1,\dotsc, n\} \; \right\}$$
These are the polynomials $\varphi_\ell\p{j-1}\left\llbracket X_j\gets \widetilde\varphi \right\rrbracket$ defined according to \eqref{eqn:poly-substitute-def}. 
The substitution lemma (Lemma~\ref{lem:substitute}) proves that these substitutions preserve the smallest solution for any assignment $\bP \in (2^V)^{\Omega_p}$.

Note that $\widetilde\varphi$ and the $\varphi_\ell\p{j-1}  \left\llbracket X_j\gets \widetilde\varphi \right\rrbracket$ 
belong to $\mathbb{P}[\Omega^{(j)}]$.
Therefore, at the end of the $j$-th loop, the unknowns $X_1, \dotsc, X_j$ are eliminated from the RHS of every inequality. 
After the $n$-th iteration of the loop, the polynomials in our system will
belong to $\mathbb{P}[\Omega^{(n)}]$. Notice that
$\Omega^{(n)} = (\emptyset,\Omega_p)$.
The smallest solution of the system $(X_j \geq \varphi\p{n}_j)_{1 \leq j \leq n}$ is then given by
$\bX_j = \conv{\evalmap(\varphi\p{n}_j;\bP)}, 1 \leq j \leq n$.
We let $\varphi_j^* = \varphi\p{n}_j$, and clearly 
$\varphi^*_1,\ldots,\varphi_n^*$ satisfies the conditions of the theorem.
\end{proof}

\begin{proposition}
\label{prop:closure-under-ops}
If $\mathcal{B}$ is one of the families in Example~\ref{eg:closure}, then $\mathcal{B}$ has the good closure property. Moreover, if $\mathcal{B}$ is one of the families 
~\eqref{itemlabel:eg:closure:3} and \eqref{itemlabel:eg:closure:4} in Example~\ref{eg:closure},  then $\mathcal{B}$ has the effective good closure property. 
\end{proposition}

\begin{proof}
Denote by
\begin{align*}
\mathcal{C}_1 &= \text{ family of bounded convex subsets of $V$} \\
\mathcal{C}_2 &= \text{ family of  bounded convex definable (in some fixed o-minimal expansion of $\mathbb{R}$) subsets of $V$} \\
\mathcal{C}_3 &= \text{ family of bounded convex semi-algebraic subsets of $V$} \\
\mathcal{C}_4 &= \text{ family of  hemihedral subsets of $V$}, \\
\end{align*}
and for $i=1,\ldots,4$, by
\begin{align*}
\mathcal{B}_i &= \text{family of finite unions of elements of $\mathcal{C}_i$}.
\end{align*}
First observe that the families 
$\mathcal{B}_1,\ldots, \mathcal{B}_4$, 
are all closed under taking finite unions and
images under linear maps.

Suppose $A =  \bigcup_{r=1}^{m} A_i$ and $B = \bigcup_{s=1}^n B_j$, where each
$A_r,B_s$ belong to $\mathcal{C}_i$.

Then, 
\begin{align*}
    A + B &= \bigcup_{r,s} (A_r+B_s), \\
    A \cmult B &= \bigcup_{r,s} (A_r \cmult B_s), \\ 
    \conv{A\cup B} &= \conv{\bigcup_{r} A_r \cup \bigcup_{s} B_s}.
\end{align*}

Because of the above identities, it suffices to prove that for any two sets $A, B \in \mathcal{C}_i$, 
$c \cdot A, c \geq 0, A+B, A \cmult B$ also belong to $\mathcal{C}_i$.

Closure under scalar multiplication is closure under linear images. 

Closure under Minkowski sum follows because
\[
A+B=\{z\in V:\exists a\in A,\exists b\in B,\ z=a+b\}
\]
is the image under a projection.

It remains to treat the positive geometric join. 
For $A \in \mathcal{C}_i$, define its positive homogenization
\[
A^+ :=
\{(t,u)\in \mathbb R_{>0}\times V:\ u=t a \text{ for some } a\in A\}.
\]

It is clear that if $A$ is convex (and definable in an o-minimal structure), then
$A^+$ is also convex (and definable).

We now prove that if $A$ is semi-linear then so is $A^+$.

If \(A\) is described by a Boolean combination of affine atoms
\(\ell(x)\ \bowtie\ 0\), where \(\bowtie\in\{=,\leq,<,\geq,>\}\), then
\(A^+\) is described by the corresponding atoms
\[
t>0,\qquad t\,\ell(u/t)\ \bowtie\ 0.
\]
Since \(\ell\) is affine, \(t\,\ell(u/t)\) is affine in \((t,u)\). Hence
\(A^+\) is semi-linear. 

Now, given $A,B \in \mathcal{C}_i$

\[
A\cmult B
=
\left\{
z\in V:
\begin{array}{l}
\exists t,s>0,\ \exists u,v\in V,\\
t+s=1,\ (t,u)\in A^+,\ (s,v)\in B^+,\\
z=u+v
\end{array}
\right\}.
\]
Since the right-hand side is the image under projection of a convex
set, \(A\mathring\star B\) is convex. It is definable, semi-algebraic,
or semi-linear whenever \(A\) and \(B\) have the corresponding
property. It is also bounded, since it is contained in
\(\operatorname{conv}(A\cup B)\). Hence
\(A\mathring\star B\in\mathcal C_i\).

Finally, to prove that $\conv{A \cup B}$ is again a set in $\mathcal{C}_i$,
it suffices to show that if $A$ is a bounded definable (resp. semi-linear) set,
then $\conv{A}$ is a bounded convex definable (resp. bounded, convex, semi-linear set).

Define
\[
A^{\geq 0}:=A^+\cup\{(0,0)\}\subset \mathbb R_{\geq 0}\times V.
\]

By Carath\'eodory's theorem, with \(d=\dim V\),
\[
\operatorname{conv}(A)
=
\left\{
z\in V:
\begin{array}{l}
\exists (t_0,u_0),\ldots,(t_d,u_d)\in A^{\geq 0},\\
\sum_{i=0}^d t_i=1,\quad z=\sum_{i=0}^d u_i
\end{array}
\right\}.
\]
Using the closure under linear projection, we again obtain that $\conv{A} \in \mathcal{C}_i$.

The effectivity claim follows from the fact that the displayed formulas 
are effectively constructed, 
and 
effective quantifier-elimination in the theory of the reals \cite[Chapter 14]{BPRbook2}
for the family $\mathcal{B}_3$, and
Fourier--Motzkin elimination \cite{Ziegler} for the family $\mathcal{B}_4$,
effectively eliminates the existentially quantified variables.
\end{proof}

\begin{proof}[Proof of Theorem~\ref{thm:hemihedra}]
It follows from  the definition of $\mathbb{P}[\cdot]$ and good closure property of $\mathcal{B}$
    that
    for any polynomial $\varphi \in \mathbb{P}[(\emptyset,\Omega_p)]$, and any assignment $\bP$ of $\Omega_p$
    such that each $\bP(P) \in \mathcal{B}$ 
    $\evalmap(\varphi;\bP) \in \mathcal{B}$.
    Indeed, a polynomial in \(P[(\emptyset,\Omega_p)]\) is obtained from the
parameter sets by finitely many applications of precisely the operations
covered 
in Definition~\ref{def:property:closure}.
    Theorem~\ref{thm:hemihedra} now follows from Theorem~\ref{thm:gaussian-elim}.
\end{proof}

\begin{proof}[Proof of Corollary~\ref{cor:hemihedra}]
By Theorem~\ref{thm:hemihedra}, 
and the fact that the family $\mathcal{B}$ of finite unions of hemihedra has the effective good closure property by Proposition~\ref{prop:closure-under-ops},
each \(\bX_i\) is a finite union of hemihedra and hence is
a bounded semi-linear set. Since \(\bX_i\) is convex, Definition~\ref{def:hemihedra}
implies that \(\bX_i\) is a hemihedron.
\end{proof}

\subsection{Iterative Realization of the Smallest Solution}
\label{subsec:op-real}
This section presents an alternative characterization of the smallest solution of a system of inequalities. 

We fix in this section 
$\Omega = (\Omega_v,\Omega_p)$ and $\Omega_v = \{X_1,\ldots,X_n\}$.
Consider a system $I$ of inequalities 
$\left( X_j \geq \varphi_j \right)_{j=1}^n$ and an 
arbitrary assignment 
$\bP \in (2^V)^{\Omega_p}$.

\begin{definition}
    \label{def:itr}
    \begin{enumerate}
        \item Let $\bX\p 0 
        \in\convset(V)^n$.
        
        \item For $i\in\{0,1,\dotsc\}$, define $\bX\p{i+1}\in\convset(V)^n$ as follows:
            For all $j\in \{1,2,\dotsc,n\}$, let 
                
                $$\bX\p{i+1}_j \defeq \conv{\bX\p{i}_j \;\cplus\; \evalmap\left(\varphi_j ; \bX\p i,\bP \right) \; }.$$
    \end{enumerate}
    We denote $\itr(i,I;\bX^{(0)},\bP) \defeq \bX\p i$. If
    $\bX\p 0= (\emptyset,\emptyset,\dotsc,\emptyset)$, then we will denote
    $\itr(i,I;\bX^{(0)},\bP)$ by $\itr(i,I;\bP)$.
\end{definition}

\begin{lemma}
    \label{lem:itr1}
    $\bX\p{i+1}\geq \bX\p i$ for $i\in \{0,1,\dotsc\}$.
\end{lemma}
\begin{proof}
Obvious.
\end{proof}

For arbitrary sets $A_1, \dotsc, A_n,B_1,\dotsc, B_n\subseteq \RR^d$, define
    $$(A_1, A_2, \dotsc, A_n) \cup (B_1, B_2, \dotsc, B_n) \defeq \left( A_1\cup B_1, A_2\cup B_2, \dotsc, A_n\cup B_n\right).$$
Finally,  we define
\begin{definition}
\label{def:itr2}
    \begin{equation}
        \label{eqn:itr}
        \itr(I;\bX^{(0)},\bP) \defeq \bigcup\limits_{i\geq 0} \; \itr(i,I;\bX^{(0)},\bP).
    \end{equation}
    
    If
    $\bX\p 0= (\emptyset,\emptyset,\dotsc,\emptyset)$, then we will denote
    $\itr(I;\bX^{(0)},\bP)$ by $\itr(I;\bP)$.
\end{definition}
    
Since each $\itr(i,I;\bX^{(0)},\bP)\in\convset(V)^n$
it follows immediately that
$\itr(I;\bX^{(0)},\bP) \in \convset(V)^n$ for any $\bX^{(0)} \in \convset(V)^n$.

We now prove:
\begin{lemma}[Iterative Construction of the Smallest Solution]
\label{lem:itr-char}
\label{lem:itr2}
    Let $I = \left(X_i \geq \varphi_i\right)_{i=1}^n $ be a  system of inequalities and $\bP \in (2^V)^{\Omega_p}$. Then, for any $\bX^{(0)} \in \convset(V)^n$,
     $\itr(I;\bX^{(0)},\bP)$ is the smallest solution of $I$ containing $\bX\p 0$. 
    In particular (setting $\bX^{(0)} = (\emptyset,\ldots,\emptyset)$) $\itr(I;\bP) = \ssol(I;\bP)$.
\end{lemma}
We will need the following lemma in the proof of Lemma~\ref{lem:itr-char}.

\begin{lemma}
\label{lem:finitary-continuity}
    Suppose $\bX^{(0)} \subset \bX^{(1)} \subset \cdots$ be an increasing sequence in $\convset(V)^{\Omega_v}$,
    and $\varphi \in \mathbb{P}[\Omega]$. Then, for all $\bP \in (2^V)^{\Omega_p}$
    \[
    \evalmap(\varphi; \bigcup_{i \geq 0} \bX^{(i)},\bP) = \bigcup_{i \geq 0} \evalmap(\varphi;  \bX^{(i)},\bP)
    \]
\end{lemma}

\begin{proof}
    For each $i \geq 0$, the inclusion
    $ \evalmap(\varphi;  \bX^{(i)},\bP) \subset \evalmap(\varphi; \bigcup_{i \geq 0} \bX^{(i)},\bP)$
    follows from Lemma~\ref{lem:monotonicity} which implies that
    \[
    \bigcup_{i \geq 0} \evalmap(\varphi;  \bX^{(i)},\bP) \subset 
    \evalmap(\varphi; \bigcup_{i \geq 0} \bX^{(i)},\bP).
    \]

To prove the opposite inclusion    
let
\[
\mathbf{x} \in
\evalmap
\left(
\varphi;
\bigcup_{i\geq 0}\mathbf X^{(i)},
\mathbf P
\right).
\]
A witness for this membership involves one monomial of \(\varphi\),
finitely many factors of that monomial, and finitely many points chosen
from the variable sets occurring in those factors. Each of these
finitely many points belongs to \(\bX^{(i)}(X_j)\) for some
\(i\). Since the sequence of assignments is increasing, there is a
single index \(N\) for which all the witnessing points belong to the
corresponding sets \(\bX^{(N)}(X_j)\). Therefore
\[
\mathbf{x} \in
\evalmap
\bigl(
\varphi;\bX^{(N)},\bP
\bigr).
\]
\end{proof}

\begin{proof}[Proof of Lemma~\ref{lem:itr-char}]

For \(i\geq 0\), write
\[
\mathbf X^{(i)}
:=
\itr
\bigl(i,I;\mathbf X^{(0)},\mathbf P\bigr),
\]
and set
\[
\mathbf X^{(\infty)}
:=
\bigcup_{i\geq 0}\mathbf X^{(i)}
=
\itr
\bigl(I;\mathbf X^{(0)},\mathbf P\bigr),
\]
where the union is taken component-wise. We prove the following two
claims:
\begin{enumerate}
\item
\(\mathbf X^{(\infty)}\) is a solution of \(I\);
\item
every solution of \(I\) containing \(\mathbf X^{(0)}\) also contains
\(\mathbf X^{(\infty)}\).
\end{enumerate}

We first prove the first claim. By Lemma~\ref{lem:finitary-continuity}, for every
\(j\in\{1,\ldots,n\}\),
\[
\evalmap
\bigl(
\varphi_j;\mathbf X^{(\infty)},\mathbf P
\bigr)
=
\bigcup_{i\geq 0}
\evalmap
\bigl(
\varphi_j;\mathbf X^{(i)},\mathbf P
\bigr).
\]
On the other hand, by the definition of the iterates,
\[
\mathbf X^{(i+1)}_j
=\conv{
\left(
\mathbf X^{(i)}_j
\mathbin{\cplus}
\evalmap
\bigl(
\varphi_j;\mathbf X^{(i)},\mathbf P
\bigr)
\right)}.
\]
Consequently,
\[
\evalmap
\bigl(
\varphi_j;\mathbf X^{(i)},\mathbf P
\bigr)
\subseteq
\mathbf X^{(i+1)}_j
\subseteq
\mathbf X^{(\infty)}_j
\]
for every \(i\geq 0\). Taking the union over \(i\), we obtain
\[
\evalmap
\bigl(
\varphi_j;\mathbf X^{(\infty)},\mathbf P
\bigr)
\subseteq
\mathbf X^{(\infty)}_j.
\]
Thus \(\mathbf X^{(\infty)}\) satisfies every inequality of \(I\), and
hence
\[
\mathbf X^{(\infty)}\in\operatorname{sol}(I;\mathbf P).
\]

We next prove the second claim. Let
\[
\mathbf Y\in\operatorname{sol}(I;\mathbf P)
\]
be a solution satisfying
\[
\mathbf Y\geq \mathbf X^{(0)}.
\]
We prove by induction on \(i\) that
\[
\mathbf Y\geq \mathbf X^{(i)}
\]
for every \(i\geq 0\).

The assertion holds for \(i=0\) by assumption. Suppose that
\(\mathbf Y\geq\mathbf X^{(i)}\). Since \(\mathbf Y\) is a solution of
\(I\), for every \(j\in\{1,\ldots,n\}\),
\[
Y_j
\supseteq
\evalmap
\bigl(
\varphi_j;\mathbf Y,\mathbf P
\bigr).
\]
By the monotonicity of evaluation (Lemma~\ref{lem:monotonicity}) the
induction hypothesis gives
\[
\evalmap
\bigl(
\varphi_j;\mathbf Y,\mathbf P
\bigr)
\supseteq
\evalmap
\bigl(
\varphi_j;\mathbf X^{(i)},\mathbf P
\bigr).
\]
Therefore,
\[
Y_j
\supseteq
\evalmap
\bigl(
\varphi_j;\mathbf X^{(i)},\mathbf P
\bigr).
\]
We also have
\[
Y_j\supseteq \mathbf X^{(i)}_j.
\]
Since \(Y_j\) is convex, it follows that
\[
\begin{aligned}
Y_j
&\supseteq
\operatorname{conv}
\left(
\mathbf X^{(i)}_j
\mathbin{\cplus}
\evalmap
\bigl(
\varphi_j;\mathbf X^{(i)},\mathbf P
\bigr)
\right) \\
&=
\mathbf X^{(i+1)}_j.
\end{aligned}
\]
This holds for every \(j\), and hence
\[
\mathbf Y\geq \mathbf X^{(i+1)}.
\]
The induction is complete.

It follows that
\[
\mathbf Y
\geq
\bigcup_{i\geq 0}\mathbf X^{(i)}
=
\mathbf X^{(\infty)}.
\]
Thus every solution containing \(\mathbf X^{(0)}\) contains
\(\mathbf X^{(\infty)}\). Together with the first claim, this proves
that
\[
\itr
\bigl(I;\mathbf X^{(0)},\mathbf P\bigr)
\]
is the smallest solution of \(I\) containing \(\mathbf X^{(0)}\).

Finally, if
\[
\mathbf X^{(0)}=(\emptyset,\ldots,\emptyset),
\]
then every solution contains \(\mathbf X^{(0)}\). Therefore,
\[
\itr(I;\mathbf P)
=
\ssol(I;\mathbf P).
\]
\end{proof}

\begin{remark}
    Starting with $\bX\p0 = (\emptyset,\dotsc,\emptyset)\in\convset(V)^n$ and an assignment where $\bP_1,\dotsc,\bP_t$ are polytopes, note that each $\bX\p i_j$ is a convex set. 
    The complexity of describing them can become increasingly complicated 
    with $i\in\{0,1,2,\dotsc\}$ 
(i.e., be unbounded as a function of $i$); for example, see Section~\ref{subsec:itr-sol-fig}. 
    However, their infinite union, the set $\itr(I;\bP)_j$, has a finite algebraic complexity. 
\end{remark}

%% file: arxiv-laminar.tex
\section{Proof of Theorem~\ref{thm:laminar}}
\label{sec:laminar}
We introduce in Subsection~\ref{subsec:grid} the notion of a finite grid
which is a finite subset of $W$. In Subsection~\ref{subsec:witness-tree} we 
introduce the notion of witness trees, and then in 
Subsection~\ref{subsec:interpolation} use it to prove (see Proposition~\ref{prop:struc}) that for any finite subset $S \subset V$, if there exists a grid $\cG$ having a certain property (Definition~\ref{def:fine-grid}), then $G^{(\infty)}_\Lambda(S)$ can be constructed semi-algebraically from the set of fibers  $G^{(\infty)}_\Lambda(S)_{\cG}$.
In Subsection~\ref{subsec:reduction}
we prove (Proposition~\ref{prop:reduction-lam-system}) 
that under the same assumption on $\cG$,
the tuple of fibers $G^{(\infty)}_\Lambda(S)_{\cG}$ (see Notation~\ref{not:fibers} below)
is the smallest solution of a certain system of inequalities. Theorem~\ref{thm:hemihedra} then implies that each fiber is a hemihedron, which together with Proposition~\ref{prop:struc}
implies that $G^{(\infty)}_\Lambda(S)_{\cG}$, being a finite (disjoint) union of hemihedra,
is a semi-algebraic set. Finally, in Subsection~\ref{subsec:existence-of-grid} we prove that (Proposition~\ref{prop:existence-of-grid}) if $\dim W_i = 1, 1 \leq i \leq k$, then there always exists a grid satisfying the required property. Together with Proposition~\ref{prop:laminar}, this suffices to prove Theorem~\ref{thm:laminar}.

\subsection{Grid Points}
\label{subsec:grid}

We fix a finite dimensional real vector space $V$ and subspaces $U,W_1,\ldots,W_k$ such that 
\[
V = U \oplus W,
\]
where $W = W_1 \oplus \cdots \oplus W_k$.
We will denote by 
\[
\pi_i:V \rightarrow V, 
\]
the canonical projection to $W_i$ for $1 \leq i \leq k$, and denote 
\[
\pi = \pi_1+\cdots+\pi_k.
\]

\begin{notation}
\label{not:fibers}
For $w \in W$, $S \subset V$, we will denote by 
$S_w = \pi^{-1}(w) \cap S$, and more generally for any subset 
$\cG \subset W$, we will denote by $S_\cG = S \cap \pi^{-1}(\cG)$.
\end{notation}

We now fix 
\[
\Lambda = \bigcup_{i=1}^k (U+W_i).
\]

Instead of directly working with $G^{(t)}_\Lambda(S)$, we will define a new (related) sequence of recursively defined sets, $\widetilde{G}^{(t)}_\Lambda(S)$. 
\begin{definition}
\label{def:T}
   \begin{enumerate}[1.] 
    \item 
        $$ \widetilde{G}^{(0)}_\Lambda(S) \defeq G^{(0)}_\Lambda(S) =S.$$
    \item 
        For $t\in\{0,1,2,\dotsc\}$, 
        $$ \widetilde{G}^{(t+1)}_\Lambda(S) \defeq \left\{
            \sum_{j=1}^p \lambda_j\cdot v_j \;\colon\; \begin{matrix}
                1 \leq p \leq \dim U + \dim W_i +1 \\ \lambda_1,\lambda_2,\dotsc, \lambda_p > 0 \\
                \lambda_1 + \lambda_2 + \dotsi + \lambda_p = 1 \\
                v_1, v_2, \dotsc, v_p \in \widetilde{G}^{(t)}_\Lambda(S)\\
                \text{$\exists i, 1\leq i \leq k$ such that for all $h \neq i, 1 \leq h \leq k$,} \\
                \pi_h(v_1) = \cdots =\pi_h(v_p)
            \end{matrix}
        \right\}$$
    \item  
        Finally, we define
        $$\widetilde{G}^{(\infty)}_\Lambda(S) \defeq \bigcup\limits_{t\geq 0}\widetilde{G}^{(t)}_\Lambda(S).$$
\end{enumerate} 
\end{definition}

We have:
\begin{lemma}
\label{lem:S-T-same}
    $\widetilde{G}^{(\infty)}_\Lambda(S)= G^{(\infty)}_\Lambda(S)$.
\end{lemma}

\begin{proof}
First, $G^{(\infty)}_\Lambda(S)\subseteq \widetilde G^{(\infty)}_\Lambda(S)$.  Indeed, if $x,y\in \widetilde G^{(t)}_\Lambda(S)$ and $x-y\in U+W_i$, then $\pi_h(x)=\pi_h(y)$ for all $h\neq i$, and the binary point $\alpha x+(1-\alpha)y$ is allowed in the definition of $\widetilde G^{(t+1)}_\Lambda(S)$ with $p=2$; if $x=y$, use $p=1$.

Conversely, let
\[
 z=\sum_{r=1}^p\lambda_r v_r\in \widetilde G^{(t+1)}_\Lambda(S)
\]
be obtained using coordinate $i$. By induction, all $v_r$ lie in $G^{(\infty)}_\Lambda(S)$.  Since the sequence $G^{(m)}_\Lambda(S)$ is increasing because $\mathbf{0}\in\Lambda$, there is $m$ with $v_1,\ldots,v_p\in G^{(m)}_\Lambda(S)$.  Every difference $v_r-v_s$ lies in $U+W_i\subseteq\Lambda$. A binary convex-combination tree therefore constructs $z$ from the $v_r$ in finitely many further $G_\Lambda$ operations. Hence $z\in G^{(\infty)}_\Lambda(S)$.
\end{proof}

\subsection{Witness trees}
\label{subsec:witness-tree}
A given point $v \in \widetilde{G}^{(\infty)}_\Lambda(S)$ is obtained from $S$ by a
finite procedure of taking convex positions of certain other points in each step. Note that this sequence is not unique, and the same point can be obtained
through more than one such sequence. It is convenient to use the following notion of \emph{witness trees} to keep track of these sequences.

\begin{notation}[Witness trees]
Each witness tree $T$ has a unique root node denoted $\mathrm{root}(T)$. 
We denote the set of nodes of $T$ by $\mathrm{nodes}(T)$.
Each node $\mathfrak{n} \in \mathrm{nodes}(T)$
has a finite set of children denoted $\mathrm{children}(\mathfrak{n}) \subset \mathrm{nodes}(T)$. 
Nodes with an empty set of children will be called leaves of $T$, and the set of leaves will be denoted by $\mathrm{Leaves}(T)$. 
\end{notation}

\begin{notation}[Height of a tree]
The height of $\mathfrak{n} \in \mathrm{nodes}(T)$,
denoted $\mathrm{ht}_T(\mathfrak{n})$, 
is the largest integer $h$ (if it exists) such that there exists 
$\mathfrak{n}_0 = \mathfrak{n}, \mathfrak{n}_1,\ldots,\mathfrak{n}_h$,
with $\mathfrak{n}_h \in \mathrm{Leaves}(T)$, and for each $i, 0 < i \leq h$,
$\mathfrak{n}_i \in \mathrm{children}(\mathfrak{n}_{i-1})$. If no such $h$ exists we set $\mathrm{ht}_T(\mathfrak{n}) = \infty$. 
We call $\mathrm{ht}_T(\mathrm{root}(T))$ to be the height of $T$.
\end{notation}

\begin{notation}
We also define for each $\mathfrak{n} \in \mathrm{nodes}(T)$, its depth
denoted $\mathrm{depth}_T(\mathfrak{n})$, inductively as follow:

\begin{enumerate}[Depth of a node]
    \item $\mathrm{depth}_T(\mathrm{root}(T)) = 0$, 
    \item $\mathrm{depth}_T(\mathfrak{n}') = \mathrm{depth}_T(\mathfrak{n})+1$
    for all $\mathfrak{n}' \in \mathrm{children}(\mathfrak{n})$.
\end{enumerate}
\end{notation}

\begin{definition}[Annotations of the nodes of a tree]
Each $\mathfrak{n} \in \mathrm{nodes}(T)$ of the tree contains:
\begin{enumerate}
    \item a point  $v(\mathfrak{n}) \in V$,
    
    \item if $\mathfrak{n} \neq \mathrm{root}(T)$,
a weight $\alpha = \alpha(\mathfrak{n}), 0< \alpha \leq 1$, 
satisfying for every $\mathfrak{n} \in \mathrm{nodes}(T) \setminus \mathrm{Leaves}(T)$,
\[
\sum_{\mathfrak{n}' \in \mathrm{children}(\mathfrak{n})} \alpha(\mathfrak{n}') = 1,
\]
    
    \item and if $\mathfrak{n} \not\in \mathrm{Leaves}(T)$,
an \emph{index}  $i = i(\mathfrak{n}), 1 \leq i \leq k$.
\end{enumerate}
\end{definition}

\begin{remark}
\label{rem:subtree}
Observe that for trees of finite height, if $\mathfrak{n'} \in \mathrm{children}(\mathfrak{n})$, then $\mathrm{ht}_T(\mathfrak{n}') \leq  \mathrm{ht}_T(\mathfrak{n}) - 1$.

Also observe that, each $\mathfrak{n} \in \mathrm{nodes}(T)$ 
determines a subtree $T'$
with $\mathrm{root}(T') = \mathfrak{n}$ in an obvious way, and it follows from the above observation that if $\mathfrak{n} \neq \mathrm{root}(T)$, then 
\[
\mathrm{ht}(T') < \mathrm{ht}(T). 
\]
\end{remark}

\begin{definition}[Witness trees associated to $v \in G^{(\infty)}_\Lambda(S)$]
\label{def:witness-tree}
    Let $v \in G^{(\infty)}_\Lambda(S)$. We say that a tree $T$  is a \emph{witness tree for $v$}  if it satisfies the following:
    \begin{enumerate}[1.]
        \item $\mathrm{ht}(T)$ is finite;
        \item $v(\mathrm{root}(T)) = v$;
        \item For each $\mathfrak{n} \in \mathrm{nodes}(T) \setminus \mathrm{Leaves}(T)$:
        \begin{enumerate}
         \item (Projection condition) for all $j, 1\leq j \leq k, j \neq i(\mathfrak{n})$ and $\mathfrak{n}' \in \mathrm{children}(\mathfrak{n})$, $\pi_j(v(\mathfrak{n}')) = \pi_j(v(\mathfrak{n}))$.
        \item (Convexity condition)
            \[
            v(\mathfrak{n}) = \sum_{\mathfrak{n}' \in \mathrm{children}(\mathfrak{n})} \alpha(\mathfrak{n}')\cdot v(\mathfrak{n}').
            \]
        \end{enumerate}
        \item For each $\mathfrak{n} \in \mathrm{Leaves}(T)$, $v(\mathfrak{n}) \in S$.
    \end{enumerate}
\end{definition}

\begin{lemma}
\label{lem:every-point-has-a-witness-tree}
A point $v \in G^{(\infty)}_\Lambda(S)$ if and only if $v$ admits a witness tree.
\end{lemma}

\begin{proof}
First note that if a node $\fn$  in a witness tree with index $i(\fn) = i$ has more than \(\dim(U+W_i)+1\) children, its convex combination
can be compressed, by Carath\'eodory's theorem in the affine space which is a translate of \(U+W_i\) containing the points associated with the children of $\fn$,
to a convex combination of at most \(\dim(U+W_i)+1\) children, without
changing $v(\fn)$ or the projection condition.
 It is now obvious from Definitions~\ref{def:T} and \ref{def:witness-tree} that
 $v \in \widetilde{G}^{(\infty)}(S)$ if and only if $v$ admits a witness tree.
 The lemma now follows from Lemma~\ref{lem:S-T-same}.
\end{proof}

For each $i, 1 \leq i \leq k$, we fix a finite set 
$\cG^{(i)} \subset W_i$,
and denote by $\cG = \cG^{(1)} \times \cdots \times \cG^{(k)} \subset W$.

\begin{notation}
For $v \in V$, we denote
\[
\cI_{\cG}(v) = \{i \;:\; 1 \leq i \leq k, \pi_i(v) \not\in \cG_i \}.
\]
\end{notation}

Let $v \in V$, $\cI_{\cG}(v) = \{i_1,\ldots,i_\ell\}$, and
$\Pi$ the ordered tuple  $(i_1,\ldots,i_\ell)$.

\begin{definition}[$(\Pi,\cG)$-gridded witness trees]
\label{def:gridded-witness-tree}
    We say that a tree $T$ is a
    \emph{$(\Pi,\cG)$-gridded witness tree for $v$}
    if $v(\mathrm{root}(T)) = v$, and 
    for each $\mathfrak{n} \in \mathrm{nodes}(T)$ with $d = \mathrm{depth}_T(\mathfrak{n})  
    $,
    \begin{enumerate}[1.]
    
        \item for 
        all $i, 1 \leq i \leq k, i \not\in \{i_{d+1},\ldots, i_\ell\}$,
        $
        \pi_{i}(v(\mathfrak{n})) \in \cG_{i},
        $ and
         \item $\mathfrak{n} \not\in \mathrm{Leaves}(T)$ and
    $i(\mathfrak{n}) = i_{d+1}$, if $d \leq \ell - 1$.
    (For \(d\ge \ell\), the set \(\{i_{d+1},\ldots,i_\ell\}\) is interpreted as empty.
Thus all \(W_i\)-coordinates are required to be grid-valued at depths \(d\ge \ell\).)
    \end{enumerate}
    
 We say that \emph{$v$ admits a $\cG$-gridded witness tree} if 
 $v$ admits a $(\Pi,\cG)$-gridded witness tree for some permutation
 $\Pi$ of the  $\cI_{\cG}(v)$.
\end{definition}

\begin{definition}
\label{def:fine-grid}
Let $S \subset V$ be a finite subset.
We say that \(G\) is sufficiently fine for \(S\) if
\(\pi(S)\subset G\) and every
\(v\in G_\Lambda^{(\infty)}(S)\) admits a
\(G\)-gridded witness tree.
\end{definition}

\subsection{Reduction to a System of Inequalities}
\label{subsec:reduction}

We assume in this section that there exists a sufficiently fine grid
$\cG = \cG^{(1)} \times \cdots \times \cG^{(k)}$ for $S$ which we fix for the rest of the section.
Our objective is to design a system of linear inequalities over convex sets so that its smallest solution corresponds to the restrictions 
$G^{(\infty)}_\Lambda(S)_w, w \in \mathcal{G}$.

\begin{definition}
\label{def:cI}
    Let $X_w$ denote a variable for each $w \in \mathcal{G}$, and let $P_v$ denote a parameter for each $v \in S$.  $I_{\cG,S}$ denote 
    the set consisting of the following inequalities:
    \begin{enumerate}
        \item For each $v \in S$, 
        \begin{equation}
        \label{eqn:I.1}
            X_{\pi(v)} \geq P_v.
        \end{equation}.
        \item For each $i, 1 \leq i \leq k$, and 
        $w_{i}^{(0)}, \ldots, w_{i}^{(p)} \in  \cG^{(i)}, 2 \leq p \leq \dim W_i+1$, 
        such that $w_{i}^{(1)},\ldots,w_{i}^{(p)}$ are affinely independent,
        and $w_j \in \cG^{(j)}, 1 \leq j \leq k, j \neq i$, 
        and $\alpha^{(1)},\ldots,\alpha^{(p)} > 0$ with
        $\alpha^{(1)}+\cdots+\alpha^{(p)}= 1$ and 
        \begin{equation}
        \label{eqn:def:cI}
         w_{i}^{(0)} = \sum_{h=1}^{p} \alpha^{(h)} \cdot w_{i}^{(h)},    
        \end{equation}
         
        \begin{equation}
        \label{eqn:I.2}
        X_{w_{i}^{(0)}+\sum_{j\neq i} w_j} \geq \sum_{h=1}^{p} \alpha^{(h)} \cdot X_{w_{i}^{(h)}+\sum_{j\neq i} w_j}.
        \end{equation}
    \end{enumerate}
\end{definition}

\begin{remark}
    Notice that in Definition~\ref{def:cI}, since $w_{i}^{(1)},\ldots,w_{i}^{(p)}$ 
    are assumed to be affinely independent and 
    $
    w_{i}^{(0)} \in \convo{\{w_{i}^{(1)},\ldots,w_{i}^{(p)}\}}
    $
    (resp. 
    $
    w_{i}^{(0)} \in \conv{\{w_{i}^{(1)},\ldots,w_{i}^{(p)}\}}
    $),
    there exists for each tuple
    $(w_{i}^{(0)},w_{i}^{(1)},\ldots,w_{i}^{(p)})$
    a unique $\alpha^{(1)},\ldots,\alpha^{(p)} > 0$ 
    (resp. $\alpha^{(1)},\ldots,\alpha^{(p)} \geq 0$)
    satisfying
    \eqref{eqn:def:cI}.
    
    Moreover, there are clearly only finitely many tuples
    \[
    (w_{i}^{(0)},w_{i}^{(1)},\ldots,w_{i}^{(p)}, w_1,\ldots,\widehat{w_i},\ldots,w_k)
    \]
    with $w_{i}^{(0)}, \ldots,w_{i}^{(p)} \in \cG^{(i)}$ $2 \leq p \leq \dim W_i+1$, and
    $w_j \in \cG^{(j)}$ for all $j \neq i, 1 \leq j \leq k$.
    
    Together, the above facts imply that $\card(I_{\cG,S}) < \infty$.
\end{remark}

\begin{proposition}[Reduction to Solving System of Linear Inequalities over Convex Sets]
\label{prop:reduction-lam-system}
    Let 
    $\left( \; \bX_w\p* \;\colon\; w\in\cG \; \right) = \ssol(I_{\cG,S};\bP)$,
    where $\bP(P_v) = \{v\}, v \in S$.
    Then, $G^{(\infty)}_\Lambda(S)_w = \bX_w\p*$ for every $w\in\cG$. 
\end{proposition}

Before proving Proposition~\ref{prop:reduction-lam-system} we need an auxiliary lemma.

\begin{lemma}
\label{lem:new}
Let \(v_1,\ldots,v_n \in 
V\), and let
\[
v=\sum_{i=1}^n \alpha_i v_i,
\qquad 
\alpha_i\geq 0,
\qquad
\sum_{i=1}^n \alpha_i=1.
\]
Let \(\mathcal S\) denote the collection of all subsets
\[
\sigma \subseteq \{1,\ldots,n\}
\]
such that the points \(\{v_i : i\in \sigma\}\) are affinely independent and
\[
v \in \operatorname{relint}
\operatorname{conv}\{v_i : i\in \sigma\}.
\]
For each \(\sigma\in\mathcal S\), let \(I_\sigma=(I_{\sigma,1},\ldots,I_{\sigma,n})\in\mathbb{R}^n\) be the coefficient vector defined by
\[
I_{\sigma,i}=0 \quad\text{if } i\notin \sigma,
\]
and, for \(i\in \sigma\), by the unique barycentric coordinates satisfying
\[
v=\sum_{i\in\sigma} I_{\sigma,i}v_i,
\qquad
\sum_{i\in\sigma} I_{\sigma,i}=1,
\qquad
I_{\sigma,i}>0.
\]
Then every convex representation of \(v\) as a convex combination of
\(v_1,\ldots,v_n\) is a convex combination of the vectors \(I_\sigma\).
In particular, there exist numbers \(t_\sigma\geq 0\), with
\[
\sum_{\sigma\in\mathcal S} t_\sigma=1,
\]
such that
\[
(\alpha_1,\ldots,\alpha_n)
=
\sum_{\sigma\in\mathcal S} t_\sigma I_\sigma.
\]
Equivalently,
\[
\alpha_i
=
\sum_{\sigma\in\mathcal S} t_\sigma I_{\sigma,i}
\qquad
\text{for every } i=1,\ldots,n.
\]
\end{lemma}

\begin{proof}[Proof of Lemma~\ref{lem:new}]
Let
\[
\Delta_{n-1}
=
\left\{
\lambda=(\lambda_1,\ldots,\lambda_n)\in \mathbb R^n :
\lambda_i\geq 0,\ \sum_{i=1}^n \lambda_i=1
\right\}
\]
be the standard simplex. Consider the fiber over \(v\):
\[
P_v
=
\left\{
\lambda\in \Delta_{n-1} :
\sum_{i=1}^n \lambda_i v_i=v
\right\}.
\]
This is a compact convex polytope, since it is the intersection of the
simplex \(\Delta_{n-1}\) with an affine subspace.

We will show that the extreme points of \(P_v\) are precisely the vectors
\(I_\sigma\), where \(\sigma\in\mathcal S\). It will then follow that
\[
P_v=\operatorname{conv}\{I_\sigma : \sigma\in\mathcal S\},
\]
because every compact convex polytope is the convex hull of its extreme
points.

First, let \(\lambda\in P_v\) be an extreme point. Define its support by
\[
\operatorname{supp}(\lambda)
=
\{i\in\{1,\ldots,n\}:\lambda_i>0\}.
\]
We claim that the points
\[
\{v_i : i\in \operatorname{supp}(\lambda)\}
\]
are affinely independent.

Suppose not. Then there exist scalars \(c_i\), not all zero, indexed by
\(i\in \operatorname{supp}(\lambda)\), such that
\[
\sum_{i\in \operatorname{supp}(\lambda)} c_i=0
\]
and
\[
\sum_{i\in \operatorname{supp}(\lambda)} c_i v_i=0.
\]
Extend \(c\) to a vector in \(\mathbb R^n\) by setting \(c_i=0\) for
\(i\notin \operatorname{supp}(\lambda)\). Since \(\lambda_i>0\) for all
\(i\in \operatorname{supp}(\lambda)\), we may choose \(\varepsilon>0\)
small enough so that
\[
\lambda_i+\varepsilon c_i\geq 0
\qquad\text{and}\qquad
\lambda_i-\varepsilon c_i\geq 0
\]
for all \(i=1,\ldots,n\).

Now define
\[
\lambda^+=\lambda+\varepsilon c,
\qquad
\lambda^-=\lambda-\varepsilon c.
\]
Then
\[
\sum_{i=1}^n \lambda_i^+
=
\sum_{i=1}^n \lambda_i+\varepsilon\sum_{i=1}^n c_i
=
1,
\]
and similarly
\[
\sum_{i=1}^n \lambda_i^-=1.
\]
Also,
\[
\sum_{i=1}^n \lambda_i^+ v_i
=
\sum_{i=1}^n \lambda_i v_i
+
\varepsilon\sum_{i=1}^n c_i v_i
=
v,
\]
and similarly
\[
\sum_{i=1}^n \lambda_i^- v_i=v.
\]
Thus \(\lambda^+,\lambda^-\in P_v\). They are distinct because \(c\neq 0\),
and
\[
\lambda=\frac{1}{2}\lambda^+ + \frac{1}{2}\lambda^-.
\]
This contradicts the assumption that \(\lambda\) is an extreme point of
\(P_v\). Therefore \(\{v_i : i\in \operatorname{supp}(\lambda)\}\) must be
affinely independent.

Let
\[
\sigma=\operatorname{supp}(\lambda).
\]
Since \(\lambda\in P_v\), we have
\[
v=\sum_{i\in\sigma} \lambda_i v_i,
\qquad
\sum_{i\in\sigma}\lambda_i=1,
\qquad
\lambda_i>0 \text{ for all } i\in\sigma.
\]
Because the points \(\{v_i:i\in\sigma\}\) are affinely independent, these
are the barycentric coordinates of \(v\) with respect to that simplex.
Moreover, all barycentric coordinates are strictly positive, so
\[
v\in \operatorname{relint}\operatorname{conv}\{v_i:i\in\sigma\}.
\]
Hence \(\sigma\in\mathcal S\), and by uniqueness of barycentric
coordinates,
\[
\lambda=I_\sigma.
\]
Thus every extreme point of \(P_v\) is one of the vectors \(I_\sigma\).

Conversely, fix \(\sigma\in\mathcal S\). We show that \(I_\sigma\) is an
extreme point of \(P_v\). Suppose that
\[
I_\sigma=t\mu+(1-t)\nu
\]
for some \(\mu,\nu\in P_v\) and some \(t\in(0,1)\). If \(j\notin\sigma\),
then
\[
0=I_{\sigma,j}=t\mu_j+(1-t)\nu_j.
\]
Since \(\mu_j,\nu_j\geq 0\), it follows that
\[
\mu_j=\nu_j=0
\qquad
\text{for every } j\notin\sigma.
\]
Therefore both \(\mu\) and \(\nu\) are convex representations of \(v\)
supported on \(\sigma\). That is,
\[
v=\sum_{i\in\sigma} \mu_i v_i
=
\sum_{i\in\sigma} \nu_i v_i,
\qquad
\sum_{i\in\sigma}\mu_i
=
\sum_{i\in\sigma}\nu_i
=
1.
\]
Since \(\{v_i:i\in\sigma\}\) is affinely independent, the barycentric
coordinates of \(v\) with respect to this simplex are unique. Hence
\[
\mu=I_\sigma
\qquad\text{and}\qquad
\nu=I_\sigma.
\]
Thus \(I_\sigma\) is an extreme point of \(P_v\).

We have therefore proved that
\[
\operatorname{ext}(P_v)
=
\{I_\sigma:\sigma\in\mathcal S\}.
\]
Since \(P_v\) is a compact convex polytope, it is the convex hull of its
extreme points. Hence
\[
P_v
=
\operatorname{conv}\{I_\sigma:\sigma\in\mathcal S\}.
\]
Finally, the coefficient vector
\[
\alpha=(\alpha_1,\ldots,\alpha_n)
\]
belongs to \(P_v\), because
\[
\alpha_i\geq 0,\qquad \sum_{i=1}^n\alpha_i=1,
\qquad
\sum_{i=1}^n\alpha_i v_i=v.
\]
Therefore there exist coefficients \(t_\sigma\geq 0\), with
\[
\sum_{\sigma\in\mathcal S} t_\sigma=1,
\]
such that
\[
\alpha=\sum_{\sigma\in\mathcal S} t_\sigma I_\sigma.
\]
This is precisely the desired decomposition.
\end{proof}

\begin{proof}[Proof of Proposition~\ref{prop:reduction-lam-system}]
Let
\[
        H_w:=G_\Lambda^{(\infty)}(S)_w,\qquad w\in \cG.
\]
Let
\[
        (\bX_w^{(r)})_{w\in \cG}:=\operatorname{itr}(r,I_{\cG,S};\bP)
\]
denote the iterates of the system \(I_{\cG,S}\), starting from
\[
        \bX_w^{(0)}=\emptyset,\qquad w\in \cG.
\]
By Lemma~\ref{lem:itr-char},
\[
        \bX_w^{(*)}=\bigcup_{r\geq 0}\bX_w^{(r)}.
\]

We prove the equality
\[
        H_w=\bX_w^{(*)}
\]
in two inclusions.

\medskip
\noindent
\textbf{First inclusion: \(H_w\subset \bX_w^{(*)}\).}

For \(t\geq 0\), let \(\bZ_w^{(t)}\subset H_w\) be the set of all points
\(v\in H_w\) admitting a \(\cG\)-gridded witness tree of height at most
\(t\).  Since \(\cG\) is sufficiently fine,
\[
        H_w=\bigcup_{t\geq 0}\bZ_w^{(t)}.
\]
We prove by induction on \(t\) that
\[
        \bZ_w^{(t)}\subset \bX_w^{(t+1)}
        \qquad\text{for every }w\in \cG.
\]

For \(t=0\), a witness tree of height \(0\) consists of a single leaf.
Thus, if \(v\in \bZ_w^{(0)}\), then \(v\in S\) and \(\pi(v)=w\).  The
inequality
\[
        \bX_{\pi(v)}\geq P_v
\]
from Definition~\ref{def:cI} \eqref{eqn:I.1} implies that
\[
        v\in \bX_w^{(1)}.
\]
This proves the base case.

Assume now that the claim has been proved for \(t-1\), and let
\[
        v\in \bZ_w^{(t)}.
\]
If \(v\in \bZ_w^{(t-1)}\), then the induction hypothesis gives
\[
        v\in \bX_w^{(t)},
\]
and Lemma~\ref{lem:itr1} gives
\[
        \bX_w^{(t)}\subset \bX_w^{(t+1)}.
\]
Thus \(v\in \bX_w^{(t+1)}\), and we are done in this case.

We may therefore assume that \(v\notin \bZ_w^{(t-1)}\).  Choose a
\(\cG\)-gridded witness tree \(T\) for \(v\) of height at most \(t\).
Since \(v\notin \bZ_w^{(t-1)}\), the root of \(T\) is an internal node.
Let its splitting direction be \(i\).  Denote the children of the root
by \(\fm\in M\), and write
\[
        a_\fm:=\alpha(\fm),\qquad v_\fm:=v(\fm).
\]
Then
\[
        v=\sum_{\fm\in M}a_\fm \cdot v_\fm,\qquad
        a_\fm>0,\qquad \sum_{\fm\in M}a_\fm=1.
\]
The split is in direction \(i\), so for every \(\fm\in M\),
\[
        \pi_q(v_\fm)=\pi_q(v)\qquad(q\neq i).
\]
Since \(T\) is \(\cG\)-gridded, the \(W_i\)-coordinates of the root
children are grid points. Put
\[
        y_\fm:=\pi_i(v_\fm)\in \cG^{(i)}.
\]
Then
\[
        \pi(v_\fm)=w-\pi_i(w)+y_\fm,
\]
and the subtree rooted at \(\fm\) is a \(\cG\)-gridded witness tree of
height at most \(t-1\).  Hence
\[
        v_\fm\in \bZ_{w-\pi_i(w)+y_\fm}^{(t-1)}.
\]
By the induction hypothesis,
\[
        v_\fm\in \bX_{w-\pi_i(w)+y_\fm}^{(t)}.
\]

We now compress equal \(W_i\)-coordinates. Let
\[
        A:=\{y_\fm:\fm\in M\}\subset \cG^{(i)}.
\]
For each \(a\in A\), define
\[
        \lambda_a:=\sum_{\fm:y_\fm=a}a_\fm.
\]
Thus \(\lambda_a>0\) and \(\sum_{a\in A}\lambda_a=1\).  Define
\[
        u_a:=\frac{1}{\lambda_a}
        \sum_{\fm:y_\fm=a}a_\fm \cdot v_\fm.
\]
Since each \(\bX_{w-\pi_i(w)+a}^{(t)}\) is convex, and each
\(v_\fm\) with \(y_\fm=a\) lies in this set, we have
\[
        u_a\in \bX_{w-\pi_i(w)+a}^{(t)}.
\]
Moreover,
\[
        v=\sum_{a\in A}\lambda_a u_a,
\]
and, after applying \(\pi_i\),
\[
        \pi_i(w)=\sum_{a\in A}\lambda_a \cdot a.
\]

Apply Lemma~\ref{lem:new} to the convex representation
\[
        \pi_i(w)=\sum_{a\in A}\lambda_a  \cdot a
\]
in the vector space \(W_i\).  Let \(\mathcal A\) be the collection of
all subsets \(\sigma\subset A\) such that the points of \(\sigma\) are
affinely independent and
\[
        \pi_i(w)\in 
        \convo{\sigma}.
\]
For \(\sigma\in\mathcal A\), let \(b_{\sigma,a}\), \(a\in\sigma\), be
the corresponding barycentric coordinates of \(\pi_i(w)\).  Lemma~\ref{lem:new}
gives coefficients \(\tau_\sigma\geq 0\), \(\sigma\in\mathcal A\), with
\[
        \sum_{\sigma\in\mathcal A}\tau_\sigma=1
\]
and
\[
        \lambda_a
        =
        \sum_{\substack{\sigma\in\mathcal A\\ a\in\sigma}}
        \tau_\sigma b_{\sigma,a}
        \qquad(a\in A).
\]
Hence
\[
\begin{aligned}
        v
        &=
        \sum_{a\in A}\lambda_a \cdot u_a                                      \\
        &=
        \sum_{\sigma\in\mathcal A}
        \tau_\sigma \cdot
        \left( \sum_{a\in\sigma}b_{\sigma,a}u_a \right).
\end{aligned}
\]

We claim that for every \(\sigma\in\mathcal A\),
\[
        \sum_{a\in\sigma}b_{\sigma,a} \cdot u_a\in \bX_w^{(t+1)}.
\]

First suppose \(\operatorname{card}(\sigma)=1\).  Then
\[
        \sigma=\{\pi_i(w)\}.
\]
Thus
\[
        \sum_{a\in\sigma}b_{\sigma,a} \cdot u_a
        =
        u_{\pi_i(w)}.
\]
But
\[
        u_{\pi_i(w)}
        \in \bX_{w-\pi_i(w)+\pi_i(w)}^{(t)}
        =
        \bX_w^{(t)}
        \subset \bX_w^{(t+1)}.
\]

Now suppose \(\operatorname{card}(\sigma)\geq 2\).  Since \(\sigma\) is
affinely independent in \(W_i\), we have
\[
        2\leq \operatorname{card}(\sigma)\leq \dim W_i+1.
\]
The grid points \(a\in\sigma\), the barycentric coordinates
\(b_{\sigma,a}>0\), and the identity
\[
        \pi_i(w)=\sum_{a\in\sigma}b_{\sigma,a} \cdot a
\]
therefore define one of the inequalities \eqref{eqn:I.2} of Definition~\ref{def:cI}, namely
\[
        X_w
        \geq
        \sum_{a\in\sigma}
        b_{\sigma,a}\,
        X_{w-\pi_i(w)+a}.
\]
Using this inequality in the iterative construction of Definition~\ref{def:itr},
we obtain
\[
        \sum_{a\in\sigma}
        b_{\sigma,a}\,
        \bX_{w-\pi_i(w)+a}^{(t)}
        \subset
        \bX_w^{(t+1)}.
\]
Since
\[
        u_a\in \bX_{w-\pi_i(w)+a}^{(t)}
        \qquad(a\in\sigma),
\]
it follows that
\[
        \sum_{a\in\sigma}b_{\sigma,a} \cdot u_a\in X_w^{(t+1)}.
\]

Thus every point
\[
        \sum_{a\in\sigma}b_{\sigma,a} \cdot u_a
\]
belongs to \(\bX_w^{(t+1)}\).  Since \(\bX_w^{(t+1)}\) is convex and
\[
        v=
        \sum_{\sigma\in\mathcal A}
        \tau_\sigma \cdot
        \left(\sum_{a\in\sigma}b_{\sigma,a} \cdot u_a \right),
\]
we conclude that
\[
        v\in \bX_w^{(t+1)}.
\]
This completes the induction and proves
\[
        H_w\subset \bX_w^{(*)}.
\]

\medskip
\noindent
\textbf{Second inclusion: \(\bX_w^{(*)}\subset H_w\).}

It is enough to show that the tuple
\[
        (H_w)_{w\in G}
\]
is a solution of the system \(I\). Since \(\bX^{(*)}\) is the smallest
solution of \(I\), this will imply
\[
        \bX_w^{(*)}\subset H_w
        \qquad(w\in \cG).
\]

First we show that each \(H_w\) is convex. Let
\[
        v,v'\in H_w.
\]
Then there exist \(t,t'\in\mathbb N\) such that
\[
        v\in G_\Lambda^{(t)}(S)_w,
        \qquad
        v'\in G_\Lambda^{(t')}(S)_w.
\]
Assume without loss of generality that \(t\geq t'\). Since
\(0\in\Lambda\), the sets \(G_\Lambda^{(r)}(S)\) are increasing, and
therefore
\[
        v'\in G_\Lambda^{(t)}(S)_w.
\]
Moreover,
\[
        \pi(v)=\pi(v')=w,
\]
so
\[
        v-v'\in U\subset\Lambda.
\]
Thus every strict convex combination of \(v\) and \(v'\) belongs to
\(G_\Lambda^{(t+1)}(S)_w\), and the endpoints already belong to
\(G_\Lambda^{(t)}(S)_w\).  Hence
\[
        [v,v']\subset H_w.
\]
So \(H_w\) is convex.

We now verify the inequalities of \(I\).

For inequalities of type \eqref{eqn:I.1}, let \(v\in S\) and put
\[
        w=\pi(v).
\]
Then
\[
        v\in G_\Lambda^{(0)}(S)\subset G_\Lambda^{(\infty)}(S),
\]
and therefore
\[
        v\in H_w.
\]
Thus \(H_{\pi(v)}\geq v\).

Now consider an inequality of type \eqref{eqn:I.2}.  Thus we have a direction
\(i\), grid points
\[
        w_i^{(0)},w_i^{(1)},\ldots,w_i^{(p)}\in \cG^{(i)},
\]
grid points \(w_j\in \cG^{(j)}\) for \(j\neq i\), and coefficients
\[
        \alpha^{(1)},\ldots,\alpha^{(p)}>0,\qquad
        \sum_{h=1}^p\alpha^{(h)}=1,
\]
such that
\[
        w_i^{(0)}
        =
        \sum_{h=1}^p\alpha^{(h)}w_i^{(h)}.
\]
Put
\[
        w^{(0)}:=w_i^{(0)}+\sum_{j\neq i}w_j,
        \qquad
        w^{(h)}:=w_i^{(h)}+\sum_{j\neq i}w_j.
\]
We must show that
\[
        \sum_{h=1}^p\alpha^{(h)}H_{w^{(h)}}
        \subset
        H_{w^{(0)}}.
\]

Choose arbitrary
\[
        v^{(h)}\in H_{w^{(h)}},
        \qquad h=1,\ldots,p.
\]
For each \(h\), choose \(t_h\) such that
\[
        v^{(h)}\in G_\Lambda^{(t_h)}(S)_{w^{(h)}}.
\]
Let
\[
        t:=\max_h t_h.
\]
Since \(\mathbf{0} \in\Lambda\), we have
\[
        v^{(h)}\in G_\Lambda^{(t)}(S)_{w^{(h)}}
        \qquad(h=1,\ldots,p).
\]
All the points \(v^{(h)}\) have the same \(W_j\)-coordinates for
\(j\neq i\).  Therefore, for any \(h,h'\),
\[
        v^{(h)}-v^{(h')}\in U+W_i\subset\Lambda.
\]
It follows by repeated applications of the definition of
\(G_\Lambda\) that every convex combination of
\[
        v^{(1)},\ldots,v^{(p)}
\]
belongs to \(G_\Lambda^{(\infty)}(S)\).  In particular,
\[
        v:=\sum_{h=1}^p\alpha^{(h)}v^{(h)}
        \in G_\Lambda^{(\infty)}(S).
\]
Moreover,
\[
\begin{aligned}
        \pi(v)
        &=
        \sum_{h=1}^p\alpha^{(h)} \cdot \pi(v^{(h)})        \\
        &=
        \sum_{h=1}^p\alpha^{(h)} \cdot 
        \left(
        w_i^{(h)}+\sum_{j\neq i}w_j
        \right)                                      \\
        &=
        w_i^{(0)}+\sum_{j\neq i}w_j
        =
        w^{(0)}.
\end{aligned}
\]
Thus
\[
        v\in H_{w^{(0)}}.
\]
This proves
\[
        \sum_{h=1}^p\alpha^{(h)}H_{w^{(h)}}
        \subset
        H_{w^{(0)}}.
\]
So the tuple \((H_w)_{w\in \cG}\) satisfies all inequalities of type
\eqref{eqn:I.2}.

Hence \((H_w)_{w\in \cG}\) is a solution of \(I\).  Since
\((\bX_w^{(*)})_{w\in \cG}\) is the smallest solution of \(I\),
\[
        \bX_w^{(*)}\subset H_w
        \qquad(w\in \cG).
\]

Combining the two inclusions gives
\[
        G_\Lambda^{(\infty)}(S)_w=H_w=\bX_w^{(*)}
        \qquad(w\in \cG).
\]
\end{proof}

\subsection{Construction of $G^{(\infty)}_\Lambda(S)$ from $G^{(\infty)}_\Lambda(S)_{\cG}$}
\label{subsec:interpolation}

\begin{proposition}[$G^{(\infty)}_\Lambda(S)$ can be reconstructed from $G^{(\infty)}_\Lambda(S)_{\cG}$ for any sufficiently fine grid $\cG$]
\label{prop:struc}
 Let $\cG$ be a sufficiently fine grid for a finite subset $S \subset V$. 
Then for each subset 
$\cI \subset [1,k]$ 
the set
\[
\{v \in G^{(\infty)}_\Lambda(S) \;:\; \cI_\cG(v) = \cI \}
\]
is a semi-algebraic subset of $V$.
In particular, 
\[
G^{(\infty)}_\Lambda(S) = \bigcup_{\cI \subset [1,k]} \{v \in G^{(\infty)}_\Lambda(S) \;:\; \cI_\cG(v) = \cI \}
\]
is a semi-algebraic subset of $V$.

Moreover, there exists an algorithm which, given a semi-algebraic description of $G^{(\infty)}_\Lambda(S)_{\cG}$, produces as output a semi-algebraic description of $G^{(\infty)}_\Lambda(S)$. 
\end{proposition}

\begin{proof}
    We denote 
    \[
    Z_{\cI} = \{v \in G^{(\infty)}_\Lambda(S) \;:\; \cI_\cG(v) = \cI \}
    \]

Let $\cI = \{i_1,\ldots,i_\ell\}$. 
For each permutation $\Pi$ of $\cI$, denote by 
\[
Z_{\cI,\Pi} = 
\{v \in V \; : \; \cI_\cG(v) = \cI \text{ and there exists $(\Pi,\cG)$-gridded witness tree for $v$}\}.
\]
Then, 
\[
Z_{\cI} = \bigcup_{\Pi} Z_{\cI,\Pi}.
\]

We now prove by induction on $\card(\cI)$ that each $Z_{\cI,\Pi}$ is a semi-algebraic set, which will
imply that $Z_{\cI}$ is semi-algebraic as well.

If $\card(\cI)= 0$, then
    the set 
    \[
Z_{\emptyset} = \bigcup_{w \in \cG} G^{(\infty)}_\Lambda(S)_w.
\] 
 It follows from Proposition~\ref{prop:reduction-lam-system} and Theorem~\ref{thm:hemihedra} that for each $w \in \cG$,
 $G^{(\infty)}_\Lambda(S)_w$ is a hemihedron and hence a semi-algebraic set.
 Since $\cG$ is a finite set, it follows that
 $\bigcup_{w \in \cG} G^{(\infty)}_\Lambda(S)_w$ is also a semi-algebraic set.

Assume now that $Z_{\cI',\Pi'}$ is semi-algebraic for all $\cI'$ with 
$\card(\cI') < \card(\cI)$ and all orderings $\Pi'$ of $\cI'$.

Without loss of generality assume that $\cI = \{1,\ldots,\ell\}$ and 
$\Pi = (1,\ldots,\ell)$. Note that each $v \in Z_{\cI,\Pi}$
admits a $(\Pi,\cG)$-gridded witness tree and 
since the root split of a $(\Pi,\cG)$-gridded witness tree is in direction 1, in this tree every child $\mathfrak{n}'$ of the root satisfies 
$\pi_q(v(\mathfrak{n}')) = \pi_q(v), 2 \leq q \leq k$. This implies in particular, that $\cI_\cG(v(\mathfrak{n}')) = \{2,\ldots,\ell\}$, and
that the subtree rooted at $\mathfrak{n}'$ is a $(\Pi',\cG)$-gridded witness tree for $v(\mathfrak{n}')$ where $\Pi' = (2,\ldots,\ell)$. 
Moreover, $v$ is a convex combination of the $v(\mathfrak{n}')$'s and hence
by Caratheodory's theorem there exists at most $\dim(U+W_1) +1 = \dim U + \dim W_1 +1$ many children such that $v$ is a convex combination of 
$v(\cdot)$ of these children.

We deduce that: 
$v \in Z_{\cI,\Pi}$, if and only if,  $\cI_\cG(v) = \cI$ and there exists $\lambda_1,\ldots,\lambda_N \geq 0$,
and $v_1,\ldots, v_N \in Z_{\cI',\Pi'}$,
where $N = \dim U + \dim W_1 +1$,
$\cI' = \{2,\ldots,\ell\}$ and $ \Pi'=(2,\ldots,\ell)$,
such that
\[
v = \sum_{i=1}^{N} \lambda_i  \cdot v_i,
\]
\[
\sum_{i=1}^N \lambda_i  = 1,
\]
and 
\[
\pi_q(v) = \pi_q(v_j), 2 \leq q \leq k, 1 \leq j \leq N.
\]

Note that we can discard all indices $j$ for which $\lambda_j = 0$.
and use the remaining $v_j$'s 
as the children of the root.
At least two positive coefficients necessarily remain. If only one positive coefficient remained, then $v=v_j$, but $\cI_\cG(v) =I$, $\cI_\cG(v_j) = \cI - \{1\}$ which is impossible.

It follows from the induction hypothesis that 
$Z_{\cI',\Pi'}$ is semi-algebraic. 
Observe that the predicate $\cI_\cG(v) = \cI$ can be expressed as a first-order formula involving linear equalities (using the fact that $\cG$ is a finite set). 
It now follows from the Tarski-Seidenberg theorem \cite{BPRbook2} that $Z_{\cI,\Pi}$ is semi-algebraic as well.

Moreover, it follows from effective quantifier elimination in the theory of reals (see for example \cite[Chapter 14]{BPRbook2}),  
that a semi-algebraic description of the various $Z_{\cI,\Pi}$ can be computed 
given semi-algebraic descriptions of the finite union
\[
Z_\emptyset=G_\Lambda^{(\infty)}(S)_\cG
=
\bigcup_{w\in \cG}G_\Lambda^{(\infty)}(S)_w,
\]
or equivalently of the individual hemihedral fibers.
This proves the proposition.
\end{proof}

\begin{proposition}
\label{prop:laminar}
 
If \(S\subset V\) is finite and admits a sufficiently fine grid \(\cG\),
then
\(G_\Lambda^{(\infty)}(S)\) is semi-algebraic.
Moreover, there exists an algorithm that, given $S$, as well
as a sufficiently fine grid $\cG$ containing $\pi(S)$ as input, 
produces a semi-algebraic description of $G_\Lambda^{(\infty)}(S)$ as its output.
\end{proposition}

\begin{proof}
By Proposition~\ref{prop:reduction-lam-system}, the grid fibers \(G_\Lambda^{(\infty)}(S)_w\), \(w\in \cG\),
are the coordinate sets of the smallest solution of the finite system \(I_{\cG, S}\).
By Theorem~\ref{thm:hemihedra}, these fibers are hemihedra and effective semi-linear
descriptions of them can be computed. Proposition~\ref{prop:struc} then reconstructs
\(G_\Lambda^{(\infty)}(S)\) semi-algebraically from these fiber descriptions.
\end{proof}

\subsection{Existence of sufficiently fine grids}
\label{subsec:existence-of-grid}
We are unable to prove the existence of sufficiently fine grids in general.
However, in the special case when $\dim W_i = 1$ for each $i, 1\leq i \leq k$
we have the following proposition.

\begin{proposition}
    \label{prop:existence-of-grid}
    Suppose $\dim W_i = 1$ for each $i, 1\leq i\leq k$, and let $S$ be a finite subset
    of $V$. Then $\cG = \pi_1(S) \times \cdots \times \pi_k(S)$ is a sufficiently
    fine grid for $S$.
\end{proposition}

For the rest of this section, we fix 
\[
\cG = \pi_1(S) \times \cdots \times \pi_k(S).
\]

In order to prove Proposition~\ref{prop:existence-of-grid} we need to prove that every $v \in G^{(\infty)}_\Lambda(S)$ admits a $\cG$-gridded witness.

\begin{notation}
\label{not:immediate}
Let \(\Delta^{(i)}\) denote the set of all non-empty affinely independent
subsets of \(G^{(i)}\).
    For $w_i \in \conv{\pi_i(S)}$, we denote by $\imm_i(w_i) \in \Delta^{(i)}$ the unique element of $\Delta^{(i)}$ having the property that $w_i \in \convo{\imm_i(w_i)}$ and 
    $\convo{\imm_i(w_i)} \subset \convo{\sigma^{(i)}}$ for any
    $\sigma^{(i)} \in \Delta^{(i)}$ with $w_i \in \convo{\sigma^{(i)}}$.
    \end{notation}
\begin{remark}
\label{rem:immediate}
   Notice that $\imm_i(w_i) = \{w_i\}$ if $w_i \in \cG^{(i)}$ and 
    $\card(\imm_i(w_i)) = 2$ if $w_i \not\in \cG^{(i)}$. 
\end{remark}

\begin{definition}[$i$-immediate]
    Let $T$ be a witness tree and $\mathfrak{n} \in \mathrm{nodes}(T)$ with
    $i(\mathfrak{n}) = i$ and $\pi_i(v(\mathfrak{n})) \not\in \cG^{(i)}$. 
    We say that $\mathfrak{n}$ satisfies the $i$-immediacy property (or $\mathfrak{n}$ is $i$-immediate) if
    for $\mathfrak{n}' \in \mathrm{children}(\mathfrak{n})$, 
 \[
 \pi_{i}(v(\mathfrak{n}')) \in \imm_{i} (\pi_i(v(\mathfrak{n}))),
 \]
and the map
\[
\mathrm{children}(\mathfrak{n}) \rightarrow \imm_{i}(\pi_i(v(\mathfrak{n}))),\;
\mathfrak{n}' \mapsto \pi_{i}(v(\mathfrak{n}'))
\]
is bijective.
\end{definition}

\begin{definition}[$(\Pi,\cG)$-gridded immediate witness trees]
Let $v \in \cG^{(\infty)}_\Lambda(S)$ and $\Pi=(i_1,\ldots,i_\ell)$, where $\cI_\cG(v) = \{i_1,\ldots,i_\ell\}$.
 We say that a witness tree $T$ for $v$ is a $(\Pi,\cG)$-gridded,  
 immediate witness tree for $v$,
 if $T$ is a $(\Pi,\cG)$-gridded witness tree for $v$, and
 for each $j, 0 \leq j \leq \ell-1$,
 each $\mathfrak{n} \in \mathrm{nodes}(T)$ with $\mathrm{depth}_T(\mathfrak{n}) = j$, and 
  $\mathfrak{n}' \in \mathrm{children}(\mathfrak{n})$, 
 \[
 \pi_{i_{j+1}}(v(\mathfrak{n}')) \in \imm_{i_{j+1}} (\pi_{i_{j+1}}(v(\mathfrak{n}))),
 \]
and the map
\[
\mathrm{children}(\mathfrak{n}) \rightarrow \imm_{i_{j+1}}(\pi_{i_{j+1}}(v(\mathfrak{n}))),\;
\mathfrak{n}' \mapsto \pi_{i_{j+1}}(v(\mathfrak{n}'))
\]
is bijective
(in other words $\mathfrak{n}$ satisfies the $i_{j+1}$-immediacy property). 
\end{definition}

We need the following lemmas.
Lemma~\ref{lem:root-immediacy-repair} repairs immediacy at the root while preserving the already immediate lower levels. Lemma~\ref{lem:permute} permits adjacent transpositions of the prescribed coordinate order. Proposition~\ref{prop:gridded-witness} combines these two operations with induction on minimal witness-tree height.

\begin{lemma}[Root-immediacy repair]
\label{lem:root-immediacy-repair}
Assume that $\dim W_i=1$ for all $1\leq i\leq k$.
Let $\Pi=(i_1,\ldots,i_\ell)$, and let $T$ be a
$(\Pi,\cG)$-gridded witness tree for a point $v$.
Assume that every node $n\in \mathrm{nodes}(T)$ with
\[
  0<d:=\mathrm{depth}_T(n)\leq \ell-1
\]
is $i_{d+1}$-immediate.  Then there exists a
$(\Pi,\cG)$-gridded immediate witness tree $\widetilde T$ for $v$.
\end{lemma}

\begin{proof}
We give the proof after relabelling the coordinates so that
\[
  \Pi=(1,\ldots,\ell).
\]
The general case is identical, replacing the coordinate $q$ below by
$i_q$ everywhere.

If $\ell=0$, then $T$ is already $(\Pi,\cG)$-gridded immediate, since
there are no prescribed levels at which immediacy must be checked.
Thus we may assume $\ell\geq 1$.

Let
\[
  \fr=\mathrm{root}(T), \qquad N=\mathrm{children}_T(\fr).
\]
Write
\[
  x_\fm=v(\fm),\qquad a_\fm=\alpha(\fm),\qquad \fm\in N.
\]
Since $T$ is a witness tree and $i(\fr)=1$, we have
\[
  v=\sum_{\fm\in N} a_\fm x_\fm,
\]
and for every $q\neq 1$ and every $\fm\in N$,
\[
  \pi_q(x_\fm)=\pi_q(v).
\]
Since $T$ is $(\Pi,\cG)$-gridded, every root child lies at
depth $1$, and therefore its first coordinate is grid-valued:
\[
  \pi_1(x_\fm)\in \cG^{(1)}\qquad (\fm\in N).
\]

Set
\[
  w_1:=\pi_1(v).
\]
Because $1\in \cI_{\cG}(v)$, we have $w_1\notin \cG^{(1)}$. Hence, since
$\dim W_1=1$, the immediate cell
\[
  \imm_1(w_1)
\]
has two elements. Write
\[
  \imm_1(w_1)=\{w_1^{(1)},w_1^{(2)}\}.
\]
There are unique positive numbers $\mu^{(1)},\mu^{(2)}$ such that
\[
  w_1=\mu^{(1)} w_1^{(1)}+\mu^{(2)}w_1^{(2)},\qquad
  \mu^{(1)}+\mu^{(2)}=1.
\]

We now decompose the old root weights conditionally over the two
immediate endpoints. Let
\[
  y_\fm:=\pi_1(x_\fm)\in \cG^{(1)}.
\]
Since
\[
  w_1=\sum_{\fm\in N} a_\fm y_\fm,
\]
we may apply Lemma~\ref{lem:new} to this convex representation.  Let
$\mathcal A$ be the collection of non-empty subsets
$\sigma\subseteq N$ such that the labeled set
$\{y_m:m\in\sigma\}$ is affinely independent and
\[
  w_1\in 
  \convo{\{y_\fm \;:\; \fm\in\sigma\}}.
\]
For $\fm\notin\sigma$, set $b_\sigma(\,\cdot\,,y_\fm)=0$.  Lemma~\ref{lem:new}  gives
numbers $t_\sigma\geq 0$, $\sigma\in\mathcal A$, satisfying
\[
  \sum_{\sigma\in\mathcal A}t_\sigma=1
\]
and, for each $\fm\in N$,
\begin{equation}\label{eq:old-weight-decomp}
  a_\fm
  =
  \sum_{\sigma\in\mathcal A}
  t_\sigma\, b_\sigma(w_1,y_\fm).
\end{equation}

For every $\sigma\in\mathcal A$, the points $w_1^{(1)}$ and
$w_1^{(2)}$ belong to
\[
  \conv{\{y_\fm \;:\; \fm\in\sigma\}}.
\]
Indeed, $\sigma$ is a one-dimensional simplex whose grid vertices
bracket $w_1$, and $w_1^{(1)},w_1^{(2)}$ are the two immediate grid
points surrounding $w_1$.

For $j=1,2$ and $m\in N$, define
\begin{equation}\label{eq:nu-def}
  \nu_\fm^{(j)}
  :=
  \sum_{\sigma\in\mathcal A}
  t_\sigma\, b_\sigma(w_1^{(j)},y_\fm).
\end{equation}
Then $\nu_\fm^{(j)}\geq 0$ and
\begin{equation}\label{eq:nu-sums}
  \sum_{\fm\in N}\nu_m^{(j)}=1,\qquad
  \sum_{\fm\in N}\nu_\fm^{(j)}y_\fm=w_1^{(j)}.
\end{equation}
Moreover, by affine linearity of barycentric coordinates and the
identity
\[
  w_1=\mu^{(1)}w_1^{(1)}+\mu^{(2)}w_1^{(2)},
\]
we have, for each $\fm\in N$,
\begin{align}
  a_\fm
  &=
  \sum_{\sigma\in\mathcal A}
  t_\sigma\,b_\sigma(w_1,y_\fm) \notag\\
  &=
  \sum_{\sigma\in\mathcal A}
  t_\sigma\,
  b_\sigma\bigl(\mu^{(1)}w_1^{(1)}+\mu^{(2)}w_1^{(2)},y_\fm\bigr)
  \notag\\
  &=
  \mu^{(1)}
  \sum_{\sigma\in\mathcal A}
  t_\sigma\,b_\sigma(w_1^{(1)},y_\fm)
  +
  \mu^{(2)}
  \sum_{\sigma\in\mathcal A}
  t_\sigma\,b_\sigma(w_1^{(2)},y_\fm) \notag\\
  &=
  \mu^{(1)}\nu_\fm^{(1)}+\mu^{(2)}\nu_\fm^{(2)}.
  \label{eq:a-nu}
\end{align}

We now construct the new tree $\widetilde T$.

First create a new root $\widetilde{\fr}$ with
\[
  v(\widetilde{\fr})=v,\qquad i(\widetilde{\fr})=1.
\]
Its two children are denoted $\widetilde{\fp}^{(1)}_{\emptyset}$ and
$\widetilde{\fp}^{(2)}_{\emptyset}$.  We set
\[
  \alpha(\widetilde{\fp}^{(j)}_{\emptyset})=\mu^{(j)}
\]
and
\begin{equation}\label{eq:root-child-value}
  v(\widetilde{\fp}^{(j)}_{\emptyset})
  :=
  \sum_{\fm\in N}\nu_\fm^{(j)}v(\fm),
  \qquad j=1,2.
\end{equation}
Then the new root convex identity holds:
\begin{align*}
  \sum_{j=1}^2 \alpha(\widetilde{\fp}^{(j)}_{\emptyset})
  v(\widetilde{\fp}^{(j)}_{\emptyset})
  &=
  \sum_{j=1}^2 \mu^{(j)}
  \sum_{\fm\in N}\nu_\fm^{(j)}v(\fm)  \\
  &=
  \sum_{\fm\in N}
  \bigl(\mu^{(1)}\nu_\fm^{(1)}+\mu^{(2)}\nu_\fm^{(2)}\bigr)v(\fm) \\
  &=
  \sum_{\fm\in N}a_\fm \cdot v(\fm)=v.
\end{align*}
Also, by \eqref{eq:nu-sums},
\[
  \pi_1(v(\widetilde{\fp}^{(j)}_{\emptyset}))=w_1^{(j)}.
\]
For $q\neq 1$, since all root children $m$ agree with $v$ in the
$W_q$-coordinate,
\[
  \pi_q(v(\widetilde{\fp}^{(j)}_{\emptyset}))
  =
  \sum_{\fm\in N}\nu_\fm^{(j)}\pi_q(v(\fm))
  =
  \pi_q(v).
\]
Thus the new root is a valid witness-tree split in direction $1$.
Furthermore, the map
\[
  \{\widetilde{\fp}^{(1)}_{\emptyset},\widetilde{\fp}^{(2)}_{\emptyset}\}
  \longrightarrow
  \imm_1(w_1),
  \qquad
  \widetilde{\fp}^{(j)}_{\emptyset}\longmapsto
  \pi_1(v(\widetilde{\fp}^{(j)}_{\emptyset}))
\]
is bijective. Hence the new root is $1$-immediate.

It remains to construct the subtrees below
$\widetilde{\fp}^{(1)}_{\emptyset}$ and $\widetilde{\fp}^{(2)}_{\emptyset}$.  This is done
synchronously by averaging corresponding nodes in the old subtrees
rooted at the old children $m\in N$.

We define, recursively, a family of words $\mathcal W_d$ for
$1\leq d\leq \ell$.  Let
\[
  \mathcal W_1:=\{\emptyset\},
\]
and for $\fm\in N$ set
\[
  \fm_{\emptyset}:=\fm.
\]
Suppose $\mathcal W_d$ has been defined for some $1\leq d<\ell$, and
let $\eta\in\mathcal W_d$.  The node $m_\eta$ has depth $d$ in the
old tree $T$.  Since $d\geq 1$ and $d\leq \ell-1$, the hypothesis
implies that $m_\eta$ is $(d+1)$-immediate.

We first observe the following invariant:

\[
  \text{for fixed }d,\eta,\text{ the points }v(\fm_\eta),\ \fm\in N,
  \text{ agree in every }W_q\text{-coordinate with }q\neq 1.
\]

For $d=1$, this is exactly the projection condition for the original
root split.  If the invariant holds for $d$, then all $\fm_\eta$ have
the same $W_{d+1}$-coordinate.  Hence the immediate cell
\[
  E_{d+1}(\eta):=
  \imm_{d+1}\bigl(\pi_{d+1}(v(\fm_\eta))\bigr)
\]
is independent of $\fm$.  Since $\fm_\eta$ is $(d+1)$-immediate, for each
$e\in E_{d+1}(\eta)$ there is a unique child of $\fm_\eta$, denoted
$m_{\eta,e}$, such that
\[
  \pi_{d+1}(v(\fm_{\eta,e}))=e.
\]
Define
\[
  \mathcal W_{d+1}
  :=
  \{(\eta,e):\eta\in\mathcal W_d,\ e\in E_{d+1}(\eta)\}.
\]
The invariant is preserved at level $d+1$: if $q=d+1$, then the
$W_q$-coordinate of all $v(\fm_{\eta,e})$ is the common endpoint $e$;
if $q\neq 1,d+1$, then the old witness-tree projection condition for
the split $\fm_\eta\to \fm_{\eta,e}$ gives
\[
  \pi_q(v(\fm_{\eta,e}))=\pi_q(v(\fm_\eta)),
\]
which is independent of $\fm$ by the induction hypothesis.

Now, for each $j=1,2$, each $1\leq d\leq \ell$, and each
$\eta\in\mathcal W_d$, introduce a new node $\widetilde{\fp}^{(j)}_\eta$ and define
\begin{equation}\label{eq:z-def}
  v(\widetilde{\fp}^{(j)}_\eta)
  :=
  \sum_{\fm\in N}\nu_\fm^{(j)}\cdot v(\fm_\eta).
\end{equation}
This agrees with \eqref{eq:root-child-value} when $d=1$ and
$\eta=\emptyset$.

If $d<\ell$, then $\widetilde{\fp}^{(j)}_\eta$ is declared to be an internal node
with
\[
  i(\widetilde{\fp}^{(j)}_\eta)=d+1.
\]
Its children are the nodes
\[
  \widetilde{\fp}^{(j)}_{\eta,e},
  \qquad e\in E_{d+1}(\eta).
\]
Let $\beta_{\eta,e}$ be the barycentric coordinate of the common point
$\pi_{d+1}(v(\fm_\eta))$ with respect to the immediate cell
$E_{d+1}(\eta)$.  These numbers are independent of $m$, and they are
precisely the weights of the old immediate split
\[
  \fm_\eta\longrightarrow \{\fm_{\eta,e}:e\in E_{d+1}(\eta)\}.
\]
Therefore
\[
  v(\fm_\eta)
  =
  \sum_{e\in E_{d+1}(\eta)}
  \beta_{\eta,e}v(\fm_{\eta,e})
\]
for every $\fm\in N$.  We set
\[
  \alpha(\widetilde{\fp}^{(j)}_{\eta,e})=\beta_{\eta,e}.
\]
Then
\begin{align*}
  \sum_{e\in E_{d+1}(\eta)}
  \alpha(\widetilde{\fp}^{(j)}_{\eta,e})v(\widetilde{\fp}^{(j)}_{\eta,e})
  &=
  \sum_{e\in E_{d+1}(\eta)}
  \beta_{\eta,e}
  \sum_{\fm\in N}\nu_\fm^{(j)}v(\fm_{\eta,e})       \\
  &=
  \sum_{\fm\in N}\nu_\fm^{(j)}
  \sum_{e\in E_{d+1}(\eta)}
  \beta_{\eta,e}v(\fm_{\eta,e})                 \\
  &=
  \sum_{\fm\in N}\nu_\fm^{(j)}v(\fm_\eta)            \\
  &=
  v(\widetilde{\fp}^{(j)}_\eta).
\end{align*}
Thus the convex identity holds at $\widetilde{\fp}^{(j)}_\eta$.

The projection condition also holds. Let $q\neq d+1$.  Then, for
each $m$ and each $e\in E_{d+1}(\eta)$,
\[
  \pi_q(v(\fm_{\eta,e}))=\pi_q(v(\fm_\eta)).
\]
Averaging over $\fm$ gives
\[
  \pi_q(v(\widetilde{\fp}^{(j)}_{\eta,e}))
  =
  \pi_q(v(\widetilde{\fp}^{(j)}_\eta)).
\]
So the split at $\widetilde{\fp}^{(j)}_\eta$ is a valid witness-tree split in
direction $d+1$.

Moreover,
\[
  \pi_{d+1}(v(\widetilde{\fp}^{(j)}_{\eta,e}))=e.
\]
Hence the map
\[
  \mathrm{children}(\widetilde{\fp}^{(j)}_\eta)\longrightarrow E_{d+1}(\eta),
  \qquad
  \widetilde{\fp}^{(j)}_{\eta,e}\longmapsto
  \pi_{d+1}(v(\widetilde{\fp}^{(j)}_{\eta,e}))
\]
is bijective.  Thus $\widetilde{\fp}^{(j)}_\eta$ is $(d+1)$-immediate.

After carrying out this construction for $d=1,\ldots,\ell-1$, we have
constructed all prescribed levels.  The nodes $\widetilde{\fp}^{(j)}_\eta$ with
$\eta\in\mathcal W_\ell$ lie at depth $\ell$.  They have all
$W$-coordinates in the grid:
coordinate $1$ is $w_1^{(j)}\in G^{(1)}$, coordinates
$2,\ldots,\ell$ are immediate endpoints, and all coordinates outside
$\{1,\ldots,\ell\}$ were already grid-valued in the original
$(\Pi,\cG)$-gridded tree.

We now attach the old lower subtrees. Fix $j\in\{1,2\}$ and
$\eta\in\mathcal W_\ell$.  Let
\[
  N_j:=\{\fm\in N:\nu_\fm^{(j)}>0\}.
\]
Declare $\widetilde{\fp}^{(j)}_\eta$ to be an internal node with index $1$, and give
it children $\widetilde{\fq}^{(j)}_{\eta,\fm}$ indexed by $\fm\in N_j$.  Set
\[
  \alpha(\widetilde{\fq}^{(j)}_{\eta,\fm})=\nu_\fm^{(j)},\qquad
  v(\widetilde{\fq}^{(j)}_{\eta,\fm})=v(\fm_\eta).
\]
Then
\[
  v(\widetilde{\fp}^{(j)}_\eta)
  =
  \sum_{\fm\in N_j}\nu_\fm^{(j)}v(\fm_\eta).
\]
For fixed $\eta$, the old nodes $\fm_\eta$, $\fm\in N$, agree in every
coordinate except possibly $W_1$.  Therefore this is a valid
witness-tree split in direction $1$.

Below each $\widetilde{\fq}^{(j)}_{\eta,\fm}$ we copy the old subtree of $T$ rooted
at $m_\eta$, with the only change that the incoming weight of the root
copy is now $\nu_\fm^{(j)}$. All outgoing weights, node values, node
indices, and descendant subtrees are copied from $T$. Leaves are
therefore still labeled by points of $S$.

This completes the construction of $\widetilde T$.

We verify that $\widetilde T$ has the required properties.

First, $\widetilde T$ is a witness tree. The root convex identity and
projection condition were checked above.  The same checks were carried
out for every synchronized node $\widetilde{\fp}^{(j)}_\eta$ at depths
$1,\ldots,\ell-1$.  The bottom attachment is a valid split in
coordinate $1$, and all lower copied subtrees inherit the witness-tree
identities from $T$.  The tree is finite because $T$ is finite and the
construction introduces only finitely many new averaged nodes.

Second, $\widetilde T$ is $(\Pi,\cG)$-gridded.  At depth $0$, this
holds because $\cI_\cG(v)=\{1,\ldots,\ell\}$.  At depth $d$ with
$1\leq d\leq \ell$, the construction has already fixed coordinates
$1,\ldots,d$ to grid endpoints, and all coordinates outside
$\{1,\ldots,\ell\}$ were grid-valued from the original gridded tree.
Thus every coordinate outside the remaining set
\[
  \{d+1,\ldots,\ell\}
\]
is grid-valued. For $d\leq \ell-1$, the node index is $d+1$ by
construction. At depths $d\geq \ell$, all coordinates are
grid-valued, either by the construction above or because the lower
subtree was copied from nodes of the original gridded tree at depth at
least $\ell$.

Third, $\widetilde T$ is immediate at all prescribed levels. The root
is $1$-immediate by construction.  For each $1\leq d\leq \ell-1$, the
children of every node $\widetilde{\fp}^{(j)}_\eta$ at depth $d$ are indexed
bijectively by the immediate cell in coordinate $d+1$.  Hence every
node at depth $d$ is $(d+1)$-immediate.

Therefore $\widetilde T$ is a $(\Pi,\cG)$-gridded immediate witness
tree for $v$.
\end{proof}

\begin{lemma}[Swap Lemma]
    \label{lem:permute}
    Suppose that $\dim W_i = 1, 1 \leq i \leq k$ and $S \subset V$ a finite subset.
    Let $v \in G^{(\infty)}_\Lambda(S)$ such that $\cI_\cG(v) =\{i_1,\ldots,i_\ell\}$ and let $\Pi=(i_1,\ldots,i_\ell)$.
    Suppose there exists a $(\Pi,\cG)$-gridded immediate witness tree for $v$. Then there exists a $(\Pi',\cG)$-gridded immediate witness tree for $v$ for every ordered tuple $\Pi'= (i_{\pi(1)},\ldots,i_{\pi(\ell)})$, where $\pi \in \mathfrak{S}_\ell$ (where $\mathfrak{S}_\ell$ denotes the symmetric group on $[1,\ell]$).
\end{lemma}

\begin{proof}
It is enough to prove the statement when $\Pi'$ is obtained from $\Pi$
by one adjacent transposition. Indeed, adjacent transpositions generate
the symmetric group $\mathfrak{S}_\ell$.

After relabelling the coordinates, assume
\[
        \Pi=(1,\ldots,\ell).
\]
Fix $p$, $1\le p<\ell$, and set
\[
        \Pi'=(1,\ldots,p-1,p+1,p,p+2,\ldots,\ell).
\]
Let $T$ be a $(\Pi,\mathcal G)$-gridded immediate witness tree for $v$.
We construct a $(\Pi',\mathcal G)$-gridded immediate witness tree $T'$.

The tree $T'$ is identical to $T$ above depth $p-1$. We now describe
the replacement of the two-level block below each node of depth $p-1$.
Let $\mathfrak{n}\in \operatorname{nodes}(T)$ be a node with
\[
        \operatorname{depth}_T(\fn)=p-1.
\]
Since $T$ is $(\Pi,\mathcal G)$-gridded, 
\[
i(\mathfrak{n}) = p.
\]

Since all earlier splits are in the coordinates
$1,\ldots,p-1$, they do not change the $W_p$- or $W_{p+1}$-coordinates.
Thus
\[
        \pi_p(v(\mathfrak{n}))=\pi_p(v) \not\in \cG^{(p)},\qquad
        \pi_{p+1}(v(\mathfrak{n}))=\pi_{p+1}(v) \not\in \cG^{(p+1)}
\]
(since $p,p+1\in I_{\mathcal G}(v)$).  

Hence, 
using Remark~\ref{rem:immediate}
\[
        \operatorname{imm}_p(\pi_p(v(\fn)))=\{a_1,a_2\},
        \qquad
        \operatorname{imm}_{p+1}(\pi_{p+1}(v(\fn)))=\{b_1,b_2\}.
\]

Because $\mathfrak{n}$ is $p$-immediate, its two children may be denoted
\[
        \mathfrak{n}^{(1)},\mathfrak{n}^{(2)}
\]
so that
\[
        \pi_p(v(\mathfrak{n}^{(j)}))=a_j,\qquad j=1,2.
\]
Let
\[
        A_j:=\alpha_T(\mathfrak{n}^{(j)}),\qquad j=1,2.
\]
Thus
\[
        A_1+A_2=1,
        \qquad
        v(\mathfrak{n})=A_1 \cdot v(\mathfrak{n}^{(1)})+A_2\cdot v(\mathfrak{n}^{(2)}).
\]

For each $j=1,2$, the $i(\mathfrak{n}^{(j)}) = p+1$, and 
$\mathfrak{n}^{(j)}$
is $(p+1)$-immediate.  Let its two children be
\[
        \mathfrak{n}^{(j,1)},\mathfrak{n}^{(j,2)},
\]
where
\[
        \pi_{p+1}(v(\mathfrak{n}^{(j,h)}))=b_h,\qquad h=1,2.
\]
Write
\[
        B_{j,h}:=\alpha_T(\mathfrak{n}^{(j,h)}).
\]
The numbers $B_{j,h}$ are independent of $j$. Indeed, the split
$\mathfrak{n}\to \mathfrak{n}^{(j)}$ is in coordinate $p$, 
so all $\mathfrak{n}^{(j)}$'s have the same
$W_{p+1}$-coordinate, namely $\pi_{p+1}(v(\mathfrak{n}))$.  Since each
$\mathfrak{n}^{(j)} \to\{\mathfrak{n}^{(j,1)},\mathfrak{n}^{(j,2)}\}$ is an immediate split in the one-dimensional coordinate $p+1$, the weights are the unique barycentric coordinates of
$\pi_{p+1}(v(\mathfrak{n}))$ with respect to $\{b_1,b_2\}$.  Hence there are
numbers $B_1,B_2>0$, $B_1+B_2=1$, such that
\[
        \alpha_T(\mathfrak{n}^{(j,h)})=B_h
        \qquad
        \text{for all }j,h\in\{1,2\}.
\]

We now replace the two-level block
\[
        \mathfrak{n} \longrightarrow \mathfrak{n}^{(j)} \longrightarrow \mathfrak{n}^{(j,h)}
\]

by the swapped two-level block
\[
        \widetilde{\mathfrak{n}} \longrightarrow \widetilde{\mathfrak{n}}^{(h)} \longrightarrow \widetilde{\mathfrak{n}}^{(j,h)}.
\]

Define a new node $\widetilde{\mathfrak{n}}$ by
\[
        v(\widetilde{\mathfrak{n}})=v(\mathfrak{n}),\qquad i(\widetilde{\mathfrak{n}})=p+1.
\]
If $\mathfrak{n}$ is not the root, set
\[
        \alpha_{T'}(\widetilde{\mathfrak{n}})=\alpha_T(\mathfrak{n}).
\]
If $\mathfrak{n}$ is the root, then $\widetilde{\mathfrak{n}}$ is the root of $T'$ and has no
incoming weight.

The children of $\widetilde{\mathfrak{n}}$ are two new nodes 
$\widetilde{\mathfrak{n}}^{(1)},\widetilde{\mathfrak{n}}^{(2)}$.  Define
\[
        \alpha_{T'}(\widetilde{\mathfrak{n}}^{(h)})=B_h,\qquad h=1,2,
\]
and
\[
        v(\widetilde{\mathfrak{n}}^{(h)})=\sum_{j=1}^2 A_j\,v(\widetilde{\mathfrak{n}}^{(j,h)}).
\]
We declare
\[
        i(\widetilde{\mathfrak{n}}^{(h)})=p.
\]

For each $h=1,2$, the children of $\widetilde{\mathfrak{n}}^{(h)}$ are two nodes
\[
        \widetilde{\mathfrak{n}}^{(1,h)}, \widetilde{\mathfrak{n}}^{(2,h)}.
\]
They are copies of the old nodes $\mathfrak{n}^{(1,h)},\mathfrak{n}^{(2,h)}$, with the incoming weights changed to
\[
        \alpha_{T'}(\widetilde{\mathfrak{n}}^{(j,h)})=A_j.
\]
Set
\[
        v(\widetilde{\mathfrak{n}}^{(j,h)})=v(\mathfrak{n}^{(j,h)}).
\]
Below $\widetilde{\mathfrak{n}}^{(j,h)}$, copy the whole subtree of $T$ rooted at $\mathfrak{n}^{(j,h)}$.
Equivalently, $\widetilde{\mathfrak{n}}^{(j,h)}$ replaces 
$\mathfrak{n}^{(j,h)}$ as the root of that copied
subtree; all descendants, node values, outgoing weights, and node
indices are copied from $T$.  In particular, if $p+1<\ell$, then
$i(\widetilde{\mathfrak{n}}^{(j,h)}) = p+2$, as in the original tree.  If $p+1=\ell$,
then no prescribed next coordinate remains at this depth, so no
additional index assignment is imposed beyond copying the old subtree
data when $\widetilde{\mathfrak{n}}^{(j,h)}$ is non-leaf.

We now verify that the replacement block is a valid witness-tree block.

First, the convex identity at $\widetilde{\mathfrak{n}}$ holds:
\[
\begin{aligned}
        \sum_{h=1}^2 \alpha_{T'} (\widetilde{\mathfrak{n}}^{(h)}) \cdot v(\widetilde{\mathfrak{n}}^{(h)})
        &=
        \sum_{h=1}^2 B_h \cdot 
        \left( \sum_{j=1}^2 A_j \cdot v(\widetilde{\mathfrak{n}}^{(j,h)}) \right)      \\
        &=
        \sum_{j=1}^2 A_j \cdot 
        \left( \sum_{h=1}^2 B_h  \cdot v(\mathfrak{n}^{(j,h)}) \right)      \\
        &=
        \sum_{j=1}^2 A_j \cdot v(\mathfrak{n}^{(j)})           \\
        &=
        v(\fn) \\
        &=v(\widetilde{\mathfrak{n}}).
\end{aligned}
\]
The projection condition at $\widetilde{\mathfrak{n}}$ also holds. Since this new
split is meant to be in direction $p+1$, we must check equality of all
coordinates except $p+1$.

Let $q\neq p,p+1$.  The old splits in directions $p$ and $p+1$ do not
change the $W_q$-coordinate, so
\[
        \pi_q(v(\mathfrak{n}^{(j,h)}))=\pi_q(v(\mathfrak{n}))
\]
for all $j,h$.  Hence
\[
        \pi_q(v(\widetilde{\mathfrak{n}}^{(h)}))=\pi_q(v(\mathfrak{n}))=\pi_q(v(\widetilde{\mathfrak{n}})).
\]
For $q=p$, the old split $\mathfrak{n}^{(j)} \to \mathfrak{n}^{(j,h)}$ is in direction $p+1$, so
\[
        \pi_p(v(\mathfrak{n}^{(j,h)}))=\pi_p(v(\mathfrak{n}^{(j)})).
\]
Therefore
\[
        \pi_p(v(\widetilde{\mathfrak{n}}^{(h)}))
        =
        \sum_{j=1}^2 A_j \cdot \pi_p(v(\mathfrak{n}^{(j)}))
        =
        \pi_p(v(\mathfrak{n}))
        =
        \pi_p(v(\widetilde{\mathfrak{n}})).
\]
Finally,
\[
        \pi_{p+1}(v(\widetilde{\mathfrak{n}}^{(h)}))
        =
        \sum_{j=1}^2 A_j \cdot \pi_{p+1}(v(\mathfrak{n}^{(j,h)}))
        =
        b_h.
\]
Thus the map
\[
        \operatorname{children}_{T'}(\widetilde{\mathfrak{n}})
        \longrightarrow
        \operatorname{imm}_{p+1}(\pi_{p+1}(v(\widetilde{\mathfrak{n}}))),
        \qquad
        \widetilde{\mathfrak{n}}^{(h)} \longmapsto \pi_{p+1}(v(\widetilde{\mathfrak{n}}^{(h)})
\]
is a bijection. Hence $\widetilde{\mathfrak{n}}$ is $(p+1)$-immediate.

Second, fix $h\in\{1,2\}$.  We verify the witness-tree conditions at
$\widetilde{\mathfrak{n}}^{(h)}$.  The convex identity is immediate from the definition:
\[
        v(\widetilde{\mathfrak{n}}^{(h)}
        )=\sum_{j=1}^2 A_j \cdot v(\mathfrak{n}^{(j,h)})
              =\sum_{j=1}^2 \alpha_{T'}(\widetilde{\mathfrak{n}}^{(j,h)}) \cdot v(\widetilde{\mathfrak{n}}^{(j,h)}).
\]

This split is meant to be in direction $p$. Let $q\neq p$.  If
$q=p+1$, then
\[
        \pi_{p+1}(v(\widetilde{\mathfrak{n}}^{(j,h)}))=b_h=\pi_{p+1}(v(\widetilde{\mathfrak{n}}^{(h)})).
\]
If $q\neq p,p+1$, then
\[
        \pi_q(v(\widetilde{\mathfrak{n}}^{(j,h)}))=\pi_q(v(\mathfrak{n}^{(j,h)}))=\pi_q(v(\mathfrak{n})),
\]
and the same equality holds for $v(\widetilde{\mathfrak{n}}^{(h)})$.  Thus
\[
        \pi_q(v(\widetilde{\mathfrak{n}}^{(j,h)}))=\pi_q(v(\widetilde{\mathfrak{n}}^{(h)}))
        \qquad(q\neq p).
\]
Moreover,
\[
        \pi_p(v(\widetilde{\mathfrak{n}}^{(j,h)}))=\pi_p(v(\mathfrak{n}^{(j,h)}))=\pi_p(v(\mathfrak{n}^{(j)})=a_j.
\]
Therefore the map
\[
        \operatorname{children}_{T'}(\widetilde{\mathfrak{n}}^{(h)})
        \longrightarrow
        \operatorname{imm}_{p}(\pi_p(v(\widetilde{\mathfrak{n}}^{(h)}))),
        \qquad
        \widetilde{\mathfrak{n}}^{(j,h)}\longmapsto \pi_p(v(\widetilde{\mathfrak{n}}^{(j,h)}))
\]
is a bijection. Hence each $\widetilde{\mathfrak{n}}^{(h)}$ is $p$-immediate.

All nodes below the $\widetilde{\mathfrak{n}}^{(j,h)}$ are copied from $T$, so all witness-tree
identities and leaf labels below these nodes remain valid. Therefore
the local replacement produces a witness tree.

We perform this replacement independently for every node of depth
$p-1$ in $T$. Since the subtrees rooted at distinct depth-$(p-1)$
nodes are disjoint, these replacements are compatible. The resulting
tree is denoted $T'$.

It remains to verify that $T'$ is $(\Pi',\cG)$-gridded and
immediate.

At all depths $<p-1$, the tree is unchanged, and the required order is
also unchanged.  At depth $p-1$, the new splitting direction is $p+1$,
as required by $\Pi'$, and we have proved $(p+1)$-immediacy.  At depth
$p$, the new splitting direction is $p$, as required by $\Pi'$, and we
have proved $p$-immediacy.

The grid-coordinate condition is also preserved. At depth $p-1$, the
set of remaining coordinates for $\Pi'$ is
\[
        \{p+1,p,p+2,\ldots,\ell\},
\]
which is the same set as
\[
        \{p,p+1,p+2,\ldots,\ell\}.
\]
Thus the grid condition is unchanged at that level. At depth $p$, the
coordinate $p+1$ has already been set to one of the grid points
$b_1,b_2$, and all coordinates outside
\[
        \{p,p+2,\ldots,\ell\}
\]
are therefore grid-valued. At depth $p+1$, both coordinates $p$ and
$p+1$ have been set to grid points. Below depth $p+1$, the copied
subtrees have the same remaining prescribed order
\[
        (p+2,\ldots,\ell)
\]
as they had in the original tree after the old two-level block. Hence
the $(\Pi',\cG)$-gridded condition holds at every node.

Thus $T'$ is a $(\Pi',\cG)$-gridded immediate witness tree for
$v$.

Since every permutation is a product of adjacent transpositions,
iterating the construction proves the lemma for an arbitrary ordered
tuple
\[
        \Pi'=(i_{\pi(1)},\ldots,i_{\pi(\ell)}).
\]
\end{proof}

\begin{proposition}
\label{prop:gridded-witness}
    Suppose that $\dim W_i = 1, 1 \leq i \leq k$ and $S \subset V$ a finite subset.
    Let $v \in G^{(\infty)}_\Lambda(S)$.
    Then for every ordered tuple $\Pi=(i_1,\ldots,i_\ell)$, $1 \leq i_j \leq k$,
    such that $\cI_\cG(v)  =  \{i_1,\ldots,i_\ell\}$,  
    there exists a $(\Pi,\cG)$-gridded 
    immediate witness tree for $v$. 
    \end{proposition}

\begin{proof}
  For $v \in G^{(\infty)}_\Lambda(S)$, we denote by $t_v$ the smallest $t \in \mathbb{N}$, such that 
there exists a witness tree $T$ for $v$ with 
height $t$.
Note that it follows from Lemma~\ref{lem:every-point-has-a-witness-tree} that for every $v \in G^{(\infty)}_\Lambda(S)$, $t_v < \infty$.

  We will prove the proposition using induction on $t_v$.
  
  \begin{enumerate}[1.]
    \item 
    \label{itemlabel:proof:prop:gridded-witness:1}
    Base case for the induction: $t_v = 0$. In this case, $\cI_\cG(v) = \emptyset$, and the tree $T$ with one node $\mathrm{root}(T) = \mathfrak{n}$, with $v(\mathfrak{n}) = v$,
    and $\mathrm{children}(\mathfrak{n}) = \emptyset$, is a $(\Pi=(),\cG)$-gridded immediate witness tree for $v$.

    \item 
    \label{itemlabel:proof:prop:gridded-witness:2}
    Induction hypothesis:  
    suppose that the statement of the proposition holds for all $v' \in G^{(\infty)}_\Lambda(S)$, with $t_{v'} < t_v$.

    \item 
    \label{itemlabel:proof:prop:gridded-witness:3}    
    Inductive step. Without loss of generality, we can assume that 
            \[
            \cI_\cG(v)  = \{1,\ldots,\ell\}, 0 \leq \ell \leq k,
            \]
            and  
            we can also assume without loss of generality that 
            $\Pi=(1,\ldots,\ell)$. This is because after we construct a 
            $(\Pi,\cG)$-gridded immediate witness tree, we can use
            Lemma~\ref{lem:permute} to obtain a $(\Pi',\cG)$-gridded immediate witness tree for any permutation $\Pi' \in \mathfrak{S}_\ell$.
            (Note that we allow $\ell = 0$ which corresponds to the case
            $\cI_\cG(v) = \emptyset$.)
            
            Let $T$ be a witness tree for $v$ with $\mathrm{height}(T) = t_v$, and let $\mathfrak{n} = \mathrm{root}(T)$.
            There are now two cases.
            \begin{enumerate}[\ref{itemlabel:proof:prop:gridded-witness:3}.a]
                \item 
                \label{itemlabel:proof:prop:gridded-witness:1.a.i}
                Case 1, $i = i(\mathfrak{n}) \leq \ell$.
                There exists a disjoint partition of $\mathrm{children}(\mathfrak{n})$ into two subsets $C_0,C_1$, where
                \begin{eqnarray*}
                    C_0 &=& \{\mathfrak{n}' \in \mathrm{children}(\mathfrak{n}) \;:\; i \not\in \cI_\cG(v(\mathfrak{n}'))\}, \\
                    C_1 &=& \{\mathfrak{n}' \in \mathrm{children}(\mathfrak{n}) \;:\; i \in \cI_\cG(v(\mathfrak{n}'))\}.
                \end{eqnarray*}
                For $\mathfrak{n}' \in C_0$, using the induction hypothesis, there exists a $(\Pi',\cG)$-gridded immediate witness tree, $T_{\mathfrak{n}'}$, for $v(\mathfrak{n}')$,
                where $\Pi'= (1,\ldots,\widehat{i}, \ldots, \ell)$ (where $\widehat{\cdot}$ denotes omission).
                For $\mathfrak{n}' \in C_1$, using the induction hypothesis, there exists a $(\Pi,\cG)$-gridded immediate witness tree, $T_{\mathfrak{n}'}$, for $v(\mathfrak{n}')$,
                and for each $\mathfrak{n}'' \in \mathrm{children}(\mathfrak{n}')$,
                the subtree of $T_{\mathfrak{n}'}$ rooted at $\mathfrak{n}''$ is a $(\Pi',\cG)$-gridded, immediate witness tree for $v(\mathfrak{n}'')$, where $\Pi' = (1,\ldots,\widehat{i}, \ldots, \ell) $.
                
                To simplify notation, we will assume below (without loss of generality) that $i=1$ and $\Pi' = (2,\ldots,\ell)$. 
                
                We now define a new witness tree $T'$ for $v$ as follows. 

For each \(\fn'\in C_0\), let \(\mathfrak{r}_{\fn'}\) denote the root of the
inductively constructed tree \(T_{n'}\). Attach \(\mathfrak{r}_{\fn'}\) to the new
root and set
\[
\alpha_{T'}(\mathfrak{r}_{\fn'})=\alpha_T(\fn').
\]

For each \(\fn'\in C_1\) and each
\[
\fn''\in
\operatorname{children}_{T_{\fn'}}
\bigl(\operatorname{root}(T_{\fn'})\bigr),
\]
attach a copy \(\widetilde \fn''\) of \(\fn''\) to the new root and set
\[
v(\widetilde \fn'')=v(\fn''),\qquad
\alpha_{T'}(\widetilde \fn'')
=
\alpha_T(\fn')\,
\alpha_{T_{\fn'}}(\fn'').
\]
If \(\fn''\) is not a leaf, set
\[
i(\widetilde \fn'')=i_{T_{\fn'}}(\fn'').
\]
Below \(\widetilde \fn''\), copy the subtree of \(T_{\fn'}\) rooted at
\(\fn''\), and set

\[
v(\operatorname{root}(T'))=v,
\qquad
i(\operatorname{root}(T'))=1.
\]

We verify that:
\[
\sum_{\fn' \in C_0} \alpha_T(\fn') v(\fn')  + \sum_{\fn' \in C_1}\sum_{\fn''} \alpha_T(\fn') \alpha_{T_{\fn'}}(\fn'')v(\fn'') 
= 
\sum_{\fn' \in C_0} \alpha_T(\fn') v(\fn') +\sum_{\fn' \in C_1} \alpha_T(\fn') v(\fn') = v,
\]
and
\[
\sum_{\fn' \in C_0} \alpha_T(\fn')  + \sum_{\fn' \in C_1}\sum_{\fn''} \alpha_T(\fn') \alpha_{T_{\fn'}}(\fn'')
= 
\sum_{\fn' \in C_0} \alpha_T(\fn')  +\sum_{\fn' \in C_1} \alpha_T(\fn') = 1.
\]

It is easy to check that the tree $T'$ constructed above is a $(\Pi,\cG)$-gridded witness tree for $v$. However, it might still violate the immediacy 
property at the root. 
We now apply Lemma~\ref{lem:root-immediacy-repair} to obtain 
a $(\Pi,\cG)$-gridded immediate witness tree for $v$. This completes
Case~\ref{itemlabel:proof:prop:gridded-witness:1.a.i}.

            \item 
            \label{itemlabel:proof:prop:gridded-witness:1.a.ii}
                Case 2, $i = i(\mathfrak{n}) > \ell $.
Let
\[
        N:=\mathrm{children}_T(\mathfrak{n}), \qquad a_{\mathfrak{m}}:=\alpha_T(\mathfrak{m}),\quad \mathfrak{m} \in N.
\]
Since \(i(\mathfrak{n})=i\), the witness-tree projection condition implies that
\[
        \pi_q(v(\mathfrak{m}))=\pi_q(v)
        \qquad\text{for every }\mathfrak{m} \in N\text{ and every }q\neq i.
\]
Since \(\cI_{\cG}(v)=\{1,\ldots,\ell\}\) and \(i>\ell\), it follows that
\[
        \cI_{\cG}(v(\mathfrak{m}))\in
        \bigl\{\{1,\ldots,\ell\},\{1,\ldots,\ell,i\}\bigr\}.
\]
For each \(\mathfrak{m} \in N\), choose, using the induction hypothesis, a witness tree \(T_{\mathfrak{m}}\) for
\(v(\mathfrak{m})\) as follows:
\[
\begin{cases}
T_{\mathfrak{m}} \text{ is a }((1,\ldots,\ell),\cG)\text{-gridded immediate witness tree},
&\text{if } \cI_{\cG}(v(\mathfrak{m}))=\{1,\ldots,\ell\},\\[2mm]
T_{\mathfrak{m}} \text{ is a }((1,\ldots,\ell,i),\cG)\text{-gridded immediate witness tree},
&\text{if } \cI_{\cG}(v(\mathfrak{m}))=\{1,\ldots,\ell,i\}.
\end{cases}
\]
This is allowed because \(t_{v(\mathfrak{m})}<t_v\) for every root child \(\mathfrak{m}\).

For \(q=1,\ldots,\ell\), write
\[
        \imm_q(\pi_q(v))=\{w_q^{(1)},w_q^{(2)}\},
\]
and let \(\lambda_q^{(1)},\lambda_q^{(2)}>0\) be the unique coefficients such that
\[
        \pi_q(v)
        =
        \lambda_q^{(1)} w_q^{(1)}
        +
        \lambda_q^{(2)} w_q^{(2)},
        \qquad
        \lambda_q^{(1)}+\lambda_q^{(2)}=1.
\]
These coefficients are independent of \(\mathfrak{m}\in N\), because all \(v(\mathfrak{m})\) have the same
\(W_q\)-coordinate as \(v\) for \(q\neq i\).

We now construct a new tree \(\widetilde{T} \). For a word
\[
        \eta=(h_1,\ldots,h_d)\in\{1,2\}^d,\qquad 0\leq d\leq \ell,
\]
define \(\mathfrak{m}_\eta\) to be the node of \(T_{\mathfrak{m}}\) obtained by following, successively, the
\(h_q\)-th immediate child in coordinate \(q\), for \(q=1,\ldots,d\).  Thus
\(\mathfrak{m}_\emptyset=\mathrm{root}(T_{\mathfrak{m}})\), and
\[
        \pi_q(v(\mathfrak{m}_{h_1,\ldots,h_q}))=w_q^{(h_q)}.
\]
This notation is well-defined because each \(T_{\mathfrak{m}}\) is immediate in the first \(\ell\)
coordinates.

For every such word \(\eta\), introduce a new node \(\widetilde{\mathfrak{m}}_\eta\) and set
\[
        v(\widetilde{\mathfrak{m}}_\eta)
        :=
        \sum_{\mathfrak{m} \in N} a_{\mathfrak{m}}\, v(\mathfrak{m}_\eta).
\]
The root of \(\widetilde{T}\) is \(\widetilde{\mathfrak{m}}_\emptyset\).  Since \(\mathfrak{m}_\emptyset=\mathrm{root}(T_{\mathfrak{m}})\) and
\(v(\mathrm{root}(T_{\mathfrak{m}}))=v(\mathfrak{m})\), we have
\[
        v(\widetilde{\mathfrak{m}}_\emptyset)=\sum_{\mathfrak{m}\in N}a_{\mathfrak{m}} v(\mathfrak{m})=v.
\]

For \(0\leq d<\ell\), declare \(\widetilde{\mathfrak{m}}_\eta\), \(|\eta|=d\), to be an internal node with
\[
        i(\widetilde{\mathfrak{m}}_\eta)=d+1,
\]
and with children
\[
        \widetilde{\mathfrak{m}}_{\eta,1},\ \widetilde{\mathfrak{m}}_{\eta,2}.
\]
Set
\[
        \alpha(\widetilde{\mathfrak{m}}_{\eta,h})=\lambda_{d+1}^{(h)},\qquad h=1,2.
\]
We verify that this gives a valid witness-tree split. For every \(\mathfrak{m}\in N\), the old tree
\(T_{\mathfrak{m}}\) has the immediate split
\[
        v(\mathfrak{m}_\eta)
        =
        \sum_{h=1}^2
        \lambda_{d+1}^{(h)} \cdot v(\mathfrak{m}_{\eta,h}),
\]
because the \(W_{d+1}\)-coordinate of \(\mathfrak{m}_\eta\) is the common value
\(\pi_{d+1}(v)\).  Therefore
\[
\begin{aligned}
        \sum_{h=1}^2 \alpha(\widetilde{\mathfrak{m}}_{\eta,h})\cdot v(\widetilde{\mathfrak{m}}_{\eta,h})
        &=
        \sum_{h=1}^2
        \lambda_{d+1}^{(h)} \cdot
        \left(
        \sum_{\mathfrak{m}\in N}a_{\mathfrak{m}} v(\mathfrak{m}_{\eta,h}) 
        \right)\\
        &=
        \sum_{\mathfrak{m}\in N}a_{\mathfrak{m}} \cdot 
        \left(
        \sum_{h=1}^2
        \lambda_{d+1}^{(h)}v(\mathfrak{m}_{\eta,h}) 
        \right)\\
        &=
        \sum_{\mathfrak{m}\in N}a_{\mathfrak{m}} \cdot v(\mathfrak{m}_\eta) \\
        &=
        v(\widetilde{\mathfrak{m}}_\eta).
\end{aligned}
\]
Also, if \(r\neq d+1\), then the old split
\(\mathfrak{m}_\eta\to \mathfrak{m}_{\eta,h}\) is in coordinate \(d+1\), so
\[
        \pi_r(v(\mathfrak{m}_{\eta,h}))=\pi_r(v(\mathfrak{m}_\eta)).
\]
Averaging over \(\mathfrak{m}\) gives
\[
        \pi_r(v(\widetilde{\mathfrak{m}}_{\eta,h}))=\pi_r(v(\widetilde{\mathfrak{m}}_\eta)).
\]
Thus \(\widetilde{\mathfrak{m}}_\eta\to \widetilde{\mathfrak{m}}_{\eta,h}\) is a valid witness-tree split in direction \(d+1\).

Moreover,
\[
        \pi_{d+1}(v(\widetilde{\mathfrak{m}}_{\eta,h}))=w_{d+1}^{(h)},
\]
so the map
\[
        \mathrm{children}(\widetilde{\mathfrak{m}}_\eta)\longrightarrow
        \imm_{d+1}(\pi_{d+1}(v(\widetilde{\mathfrak{m}}_\eta))),
        \qquad
        \widetilde{\mathfrak{m}}_{\eta,h}
        \longmapsto \pi_{d+1}(v(\widetilde{\mathfrak{m}}_{\eta,h}))
\]
is bijective. Hence every node at depth \(d<\ell\) is \((d+1)\)-immediate.

It remains to attach subtrees below the nodes \(\widetilde{\mathfrak{m}}_\eta\) with \(|\eta|=\ell\).  Fix such
an \(\eta\).  We first observe that \(\widetilde{\mathfrak{m}}_\eta\) has all \(W\)-coordinates in the grid.  Indeed,
coordinates \(1,\ldots,\ell\) are the chosen grid endpoints \(w_q^{(h_q)}\); coordinates
outside \(\{1,\ldots,\ell,i\}\) were already grid coordinates in the original root split; and
for the old root direction \(i\),
\[
        \pi_i(v(\widetilde{\mathfrak{m}}_\eta))
        =
        \sum_{\mathfrak{m}\in N}a_{\mathfrak{m}} \pi_i(v(\mathfrak{m}_\eta))
        =
        \sum_{\mathfrak{m}\in N}a_{\mathfrak{m}} \pi_i(v(\mathfrak{m}))
        =
        \pi_i(v)\in \cG^{(i)}.
\]
Here we used that the first \(\ell\) splits in \(T_\mathfrak{m}\) do not change the \(W_i\)-coordinate,
and that the original root split of \(T\) had direction \(i\).

Now define the children of \(\widetilde{\mathfrak{m}}_\eta\) as follows.  Split \(N\) into
\[
        N_0:=\{\mathfrak{m}\in N: I_{\mathcal G}(v(\mathfrak{m}))=\{1,\ldots,\ell\}\},
        \qquad
        N_1:=\{\mathfrak{m}\in N: I_{\mathcal G}(v(\mathfrak{m}))=\{1,\ldots,\ell,i\}\}.
\]
For \(\mathfrak{m}\in N_0\), the node \(\mathfrak{m}_\eta\) has all coordinates in the grid.  We include one
child 
\(\widetilde{\fp}_{\eta,\mathfrak{m}}\) of \(\widetilde{\mathfrak{m}}_\eta\), set
\[
        v(\widetilde{\fp}_{\eta,\mathfrak{m}})=v(\mathfrak{m}_\eta),
        \qquad
        \alpha(\widetilde{\fp}_{\eta,\mathfrak{m}})=a_{\mathfrak{m}},
\]
and below \(\widetilde{\fp}_{\eta,\mathfrak{m}}\) copy the subtree of \(T_{\mathfrak{m}}\) rooted at \(\mathfrak{m}_\eta\), changing only
the incoming weight at the copied root.

For \(\mathfrak{m}\in N_1\), the tree \(T_{\mathfrak{m}}\) is
\(((1,\ldots,\ell,i),\cG)\)-gridded immediate.  Thus the node \(m_\eta\), at depth
\(\ell\) in \(T_{\mathfrak{m}}\), is \(i\)-immediate.  Let
\[
        \mathrm{children}_{T_{\mathfrak{m}}}(\mathfrak{m}_\eta)=\{\mathfrak{m}_{\eta,f}: f\in \imm_i(\pi_i(v(\mathfrak{m}_\eta)))\},
\]
where
\[
        \pi_i(v(\mathfrak{m}_{\eta,f}))=f.
\]
Write
\[
        v(\mathfrak{m}_\eta)
        =
        \sum_{f\in \imm_i(\pi_i(v(\mathfrak{m}_\eta)))}
        \gamma_{\mathfrak{m},f}v(\mathfrak{m}_{\eta,f}),
        \qquad
        \gamma_{\mathfrak{m},f}:=\alpha_{T_\mathfrak{m}}(\mathfrak{m}_{\eta,f}).
\]
For each such \(f\), include one child \(\widetilde{\fp}_{\eta,\mathfrak{m},f}\) of \(\widetilde{\mathfrak{m}}_\eta\), set
\[
        v(\widetilde{\fp}_{\eta,\mathfrak{m},f})=v(\mathfrak{m}_{\eta,f}),
        \qquad
        \alpha(\widetilde{\fp}_{\eta,\mathfrak{m},f})=a_{\mathfrak{m}}\gamma_{\mathfrak{m},f},
\]
and below \(\widetilde{\fp}_{\eta,\mathfrak{m},f}\) copy the subtree of \(T_{\mathfrak{m}}\) rooted at \(\mathfrak{m}_{\eta,f}\), again
changing only the incoming weight at the copied root.

Finally declare
\[
        i(\widetilde{\mathfrak{m}}_\eta)=i
\]
for each \(|\eta|=\ell\).  This final split is not part of the prescribed tuple
\((1,\ldots,\ell)\); it occurs only after all prescribed levels have already been completed.

We check the witness-tree identity at \(\widetilde{\mathfrak{m}}_\eta\).  The children just defined have total
weight
\[
        \sum_{\mathfrak{m}\in N_0}a_{\mathfrak{m}}
        +
        \sum_{\mathfrak{m}\in N_1}a_{\mathfrak{m}}
        \sum_f \gamma_{\mathfrak{m},f}
        =
        \sum_{\mathfrak{m}\in N}a_{\mathfrak{m}}=1.
\]
Moreover,
\[
\begin{aligned}
        &\sum_{\mathfrak{m}\in N_0}a_{\mathfrak{m}} v(\widetilde{\fp}_{\eta,\mathfrak{m}})
        +
        \sum_{\mathfrak{m}\in N_1}\sum_f a_{\mathfrak{m}}\gamma_{\mathfrak{m},f}v(\widetilde{\fp}_{\eta,\mathfrak{m},f})        \\
        &\qquad =
        \sum_{\mathfrak{m}\in N_0}a_{\mathfrak{m}} v(\mathfrak{m}_\eta)
        +
        \sum_{\mathfrak{m}\in N_1}a_{\mathfrak{m}}
        \sum_f \gamma_{\mathfrak{m},f}v(\mathfrak{m}_{\eta,f})                             \\
        &\qquad =
        \sum_{\mathfrak{m}\in N_0}a_{\mathfrak{m}} v(\mathfrak{m}_\eta)
        +
        \sum_{\mathfrak{m}\in N_1}a_{\mathfrak{m}} v(\mathfrak{m}_\eta)                                  \\
        &\qquad =
        \sum_{\mathfrak{m}\in N}a_{\mathfrak{m}} v(\mathfrak{m}_\eta)
        =
        v(\widetilde{\mathfrak{m}}_\eta).
\end{aligned}
\]
The projection condition also holds.  All children of \(\widetilde{\mathfrak{m}}_\eta\) agree with \(\widetilde{\mathfrak{m}}_\eta\) in
every coordinate except possibly \(i\): for \(\mathfrak{m}\in N_0\) this follows from the invariant that
the nodes \(\mathfrak{m}_\eta\) agree outside \(W_i\), and for \(\mathfrak{m}\in N_1\) it follows because
\(\mathfrak{m}_\eta\to \mathfrak{m}_{\eta,f}\) is an \(i\)-split in \(T_{\mathfrak{m}}\).  Thus the final split at \(\widetilde{\mathfrak{m}}_\eta\)
is a valid witness-tree split in direction \(i\).

All children introduced at this final step have all coordinates in the grid. For
\(\mathfrak{m}\in N_0\), this is because \(\mathfrak{m}_\eta\) lies at depth \(\ell\) in a
\(((1,\ldots,\ell),\cG)\)-gridded tree.  For \(\mathfrak{m}\in N_1\), this is because
\(\mathfrak{m}_{\eta,f}\) lies immediately after the \(i\)-split at depth \(\ell\) in a
\(((1,\ldots,\ell,i),\cG)\)-gridded immediate tree.  Hence all copied descendants
also have all coordinates in the grid.

We have therefore constructed a finite witness tree \(\widetilde{T}\) for \(v\). It is
\(((1,\ldots,\ell),\mathcal G)\)-gridded: at depth \(d<\ell\) the splitting direction is
\(d+1\), and all coordinates outside \(\{d+1,\ldots,\ell\}\) are already grid-valued; at
depth \(d\geq\ell\), all coordinates are grid-valued.  It is immediate at every prescribed
level \(0,\ldots,\ell-1\), by the bijectivity checks above.  Thus \(\widetilde{T}\) is a
\(((1,\ldots,\ell),\cG)\)-gridded immediate witness tree for \(v\).

This completes Case 2.
\end{enumerate}
This completes the inductive step, and the proof of the proposition is complete. 
\end{enumerate}
\end{proof}

\begin{proof}[Proof of Proposition~\ref{prop:existence-of-grid}]
   Follows immediately from Proposition~\ref{prop:gridded-witness}. 
\end{proof}

\begin{proof}[Proof of Theorem~\ref{thm:laminar}]
By Proposition~\ref{prop:existence-of-grid}, when \(\dim W_i=1\) for all \(i\), the grid
\(G=\pi_1(S)\times\cdots\times\pi_k(S)\) is sufficiently fine for \(S\).
Applying Proposition~\ref{prop:laminar} to this grid gives the semi-algebraicity of
\(G_\Lambda^{(\infty)}(S)\). The algorithmic statement follows from the
effective parts of Theorem~\ref{thm:hemihedra}, Proposition~\ref{prop:struc}, and effective quantifier elimination over the reals.
\end{proof}

%% file: arxiv-example.tex
\section{Solving Example System}
\label{sec:example}

We will find the smallest solution of the following system over arbitrary convex sets using the algorithm in \figureref{gaussian-elim}.

    \begin{align*}
        X_1 &\geq P_1 \;\;\cplus\;\; X_1 \cmult \left(\frac12\cdot X_2 + \frac12\cdot P_3\right)\\
        X_2 &\geq P_2 \;\;\cplus\;\; X_2 \cmult \left(\frac12\cdot X_1 + \frac12\cdot P_4\right)
    \end{align*}
We are going to use 
\begin{enumerate}[(a)]
    \item rearrangement as in Lemma~\ref{lem:rearrangement},
    \item cancellation as in Lemma~\ref{lem:cancel},
    \item substitution as in Lemma~\ref{lem:substitute}.
\end{enumerate}
The equation of $X_1$ does not need to be rearranged; we can proceed with cancellation. 
After canceling $X_1$, we get the following system. 

    \begin{align*}
        X_1 &\geq P_1 \cplus P_1 \cmult \left(\frac12\cdot X_2 + \frac12\cdot P_3\right)\\
        X_2 &\geq P_2 \cplus X_2 \cmult \left(\frac12\cdot X_1 + \frac12\cdot P_4\right)
    \end{align*}
Next, we aim to substitute $X_1$ with the RHS of the first inequality in $X_2$'s inequality (step 2.c. in \figureref{gaussian-elim} with $j=1$). 
Below, we will illustrate how the substituted polynomial is obtained. 
\begin{align*}
    X_2 &\geq P_2 \cplus X_2\cmult \left(\frac12\cdot X_1 + \frac12\cdot P_4\right)
                \tag{original equation of $X_2$}\\
    &\geq P_2 \cplus X_2\cmult \left(\frac12\cdot \left[P_1 \cplus P_1 \cmult \left(\frac12\cdot X_2 + \frac12\cdot P_3\right)\right] + \frac12\cdot P_4\right)
                \tag{substituting the symbol $X_1$ with the RHS of $X_1$'s inequality}\\
                \tag{Remark: this expression is {\em not} a polynomial}\\
    &= P_2 \cplus X_2\cmult \left( \left[\frac12\cdot P_1 \cplus \frac12\cdot P_1 \cmult \left(\frac14\cdot X_2 + \frac14\cdot P_3\right)\right] + \frac12\cdot P_4\right)
                \tag{scalar multiplication distributes over $\cplus$ and $\cmult$}\\
                \tag{Remark: this expression is {\em not} a polynomial}\\
    &\sim P_2 \cplus X_2\cmult\left(\frac12\cdot P_1 + \frac12\cdot P_4\right) 
            \cplus X_2\cmult\left(\frac12\cdot P_1+\frac12\cdot P_4\right)\cmult \left(\frac14\cdot X_2 + \frac14\cdot P_3 + \frac12\cdot P_4\right)
                \tag{Minkowski sum distributes over $\cplus$ and $\cmult$}
\end{align*}        
This final expression is a polynomial.
Verify that this ``derivation'' of the substituted polynomial, capturing what we intend to achieve, matches the polynomial computed using our definition of substituted polynomials in Example~\ref{eg:1} of Subsection~\ref{subsec:sub-compute}.
After substitution, we get the system:

    \begin{align*}
        X_1 &\geq P_1 \cplus P_1 \cmult \left(\frac12\cdot X_2 + \frac12\cdot P_3\right)\\
        X_2 &\geq P_2 \cplus X_2\cmult\left(\frac12\cdot P_1 + \frac12\cdot P_4\right) 
            \cplus X_2\cmult\left(\frac12\cdot P_1+\frac12\cdot P_4\right)\cmult \left(\frac14\cdot X_2 + \frac14\cdot P_3 + \frac12\cdot P_4\right)
    \end{align*}
At this point, note that $X_1$ has been eliminated from the RHS of every inequality. 
This corresponds to completing the $j=1$ loop in \figureref{gaussian-elim}.

After that, in $j=2$ loop, we begin by rearranging the inequality for $X_2$ as follows:
\begin{align*}
    && X_2 &\geq P_2 \cplus X_2\cmult\left(\frac12\cdot P_1 \cplus \frac12\cdot P_4\right) 
            \cplus X_2\cmult\left(\frac12\cdot P_1+\frac12\cdot P_4\right)\cmult \left(\frac14\cdot X_2 + \frac14\cdot P_3 + \frac12\cdot P_4\right)\\
    \iff&& X_2 &\geq P_2 \cplus X_2\cmult\left(\frac12\cdot P_1 \cplus \frac12\cdot P_4\right) 
            \cplus X_2\cmult\left(\frac12\cdot P_1+\frac12\cdot P_4\right)\cmult \left(\frac13\cdot P_3 + \frac23\cdot P_4\right)
\end{align*}
This derivation relies on the fact that $X\geq \left(\rho\cdot X + (1-\rho)\cdot A\right)\cmult B$ if (and only if) $X\geq X\cmult A\cmult B$, for arbitrary sets $X\in\convset$, $A,B\subseteq \RR^d$ and $0<\rho<1$ (see Lemma~\ref{lem:pullout-X}). 
After that, we cancel $X_2$ from this rewritten inequality to get the following system. 

    \begin{align*}
        X_1 &\geq P_1 \cplus P_1 \cmult \left(\frac12\cdot X_2 + \frac12\cdot P_3\right)\\
        X_2 &\geq P_2 \cplus P_2\cmult\left(\frac12\cdot P_1 + \frac12\cdot P_4\right) 
            \cplus P_2\cmult\left(\frac12\cdot P_1+\frac12\cdot P_4\right)\cmult \left(\frac13\cdot P_3 + \frac23\cdot P_4\right)
    \end{align*}
Next, we aim to substitute the RHS of $X_2$'s inequality into the symbol $X_2$ in $X_1$'s inequality. 
To illustrate how the substituted polynomial is defined, we elaborate on the substitution process. 
{\footnotesize\allowdisplaybreaks
\begin{align*}
    X_1 &\geq P_1 \cplus P_1 \cmult \left(\frac12\cdot X_2 + \frac12\cdot P_3\right)
                \tag{original equation}\\
    &\geq P_1 \cplus P_1 \cmult \left(\frac12\cdot \left[P_2 \cplus P_2\cmult\left(\frac12\cdot P_1 + \frac12\cdot P_4\right) 
            \cplus P_2\cmult\left(\frac12\cdot P_1+\frac12\cdot P_4\right)\cmult \left(\frac13\cdot P_3 + \frac23\cdot P_4\right)\right] + \frac12\cdot P_3\right)
                \tag{substituting the symbol $X_2$ with the RHS of $X_2$'s inequality}\\
                \tag{Remark: this expression is {\em not} a polynomial}\\
    &= P_1 \cplus P_1 \cmult \left( \left[\frac12\cdot P_2 \cplus \frac12\cdot P_2\cmult\left(\frac14\cdot P_1 + \frac14\cdot P_4\right) 
            \cplus \frac12\cdot P_2\cmult\left(\frac14\cdot P_1+\frac14\cdot P_4\right)\cmult \left(\frac16\cdot P_3 + \frac13\cdot P_4\right)\right] + \frac12\cdot P_3\right)
                \tag{scalar multiplication distributes over $\cplus$ and $\cmult$}\\
                \tag{Remark: this expression is {\em not} a polynomial}\\
    &\sim P_1 \cplus P_1\cmult\left(\frac12\cdot P_2 + \frac12\cdot P_3\right) \cplus P_1\cmult\left(\frac12\cdot P_2 + \frac12\cdot P_3\right)\cmult\left(\frac14\cdot P_1+\frac14\cdot P_4+\frac12\cdot P_3\right)\\
            &\qquad\qquad \cplus P_1\cmult\left(\frac12\cdot P_2 + \frac12\cdot P_3\right)\cmult\left(\frac14\cdot P_1+\frac14\cdot P_4+\frac12\cdot P_3\right)\cmult\left(\frac23\cdot P_3 + \frac13\cdot P_4\right)
                \tag{Minkowski sum distributes over $\cplus$ and $\cmult$}
\end{align*}%
}
This final expression is a polynomial, and it is used on the RHS of the substituted system below.
Example~\ref{eg:2} of Subsection~\ref{subsec:sub-compute} elaborates on how this polynomial is computed using our definition of substituted polynomial. 

    \begin{align*}
        X_1 &\geq P_1 \cplus P_1\cmult\left(\frac12\cdot P_2 + \frac12\cdot P_3\right) \cplus P_1\cmult\left(\frac12\cdot P_2 + \frac12\cdot P_3\right)\cmult\left(\frac14\cdot P_1+\frac14\cdot P_4+\frac12\cdot P_3\right)\\
            &\qquad\qquad \cplus P_1\cmult\left(\frac12\cdot P_2 + \frac12\cdot P_3\right)\cmult\left(\frac14\cdot P_1+\frac14\cdot P_4+\frac12\cdot P_3\right)\cmult\left(\frac23\cdot P_3 + \frac13\cdot P_4\right)\\
        X_2 &\geq P_2 \cplus P_2\cmult\left(\frac12\cdot P_1 + \frac12\cdot P_4\right) 
            \cplus P_2\cmult\left(\frac12\cdot P_1+\frac12\cdot P_4\right)\cmult \left(\frac13\cdot P_3 + \frac23\cdot P_4\right)
    \end{align*}
This system, at the end of the $j=2$ loop, has eliminated all unknowns from the inequalities. 
As a result, the smallest convex solution is straightforward to obtain; it is just the convex hull of the RHS expressions.
So, the smallest solution is:

    \begin{align*}
        X_1 &= \mathrm{conv}\left(\; P_1 \cplus P_1\cmult\left(\frac12\cdot P_2 + \frac12\cdot P_3\right) \cplus P_1\cmult\left(\frac12\cdot P_2 + \frac12\cdot P_3\right)\cmult\left(\frac14\cdot P_1+\frac14\cdot P_4+\frac12\cdot P_3\right) \right.\\
            &\qquad\qquad \left. \cplus P_1\cmult\left(\frac12\cdot P_2 + \frac12\cdot P_3\right)\cmult\left(\frac14\cdot P_1+\frac14\cdot P_4+\frac12\cdot P_3\right)\cmult\left(\frac23\cdot P_3 + \frac13\cdot P_4\right) \;\right)\\
        X_2 &=\conv{\; P_2 \cplus P_2\cmult\left(\frac12\cdot P_1 + \frac12\cdot P_4\right) 
            \cplus P_2\cmult\left(\frac12\cdot P_1+\frac12\cdot P_4\right)\cmult \left(\frac13\cdot P_3 + \frac23\cdot P_4\right) \;}
    \end{align*}
Note that the smallest solution characterized here for $X_2$ is identical to the predicted solution $\ssol(I;\bP)_2$ in 
\eqref{eqn:example:ss:2}.
The expression for $X_1$ appears different from that in 
\eqref{eqn:example:ss:1}; however, it describes the same set (for any constant assignment $\bP$). 
The extra expression $\frac14\cdot \bP_1 + \frac14\cdot \bP_4 + \frac12\cdot \bP_3$ is redundant in the expression. 
It is a convex linear combination of the sets $\bP_1$ and $\frac23\cdot \bP_3 + \frac13\cdot \bP_4$. 
After accounting for this geometric property, it turns out to be identical to the solution $\ssol(I;\bP)_1$ 
in \eqref{eqn:example:ss:1}.

    \begin{align*}
        X_1 &= \conv{ P_1 \cplus P_1\cmult\left(\frac12\cdot P_2 + \frac12\cdot P_3\right) \cplus P_1\cmult\left(\frac12\cdot P_2 + \frac12\cdot P_3\right)\cmult\left(\frac23\cdot P_3 + \frac13\cdot P_4\right) \;}\\
        X_2 &=\conv{\; P_2 \cplus P_2\cmult\left(\frac12\cdot P_1 + \frac12\cdot P_4\right) 
            \cplus P_2\cmult\left(\frac12\cdot P_1+\frac12\cdot P_4\right)\cmult \left(\frac13\cdot P_3 + \frac23\cdot P_4\right) \;}
    \end{align*}
This optimization in representing the sets is not the focus of our current work, so we do not pursue this optimization.
Additionally, if we eliminated $X_2$ first and $X_1$ next, our expressions would have a similar redundancy in the $\ssol(I;\bP)_2$ expression. 
If $P_1, P_2, P_3, P_4$ are assigned convex sets, then the expressions within the ``$\conv{\cdot}$'' are already convex; this may not hold in general. 

In Subsection~\ref{subsec:fig-smallest-sol-eg} we will illustrate the solution for a specific assignment.

\subsection{Figure of the Smallest Solution for an Assignment}
\label{subsec:fig-smallest-sol-eg}

Suppose $P_1, P_2, P_3, P_4$ are assigned singleton sets in $\RR^2$. 
For that constant assignment, \figureref{our-sol} presents the smallest solution to our example system from Subsection~\ref{subsec:op-real}. 

\input{arxiv-fig-our-sol}

\subsection{Examples of Substitution}
\label{subsec:sub-compute}

\begin{example}
\label{eg:1}    

We will show the computation of $\varphi\substitute{X_1}{\varphi_{X_1}}$ where
    $$ \varphi = P_2 \cplus X_2\cmult \left(\frac12\cdot X_1  + \frac12\cdot P_4\right)\text{ and }\varphi_{X_1} = P_1 \cplus P_1\cmult \left(\frac12\cdot X_2 + \frac12\cdot P_3\right).$$
Using \eqref{eqn:poly-substitute-def}, we have 
    $$\varphi\substitute{X_1}{\varphi_{X_1}} = P_2\substitute{X_1}{\varphi_{X_1}} \cplus X_2\cmult \left(\frac12\cdot X_1  + \frac12\cdot P_4\right)\substitute{X_1}{\varphi_{X_1}}.$$
Let us demonstrate the computation of the two substitutions on the RHS expression above. 
\begin{align*}
    \text{Part 1.} && P_2\substitute{X_1}{\varphi_{X_1}} 
        &= P_2\substitute{X_1}{P_1} \cplus P_2\substitute{X_1}{P_1\cmult \left(\frac12\cdot X_2 + \frac12\cdot P_3\right)}
            \tag{first step in the step-wise application of \eqref{eqn:monomial-substitute-def}}\\
        &&&= P_2\substitute{X_1}{P_1} \cplus P_2\substitute{X_1}{P_1}\cmult P_2\substitute{X_1}{\frac12\cdot X_2 + \frac12\cdot P_3}
            \tag{final step in the step-wise application of \eqref{eqn:monomial-substitute-def}}\\
        &&&= P_2 \cplus P_2 \cmult P_2
            \tag{using \eqref{eqn:element-substitute-def}}\\
        &&&\sim P_2 \tag{using idempotence laws}.
\end{align*}
Idempotence laws are applied only for brevity; our proposed algorithms do not perform this optimization.

In the substitution computation below, we will need all $\elem(M)\to\mono(\varphi_{X_1})$ functions, where $M=X_2\cmult\left(\frac12\cdot X_1  + \frac12\cdot P_4\right)$.
That is, functions of the following form. 
    $$ \left\{ X_2, \frac12\cdot X_1 + \frac12\cdot P_4\right\} \to \left\{ P_1, P_1\cmult\left(\frac12\cdot X_2 + \frac12\cdot P_3\right) \right\}.$$
{\footnotesize\allowdisplaybreaks
\begin{align*}
    \text{Part 2.} && X_2\cmult&\left(\frac12\cdot X_1  + \frac12\cdot P_4\right)\substitute{X_1}{\varphi_{X_1}}\\
        &&&= X_2\substitute{X_1}{P_1} \cmult \left(\frac12\cdot X_1  + \frac12\cdot P_4\right)\substitute{X_1}{P_1}\\
        &&&\qquad \cplus X_2\substitute{X_1}{P_1} \cmult \left(\frac12\cdot X_1  + \frac12\cdot P_4\right)\substitute{X_1}{P_1\cmult\left(\frac12\cdot X_2 + \frac12\cdot P_3\right)}\\
        &&&\qquad \cplus X_2\substitute{X_1}{P_1\cmult\left(\frac12\cdot X_2 + \frac12\cdot P_3\right)} \cmult \left(\frac12\cdot X_1  + \frac12\cdot P_4\right)\substitute{X_1}{P_1}\\
        &&&\qquad \cplus X_2\substitute{X_1}{P_1\cmult\left(\frac12\cdot X_2 + \frac12\cdot P_3\right)} \cmult \left(\frac12\cdot X_1  + \frac12\cdot P_4\right)\substitute{X_1}{P_1\cmult\left(\frac12\cdot X_2 + \frac12\cdot P_3\right)}
            \tag{first step in the step-wise application of \eqref{eqn:monomial-substitute-def}}\\
        &&&= X_2\substitute{X_1}{P_1} \cmult \left(\frac12\cdot X_1  + \frac12\cdot P_4\right)\substitute{X_1}{P_1}\\
        &&&\qquad \cplus X_2\substitute{X_1}{P_1} \cmult \left(\frac12\cdot X_1  + \frac12\cdot P_4\right)\substitute{X_1}{P_1} \cmult \left(\frac12\cdot X_1  + \frac12\cdot P_4\right)\substitute{X_1}{\frac12\cdot X_2 + \frac12\cdot P_3}\\
        &&&\qquad \cplus X_2\substitute{X_1}{P_1} \cmult X_2\substitute{X_1}{\frac12\cdot X_2 + \frac12\cdot P_3} \cmult\left(\frac12\cdot X_1  + \frac12\cdot P_4\right)\substitute{X_1}{P_1}\\
        &&&\qquad \cplus X_2\substitute{X_1}{P_1} \cmult X_2\substitute{X_1}{\frac12\cdot X_2 + \frac12\cdot P_3} \cmult \left(\frac12\cdot X_1  + \frac12\cdot P_4\right)\substitute{X_1}{P_1} \\
        &&& \qquad\qquad\qquad\qquad\qquad\qquad\qquad\qquad\qquad\qquad\qquad\qquad \cmult \left(\frac12\cdot X_1  + \frac12\cdot P_4\right)\substitute{X_1}{\frac12\cdot X_2 + \frac12\cdot P_3}
            \tag{final step in the step-wise application of \eqref{eqn:monomial-substitute-def}}\\
        &&&= X_2\cmult \left(\frac12\cdot P_1 + \frac12\cdot P_4\right)\\
        &&&\qquad \cplus X_2\cmult \left(\frac12\cdot P_1+\frac12\cdot P_4\right)\cmult\left(\frac14\cdot X_2 + \frac14\cdot P_3 + \frac12\cdot P_4\right)\\
        &&&\qquad \cplus X_2\cmult X_2 \cmult \left(\frac12\cdot P_1+ \frac12\cdot P_4\right)\\
        &&&\qquad \cplus X_2\cmult X_2 \cmult \left(\frac12\cdot P_1+\frac12\cdot P_4\right)\cmult\left(\frac14\cdot X_2 + \frac14\cdot P_3 + \frac12\cdot P_4\right)
            \tag{using \eqref{eqn:element-substitute-def}}\\
        &&&\sim X_2\cmult \left(\frac12\cdot P_1 + \frac12\cdot P_4\right) \cplus X_2\cmult \left(\frac12\cdot P_1+\frac12\cdot P_4\right)\cmult\left(\frac14\cdot X_2 + \frac14\cdot P_3 + \frac12\cdot P_4\right)
            \tag{using idempotence laws}.
\end{align*}
}
We want to emphasize that every step of the derivation above is a polynomial. 
To conclude, putting these two derivations together, we have:
    $$ \varphi\substitute{X_1}{\varphi_{X_1}} \sim P_2 \cplus  X_2\cmult \left(\frac12\cdot P_1 + \frac12\cdot P_4\right) \cplus X_2\cmult \left(\frac12\cdot P_1+\frac12\cdot P_4\right)\cmult\left(\frac14\cdot X_2 + \frac14\cdot P_3 + \frac12\cdot P_4\right).$$
\end{example}

\begin{example}
\label{eg:2}
We will show the computation of $\varphi\substitute{X_2}{\varphi_{X_2}}$ where
\begin{align*}
    \varphi &= P_1 \cplus P_1\cmult\left(\frac12\cdot X_2 + \frac12\cdot P_3\right) \\ 
    \varphi_{X_2} &= P_2 \cplus P_2\cmult\left(\frac12\cdot P_1 + \frac12\cdot P_4\right) \cplus P_2\cmult\left(\frac12\cdot P_1 + \frac12\cdot P_4\right)\cmult\left(\frac13\cdot P_3+\frac23\cdot P_4\right)
\end{align*}
By \eqref{eqn:poly-substitute-def}, we have 
    $$\varphi\substitute{X_2}{\varphi_{X_2}} = P_1\substitute{X_2}{\varphi_{X_2}} \cplus P_1\cmult \left(\frac12\cdot X_2  + \frac12\cdot P_3\right)\substitute{X_2}{\varphi_{X_2}}.$$
Next, we compute the substituted polynomials (short-circuiting the trivial substitutions).
{
\begin{align*}
    \varphi\substitute{X_2}{\varphi_{X_2}} &= P_1 \cplus P_1\cmult \left(\frac12\cdot X_2  + \frac12\cdot P_3\right)\substitute{X_2}{P_2} \\
    &\qquad \cplus P_1\cmult \left(\frac12\cdot X_2  + \frac12\cdot P_3\right)\substitute{X_2}{P_2\cmult\left(\frac12\cdot P_1 + \frac12\cdot P_4\right)} \\
    &\qquad \cplus P_1\cmult \left(\frac12\cdot X_2  + \frac12\cdot P_3\right)\substitute{X_2}{P_2\cmult\left(\frac12\cdot P_1 + \frac12\cdot P_4\right)\cmult\left(\frac13\cdot P_3+\frac23\cdot P_4\right)} \\
    &= P_1 \cplus P_1\cmult \left(\frac12\cdot X_2  + \frac12\cdot P_3\right)\substitute{X_2}{P_2} \\
    &\qquad \cplus P_1\cmult \left(\frac12\cdot X_2  + \frac12\cdot P_3\right)\substitute{X_2}{P_2}\cmult \left(\frac12\cdot X_2  + \frac12\cdot P_3\right)\substitute{X_2}{\frac12\cdot P_1 + \frac12\cdot P_4} \\
    &\qquad \cplus P_1\cmult \left(\frac12\cdot X_2  + \frac12\cdot P_3\right)\substitute{X_2}{P_2} \cmult \left(\frac12\cdot X_2  + \frac12\cdot P_3\right)\substitute{X_2}{\frac12\cdot P_1 + \frac12\cdot P_4} \\
    &\qquad\qquad\qquad\qquad\qquad\qquad\qquad\qquad\qquad\ \cmult \left(\frac12\cdot X_2  + \frac12\cdot P_3\right)\substitute{X_2}{\frac13\cdot P_3+\frac23\cdot P_4}\\
    &= P_1 \cplus P_1\cmult \left(\frac12\cdot P_2  + \frac12\cdot P_3\right) \\
    &\qquad \cplus P_1\cmult \left(\frac12\cdot P_2  + \frac12\cdot P_3\right) \cmult \left(\frac14\cdot P_1 + \frac14\cdot P_4 + \frac12\cdot P_3\right) \\
    &\qquad \cplus P_1\cmult \left(\frac12\cdot P_2  + \frac12\cdot P_3\right) \cmult \left(\frac14\cdot P_1 + \frac14\cdot P_4 + \frac12\cdot P_3\right) \cmult \left(\frac23\cdot P_3 + \frac13\cdot P_4\right)
\end{align*}
}
This concludes the derivation of the substituted polynomial. 
\end{example}

\subsection{Iterated Solution Evolution for an Assignment}
\label{subsec:itr-sol-fig}

Suppose $P_1, P_2, P_3, P_4$ are assigned singleton sets in $\RR^2$. 
\figureref{itr} illustrates the evolution of the iterated solutions $\bX\p i$, for $i\in\{0,1,\dotsc\}$ introduced in Subsection~\ref{subsec:op-real}, corresponding to our example system.

\input{arxiv-fig-itr}

%% file: arxiv-fig-our-sol.tex
\begin{figure}[htp]
    \centering\footnotesize

    \pgfdeclarelayer{background}
    \pgfdeclarelayer{face}
    \pgfdeclarelayer{lines}
    \pgfdeclarelayer{vertices}
    \pgfdeclarelayer{main}
    \pgfdeclarelayer{foreground}
    \pgfsetlayers{background, face, lines, vertices, main, foreground}

    \begin{tikzpicture}[scale=1.5, >=stealth]

        \coordinate (p1) at (0,2); 
        \coordinate (p2) at (6,2);
        \coordinate (p3) at (6,0); 
        \coordinate (p4) at (0,0); 

        \coordinate (p23) at ($(p2)!0.50!(p3)$);
        \coordinate (p14) at ($(p1)!0.50!(p4)$); 

        \coordinate (p334) at (barycentric cs:p3=2,p4=1);
        \coordinate (p344) at (barycentric cs:p3=1,p4=2);

        \foreach \i/\pos in { 1/south east, 2/south west, 3/north west, 4/north east} {
            \begin{pgfonlayer}{foreground}
                \node [anchor=\pos] at (p\i) {$\bP_{\i}$}; 
            \end{pgfonlayer}
        }
        \begin{pgfonlayer}{foreground}
            \node [anchor=west] at (p23) {$\frac12\cdot \bP_2 + \frac12\cdot \bP_3$}; 
            \node [anchor=east] at (p14) {$\frac12\cdot \bP_1 + \frac12\cdot \bP_4$};
            \node [anchor=north] at (p334) {$\frac23\cdot \bP_3 + \frac13\cdot \bP_4$};
            \node [anchor=north] at (p344) {$\frac13\cdot \bP_3 + \frac23\cdot \bP_4$};
            \node [anchor=south] at (barycentric cs:p1=1,p2=1) {$\bX_1\p*$};
        \end{pgfonlayer}

        \begin{pgfonlayer}{face}
            \fill[fill=gray] (p1) -- (p23) -- (p334) -- cycle; 
        \end{pgfonlayer} 

        \begin{pgfonlayer}{lines}
            \path[draw, line width=8pt, white] (p1) -- (p23) -- (p334) -- cycle; 
            \path[draw, gray] (p1) -- (p23); 
        \end{pgfonlayer}

        \begin{pgfonlayer}{vertices}
            \foreach \i/\c in { 1/gray, 2/white, 3/white, 4/white, 23/white, 14/white, 334/white, 344/white} {
                \draw[fill=\c, preaction={draw, line width=8pt, white}] (p\i) circle (0.05);
            }
        \end{pgfonlayer}

        \coordinate (p1) at (0,-2); 
        \coordinate (p2) at (6,-2);
        \coordinate (p3) at (6,-4); 
        \coordinate (p4) at (0,-4); 

        \coordinate (p23) at ($(p2)!0.50!(p3)$);
        \coordinate (p14) at ($(p1)!0.50!(p4)$); 

        \coordinate (p334) at (barycentric cs:p3=2,p4=1);
        \coordinate (p344) at (barycentric cs:p3=1,p4=2);

        \foreach \i/\pos in { 1/south east, 2/south west, 3/north west, 4/north east} {
            \begin{pgfonlayer}{foreground}
                \node [anchor=\pos] at (p\i) {$\bP_{\i}$};
            \end{pgfonlayer}
        }
        \begin{pgfonlayer}{foreground}
            \node [anchor=west] at (p23) {$\frac12\cdot \bP_2 + \frac12\cdot \bP_3$};
            \node [anchor=east] at (p14) {$\frac12\cdot \bP_1 + \frac12\cdot \bP_4$};
            \node [anchor=north] at (p334) {$\frac23\cdot \bP_3 + \frac13\cdot \bP_4$};
            \node [anchor=north] at (p344) {$\frac13\cdot \bP_3 + \frac23\cdot \bP_4$};
            \node [anchor=south] at (barycentric cs:p1=1,p2=1) {$\bX_2\p*$};
        \end{pgfonlayer}

        \begin{pgfonlayer}{face}
            \fill[fill=gray] (p2) -- (p14) -- (p344) -- cycle; 
        \end{pgfonlayer} 

        \begin{pgfonlayer}{lines}
            \path[draw, line width=8pt, white] (p2) -- (p14) -- (p344) -- cycle; 
            \path[draw, gray] (p2) -- (p14); 
        \end{pgfonlayer}

        \begin{pgfonlayer}{vertices}
            \foreach \i/\c in { 1/white, 2/gray, 3/white, 4/white, 23/white, 14/white, 334/white, 344/white} {
                \draw[fill=\c, preaction={draw, line width=8pt, white}] (p\i) circle (0.05);
            }
        \end{pgfonlayer}

    \end{tikzpicture}
    \caption{The smallest convex solutions $\left(\bX_1\p*, \bX_2\p*\right)$ of the example system from \sectionref{system-algebra}. 
    }
    \label{fig:our-sol}
\end{figure}

%% file: arxiv-fig-itr.tex
\begin{figure}[htp]
    \centering\footnotesize

    \pgfdeclarelayer{background}
    \pgfdeclarelayer{face}
    \pgfdeclarelayer{lines}
    \pgfdeclarelayer{vertices}
    \pgfdeclarelayer{main}
    \pgfdeclarelayer{foreground}
    \pgfsetlayers{background, face, lines, vertices, main, foreground}

    \begin{tikzpicture}[scale=1.5, >=stealth]
        \coordinate (p1) at (0,2); 
        \coordinate (p2) at (6,2);
        \coordinate (p3) at (6,0); 
        \coordinate (p4) at (0,0); 

        \coordinate (p23) at ($(p2)!0.50!(p3)$);
        \coordinate (p14) at ($(p1)!0.50!(p4)$); 

        \coordinate (p334) at (barycentric cs:p3=2,p4=1);
        \coordinate (p344) at (barycentric cs:p3=1,p4=2);

        \coordinate (x) at (6,0); 
        \coordinate (y) at (0,2);

        \foreach \col/\L in {
                    red/{ {1/16}/{17/32}, {9/64}/{57/128}, {15/64}/{47/128}, {43/128}/{81/256}, {57/128}/{95/256}, {9/16}/{15/32}, {23/32}/{41/64}, {1}/{1}}, 
                    blue/{ {0}/{1}, {9/32}/{41/64}, {7/16}/{15/32}, {71/128}/{95/256}, {85/128}/{81/256}, {49/64}/{47/128}, {55/64}/{57/128}, {15/16}/{17/32}}} { 
            \edef\points{}
            \foreach \i/\j [count=\n] in \L {
                \def\this{q\n}
                \coordinate (\this) at ($\i*(x)+\j*(y)$); 
                \fill (\this) circle (0.05); 
                \xdef\points{(\this) \points}
            }
            \fill[\col, draw=black, fill opacity=0.5] plot coordinates \points -- cycle;
        }

        \foreach \i/\pos in { 1/south east, 2/south west, 3/north west, 4/north east} {
            \begin{pgfonlayer}{foreground}
                \node [anchor=\pos] at (p\i) {$P\p{\i}$};
            \end{pgfonlayer}
        }
        \begin{pgfonlayer}{foreground}
            \node [anchor=west] at (p23) {$\frac12\cdot P\p2 + \frac12\cdot P\p3$};
            \node [anchor=east] at (p14) {$\frac12\cdot P\p1 + \frac12\cdot P\p4$};
            \node [anchor=north] at (p334) {$\frac23\cdot P\p3 + \frac13\cdot P\p4$};
            \node [anchor=north] at (p344) {$\frac13\cdot P\p3 + \frac23\cdot P\p4$};
            \node [anchor=south] at ($0.2*(x)+0.75*(y)$) {$\bX_1\p5$};
            \node [anchor=south] at ($0.8*(x)+0.75*(y)$) {$\bX_2\p5$};
        \end{pgfonlayer}

        \begin{pgfonlayer}{vertices}
            \foreach \i/\c in { 1/black, 2/black, 3/white, 4/white, 23/white, 14/white, 334/white, 344/white} {
                \draw[fill=\c, preaction={draw, line width=8pt, white}] (p\i) circle (0.05);
            }
        \end{pgfonlayer}

        \coordinate (p1) at (0,-2); 
        \coordinate (p2) at (6,-2);
        \coordinate (p3) at (6,-4); 
        \coordinate (p4) at (0,-4); 

        \coordinate (p23) at ($(p2)!0.50!(p3)$);
        \coordinate (p14) at ($(p1)!0.50!(p4)$); 

        \coordinate (p334) at (barycentric cs:p3=2,p4=1);
        \coordinate (p344) at (barycentric cs:p3=1,p4=2);

        \coordinate (x) at (6,0); 
        \coordinate (y) at (0,2);

        \foreach \col/\L in {
                    red/{ {1/64}/{65/128}, {21/256}/{213/512}, {39/256}/{167/512}, {221/1024}/{525/2048}, {281/1024}/{425/2048}, {683/2048}/{729/4096}, {805/2048}/{851/4096}, {29/64}/{33/128}, {267/512}/{341/1024}, {39/64}/{57/128}, {95/128}/{161/256}, {1}/{1}}, 
                    blue/{ {0}/{1}, {33/128}/{161/256}, {25/64}/{57/128}, {245/512}/{341/1024}, {35/64}/{33/128}, {1243/2048}/{851/4096}, {1365/2048}/{729/4096}, {743/1024}/{425/2048}, {803/1024}/{525/2048}, {217/256}/{167/512}, {235/256}/{213/512}, {63/64}/{65/128}}} { 
            \edef\points{}
            \foreach \i/\j [count=\n] in \L {
                \def\this{q\n}
                \coordinate (\this) at ($\i*(x)+\j*(y)+(0,-4)$); 
                \fill (\this) circle (0.05); 
                \xdef\points{(\this) \points}
            }
            \fill[\col, draw=black, fill opacity=0.5] plot coordinates \points -- cycle;
        }

        \foreach \i/\pos in { 1/south east, 2/south west, 3/north west, 4/north east} {
            \begin{pgfonlayer}{foreground}
                \node [anchor=\pos] at (p\i) {$P\p{\i}$};
            \end{pgfonlayer}
        }
        \begin{pgfonlayer}{foreground}
            \node [anchor=west] at (p23) {$\frac12\cdot P\p2 + \frac12\cdot P\p3$};
            \node [anchor=east] at (p14) {$\frac12\cdot P\p1 + \frac12\cdot P\p4$};
            \node [anchor=north] at (p334) {$\frac23\cdot P\p3 + \frac13\cdot P\p4$};
            \node [anchor=north] at (p344) {$\frac13\cdot P\p3 + \frac23\cdot P\p4$};
            \node [anchor=south] at ($0.2*(x)+0.75*(y)+(0,-4)$) {$\bX_1\p7$};
            \node [anchor=south] at ($0.8*(x)+0.75*(y)+(0,-4)$) {$\bX_2\p7$};
        \end{pgfonlayer}

        \begin{pgfonlayer}{vertices}
            \foreach \i/\c in { 1/black, 2/black, 3/white, 4/white, 23/white, 14/white, 334/white, 344/white} {
                \draw[fill=\c, preaction={draw, line width=8pt, white}] (p\i) circle (0.05);
            }
        \end{pgfonlayer}

        \coordinate (p1) at (0,-6); 
        \coordinate (p2) at (6,-6);
        \coordinate (p3) at (6,-8); 
        \coordinate (p4) at (0,-8); 

        \coordinate (p23) at ($(p2)!0.50!(p3)$);
        \coordinate (p14) at ($(p1)!0.50!(p4)$); 

        \coordinate (p334) at (barycentric cs:p3=2,p4=1);
        \coordinate (p344) at (barycentric cs:p3=1,p4=2);

        \coordinate (x) at (6,0); 
        \coordinate (y) at (0,2);

        \foreach \col/\L in {
                    red/{ {0.00390625}/{0.501953125}, {0.0673828125}/{0.40869140625}, {0.1318359375}/{0.31591796875}, {0.185791015625}/{0.2413330078125}, {0.229736328125}/{0.1851806640625}, {0.26666259765625}/{0.144561767578125}, {0.29998779296875}/{0.116790771484375}, {0.333343505859375}/{0.1001129150390625}, {0.366729736328125}/{0.1168060302734375}, {0.400146484375}/{0.1446533203125}, {0.4373779296875}/{0.18560791015625}, {0.482421875}/{0.2431640625}, {0.54052734375}/{0.323486328125}, {0.62109375}/{0.439453125}, {0.748046875}/{0.6259765625}, {1.0}/{1.0}}, 
                    blue/{ {0.0}/{1.0}, {0.251953125}/{0.6259765625}, {0.37890625}/{0.439453125}, {0.45947265625}/{0.323486328125}, {0.517578125}/{0.2431640625}, {0.5626220703125}/{0.18560791015625}, {0.599853515625}/{0.1446533203125}, {0.633270263671875}/{0.1168060302734375}, {0.666656494140625}/{0.1001129150390625}, {0.70001220703125}/{0.116790771484375}, {0.73333740234375}/{0.144561767578125}, {0.770263671875}/{0.1851806640625}, {0.814208984375}/{0.2413330078125}, {0.8681640625}/{0.31591796875}, {0.9326171875}/{0.40869140625}, {0.99609375}/{0.501953125}}} { 
            \edef\points{}
            \foreach \i/\j [count=\n] in \L {
                \def\this{q\n}
                \coordinate (\this) at ($\i*(x)+\j*(y)+(0,-8)$); 
                \fill (\this) circle (0.05); 
                \xdef\points{(\this) \points}
            }
            \fill[\col, draw=black, fill opacity=0.5] plot coordinates \points -- cycle;
        }

        \foreach \i/\pos in { 1/south east, 2/south west, 3/north west, 4/north east} {
            \begin{pgfonlayer}{foreground}
                \node [anchor=\pos] at (p\i) {$P\p{\i}$};
            \end{pgfonlayer}
        }
        \begin{pgfonlayer}{foreground}
            \node [anchor=west] at (p23) {$\frac12\cdot P\p2 + \frac12\cdot P\p3$};
            \node [anchor=east] at (p14) {$\frac12\cdot P\p1 + \frac12\cdot P\p4$};
            \node [anchor=north] at (p334) {$\frac23\cdot P\p3 + \frac13\cdot P\p4$};
            \node [anchor=north] at (p344) {$\frac13\cdot P\p3 + \frac23\cdot P\p4$};
            \node [anchor=south] at ($0.2*(x)+0.75*(y)+(0,-8)$) {$\bX_1\p9$};
            \node [anchor=south] at ($0.8*(x)+0.75*(y)+(0,-8)$) {$\bX_2\p9$};
        \end{pgfonlayer}

        \begin{pgfonlayer}{vertices}
            \foreach \i/\c in { 1/black, 2/black, 3/white, 4/white, 23/white, 14/white, 334/white, 344/white} {
                \draw[fill=\c, preaction={draw, line width=8pt, white}] (p\i) circle (0.05);
            }
        \end{pgfonlayer}

    \end{tikzpicture}
    \caption{Illustration of the iterated convex sets $\left\{\bX\p i\right\}_{i\geq0}$ in $\RR^2$ 
    for the system in Section~\ref{sec:system-algebra}, when $i\in\{5,7,9\}$. 
    }
    \label{fig:itr}
\end{figure}

%% file: arxiv-ai.tex
\section*{Declaration of generative AI and AI-assisted technologies in the writing process}

During the preparation of this work, the authors used OpenAI's ChatGPT, accessed through a ChatGPT Pro subscription, and Anthropic's Claude, accessed through a Claude Max subscription, solely to assist
with proofreading and the supplementary verification of mathematical arguments. 
The authors independently reviewed and verified all suggestions and outputs produced by these tools, revised the manuscript as appropriate, and take full responsibility for the content and correctness of the work.